\documentclass[a4paper,11pt,leqno]{article}

\usepackage[latin1]{inputenc}
\usepackage[T1]{fontenc} 
\usepackage[francais, english]{babel}
\usepackage{verbatim}
\usepackage{calligra}
\usepackage{enumerate}
\usepackage{dsfont}
\usepackage{amsfonts}
\usepackage{amsmath}
\usepackage{amsthm}
\usepackage{amssymb}
\usepackage{fancyhdr}
\usepackage{mathtools}
\usepackage{mathrsfs}
\usepackage{color}
\usepackage{graphicx}
\usepackage{bm}

\usepackage{fontenc}
\usepackage{tikz}
\usepackage{bbold}
\usepackage[retainorgcmds]{IEEEtrantools}

\usepackage{graphicx}		
\usepackage[hypcap=false]{caption}

\DeclareMathOperator*{\tend}{\longrightarrow}

\DeclareMathOperator*{\D}{\rm{div}}

\DeclareMathOperator*{\wtend}{\rightharpoonup}

\theoremstyle{definition}
\newtheorem{defi}{Definition}[section]

\newtheorem{rmk}[defi]{Remark}

\theoremstyle{plane}
\newtheorem{thm}[defi]{Theorem}
\newtheorem{prop}[defi]{Proposition}
\newtheorem{cor}[defi]{Corollary}
\newtheorem{lemma}[defi]{Lemma}

\newcommand{\mc}{\mathcal}

\newcommand{\mds}{\mathds}
\newcommand{\veps}{\varepsilon}
\newcommand{\eps}{\epsilon}
\newcommand{\what}{\widehat}

\newcommand{\vphi}{\varphi}
\newcommand{\oline}{\overline}

\newcommand{\ra}{\rightarrow}

\newcommand{\g}{\gamma}

\newcommand{\s}{\sigma}

\newcommand{\de}{\delta}

\newcommand{\R}{\mathbb{R}}

\newcommand{\N}{\mathbb{N}}
\newcommand{\Z}{\mathbb{Z}}
\newcommand{\T}{\mathbb{T}}

\renewcommand{\div}{{\rm div}\,}

\newcommand{\curl}{{\rm curl}\,}

\newcommand{\Supp}{{\rm Supp}\,}

\newcommand{\dx}{ \, {\rm d} x}
\newcommand{\dt}{ \, {\rm d} t}

\allowdisplaybreaks

\def\d{\partial}
\def\div{{\rm div}\,}

\begin{document}

\newcommand{\dimitri}[1]{\textcolor{red}{[***DC: #1 ***]}}
\newcommand{\fra}[1]{\textcolor{blue}{[***FF: #1 ***]}}

\title{\textsc{\Large{\textbf{On the fast rotation asymptotics of a \\ non-homogeneous incompressible MHD system}}}}

\author{\normalsize\textsl{Dimitri Cobb}$\,^1\qquad$ and $\qquad$
\textsl{Francesco Fanelli}$\,^{2}$ \vspace{.5cm} \\
\footnotesize{$\,^{1,} \,^2\;$ \textsc{Universit\'e de Lyon, Universit\'e Claude Bernard Lyon 1}} \vspace{.1cm} \\
{\footnotesize \it Institut Camille Jordan -- UMR 5208}  \vspace{.1cm}\\
{\footnotesize 43 blvd. du 11 novembre 1918, F-69622 Villeurbanne cedex, FRANCE} \vspace{0.2cm} \vspace{.2cm} \\
\footnotesize{$\,^{1}\;$\ttfamily{cobb@math.univ-lyon1.fr}}\\
\footnotesize{$\,^{2}\;$\ttfamily{fanelli@math.univ-lyon1.fr}}
\vspace{.2cm}
}

\date\today

\maketitle

\subsubsection*{Abstract}
{\footnotesize This paper is devoted to the analysis of a singular perturbation problem for a $2$-D incompressible MHD system with density variations
and Coriolis force, in the limit of small Rossby numbers. Two regimes are considered. The first one is the quasi-homogeneous regime, where the densities are small perturbations
around a constant state. The limit dynamics is identified as an incompressible homogeneous MHD system, coupled with an additional transport equation for the limit of the density variations.
The second case is the fully non-homogeneous regime, where the densities vary around a general non-constant profile. In this case, in the limit, the equation for the magnetic field combines
with an underdetermined linear equation, which links the limit density variation function with the limit velocity field. The proof is based on a compensated compactness argument, which enables
us to consider general ill-prepared initial data. An application of Di Perna-Lions theory for transport equations allows to treat the case of density-dependent viscosity and resistivity coefficients.
}

\paragraph*{2010 Mathematics Subject Classification:}{\small 35Q35
(primary);
35B25,
76U05,
35B40,
76W05
(secondary).}

\paragraph*{Keywords: }{\small incompressible MHD; density variations; variable viscosity and resistivity; Coriolis force; low Rossby number; singular perturbation problem.}

\section{Introduction} \label{s:intro}

In this article, we study the $\epsilon \rightarrow 0^+$ asymptotics of the following non-homogeneous incompressible MHD system with Coriolis force:
\begin{equation}\label{intro_eq:MHD}
\begin{cases}
\d_t\rho\,+\,\div\big(\rho\,u\big)\,=\,0 \\[1ex]
\partial_t\big(\rho u\big)\,+\,\D\big(\rho\, u \otimes u\big)\,+\,\dfrac{1}{\epsilon}\, \nabla \pi\,+\,\dfrac{1}{\epsilon}\, \rho\, u^\perp\,=\,\D\big(\nu(\rho)\, \nabla u\big)\,+\,
\D\big(b \otimes b\big)\,-\,\dfrac{1}{2}\nabla |b|^2 \\[1ex]
\partial_t b\, +\, \D\big(u \otimes b\big) \,-\,\div\big(b \otimes u\big)\, =\,  \nabla^\perp \big( \mu(\rho)\, \curl (b) \big) \\[1ex]
\div u\,=\,\div b\, =\, 0\,.
\end{cases}
\end{equation}
These equations are set in a two dimensional domain $\Omega$, which is either the plane $\mathbb{R}^2$ or the torus $\mathbb{T}^2$. The vector fields $u$ and $b$ are the
velocity and the magnetic fields, the scalar fields $\rho\geq0$ and $\pi$ represent the density and the pressure fields and $\nu$ and $\mu$ are two
functions defined on $\mathbb{R}_+$. The notation $u^\perp$ refers the rotation of angle $\pi/2$ of the vector $u$: in other words, if $u = (u_1, u_2)$, then $u^\perp = (-u_2, u_1)$. 
Analogously, we have set $\nabla^\perp\,=\,(-\d_2,\d_1)$. We have also defined $\curl u\,=\,\d_1u_2-\d_2u_1$ to be the $\curl$ of the $2$-D vector $u$.
As we will see, since the first equation is written up to a gradient field, the term $\nabla |b|^2/2$ does not appear in the weak form of the equations.

The main goal of this article is to show that solutions $(\rho_\epsilon, u_\epsilon, b_\epsilon)_{\epsilon > 0}$ to system \eqref{intro_eq:MHD} converge (in some way) to some
functions $(\rho, u, b)$, and to describe the limit dynamics by proving that the limit $(\rho, u, b)$ solves an evolution PDE system.

\subsection{General physical remarks}

Magnetohydrodynamic models are used whenever describing a fluid which is subject to the magnetic field it generates through its own motion. Examples of such fluids range from the industrial scale,
with plasma confinement in fusion research or some types of electrolytics, to the geophysical or astrophysical scale, with atmospheric plasmas, planetary mantle convection or the solar interior.
Their mathematical study thus combines the Navier-Stokes and the Maxwell equations.

We focus on fluids on which the Coriolis force $\epsilon^{-1}\rho u^\perp$ has a major influence compared to the kinematics of the said fluid, such as large-scale fluids evolving on a celestial
body. 
The importance of this effect is measured by the Rossby number of the fluid $Ro = \epsilon$, the condition $\epsilon \ll 1$
defining the regime of large-scale planetary or stellar fluid dynamics. At the mathematical level, in the limit $\epsilon \rightarrow 0^+$, the Coriolis force can only be balanced by the pressure
term, which is reflected by the $\epsilon^{-1}$ factor in front of $\nabla \pi$.

\subsubsection{Derivation of the equations}

Let us give a few additional details concerning the derivation of the MHD equations\footnote{In order to simplify notations, we assume the electrical permittivity and the magnetic
permeability to be of unit value, which can be done by working in an appropriate set of units.}. We conduct these computations in the physically relevant three-dimensional setting
$\Omega = \mathbb{R}^3$.

First of all, the Laplace force exerted on the fluid is $f_L = j \times b$, where $j$ is the electric current density. We further assume the fluid viscosity $\nu(\rho)$ to depend on the density
$\rho$, the precise nature of the function $\nu$ depending on the exact composition of the fluid (alternate models include anisotropic scaling in the viscosity to take into
account the joint effects of turbulence and the asymmetry induced by the rotation, see \cite{CDGG}, \cite{DDG} for more on this topic). Writing Newton's law for a density-dependent fluid, we get
\begin{equation*} \label{eq:Nwt}
\rho \partial_t u + \rho (u \cdot \nabla)u + \frac{1}{\eps}\nabla \pi + \frac{1}{\epsilon} \mathfrak{C}[\rho,u] - \D \big(\nu (\rho) \nabla u\big) = j \times b\,,
\end{equation*}
where $\mathfrak{C}[\rho,u] = e^3 \times \rho u = (-\rho u_2, \rho u_1, 0)$ is the (three-dimensional) Coriolis force. The rotation axis is taken constant, parallel to the vertical
unit vector $e^3 = (0, 0, 1)$. We will comment more on this in the next paragraph.

The fluid is assumed to be \emph{non-relativistic}, with negligible velocities $|u| \ll c$ when compared to the speed of light. This justifies the use of an electrostatic
approximation in the Maxwell equations, which are simplified by omitting the time derivative of the electric field. Obviously, this is not a wild assumption since we intend to work on planetary
or stellar fluids subject to the body's rotation (see also \cite{DDG}, \cite{N}). Thus, Amp\`ere's circuital law reads
\begin{equation*}
\curl (b) \,=\, j + \partial_t e\, \approx\, j\,,
\end{equation*}
where $e$ is the electric field.
The electrical resistivity is described by Ohm's law, which links the electrical current $j$ to the electrical field $e$ and the other physical quantities:
$$j = \sigma(\rho)\,\big( e + u \times b \big)\,.$$
Note that we have assumed that the conductivity $\sigma(\rho)$ depends on the density, the precise nature of the function $\sigma$ depending on the exact composition of the fluid.
Combining the above relation with the Maxwell-Faraday equation $\partial_t b\, =\, - \curl(e)$ gives a relation describing the evolution of the magnetic field:
$$ \partial_t b \,=\, - \curl \left( \frac{1}{\sigma(\rho)}\,\curl(b)\,-\, u \times b \right)\,.$$
By noting $\mu(\rho) = 1/\sigma(\rho)$ the electrical resistivity, we thus get
\begin{equation*}\label{eq:Mxw}
\partial_t b\,+\, (u \cdot \nabla) b \,-\, (b \cdot \nabla)u\, =\, -\,\curl \big( \mu(\rho) \curl (b) \big)\,.
\end{equation*}
Next, Gauss's law for magnetism, which rules out the possibility of magnetic monopoles, gives the divergence-free condition $\D(b) = 0$.

Finally, we assume the fluid to be incompressible, so that $\D(u) = 0$. Therefore, the mass conservation equation reads
\begin{equation*}
\partial_t \rho + \D(\rho u)\, =\, \partial_t \rho + u\cdot \nabla \rho \, = \,0\,.
\end{equation*}

Putting everything together, and using the conservative form of the equations, we obtain the following three-dimensionnal MHD system,
\begin{equation}\label{eq:MHD3d}
\begin{cases}
\partial_t \rho + \D (\rho u)\, =\, 0 \\[1ex]
\partial_t(\rho u)+\D (\rho u \otimes u)+\nabla\pi+\dfrac{1}{\epsilon}\,\mathfrak{C}[\rho,u]\,=\,\D\big(\nu(\rho)\,\nabla u\big)+\D(b \otimes b)-\dfrac{1}{2}\nabla\left( |b|^2 \right) \\[1ex]
\partial_t b + \D (u \otimes b - b \otimes u)\, =\, - \curl \big( \mu(\rho) \curl (b) \big)\\[1ex]
\D(b) = \D(u) = 0\,,
\end{cases}
\end{equation}
of which system \eqref{intro_eq:MHD} is the $2$-D equivalent.



\subsubsection{Physical relevance of the system}

This paragraph is devoted to a few critical remarks concerning system \eqref{intro_eq:MHD}, and the physical setting that led to its derivation.

First of all, the model neglects any effect due to temperature variations, which is a debatable simplification, even in the case of non-conducting fluids.
For instance, ocean water density is an intricate function of the pressure, salinity and temperature, and the temperature of air masses plays a major role in weather evolution.
In those cases, dependence on the pressure is often neglected, and both temperature and salinity are assumed to evolve through a diffusion process (see \cite{Cr}, Chapter 3).
For conducting fluids, which are generally heated magma or plasmas, the temperature is expected to play an even greater role. However, the equations, as they are, already provide
interesting challenges and are widely used by physicists for practical purposes.

Secondly, we spend a few words concerning the two-dimensional setting. Our main motivation for restricting to $2$-D domains is purely technical: in our analysis, we face similar difficulties as in
\cite{FG}, devoted to the fast rotation asymptotics for density-dependent incompressible Navier-Stokes equations in dimension $d=2$.
In particular, the fast rotation limit for incompressible non-homogeneous fluids in $3$-D is a widely open problem (see more details here below).
However, let us notice that one of the common features of highly rotating fluids is to be, in a first approximation, planar:
the fluid is devoid of vertical motion and the particles move in columns. This property is known as the Taylor-Proudman theorem (see \cite{Cr}, \cite{Ped} for useful insight).
Therefore, the $2$-D setting is in itself a relevant approximation for geophysical fluids.

At this point, note that equations \eqref{intro_eq:MHD} \textit{per se} do not describe a conducting fluid confined to a quasi-planar domain. If that were the case, the magnetic field
would circulate \textit{around} the current lines, hence being orthogonal to the plane of the fluid, assuming the form $b = b_3(t, x) e^3$ for some scalar function $b_3$. Our problem,
which involves a $2$-D magnetic field $b = (b_1, b_2)$ is a projection of the full three-dimensional MHD system \eqref{eq:MHD3d}. Taking this step away from the physical problem brings us closer
to the actual form of the physically relevant $3$-D problem, and we hope it will provide a step towards its understanding.

Next, we remark that we have taken a quite simple form for the Coriolis force: namely, $\mathfrak{C}[\rho,u] = \rho u^\perp$. This means that the rotation axis is constant and normal to the plane
where the fluid moves. Of course, more complicated choices are possible. However, on the one hand this form for the rotation term is physically consistent with a fluid evolving at mid-latitudes,
in a region small enough compared to the radius of the planetary or
stellar body. On the other hand, this choice is quite common in mathematical studies, and the obtained model is already able to explain several physical phenomena.

Finally, we point out that the incompressibility assumption $\D (u) = 0$ is a valid approximation for flows in the ocean and in the atmosphere, and we will assume it.
On the other hand, by our choice of considering domains with a very simple geometry, we completely avoid boundary effects.

\subsection{Previous mathematical results on fast rotating fluids}

The mathematical study of rotating fluids is by no means new in the mathematical litterature. It has started in the 1990s with the pioneering works \cite{B-M-N_1996}-\cite{B-M-N_1997}-\cite{B-M-N_1999}
of Babin, Mahalov and Nikolaenko, and has since been deeply investigated, above all for models of homogeneous incompressible fluids. We refer to book \cite{CDGG} for a complete treatement
of the incompressible Navier-Stokes equations with Coriolis force, and for further references on this subject.

The study of fast rotation asymptotics for non-homogeneous fluids has a much more recent history. However, efforts have mainly been focused on compressible fluid models: see e.g.
\cite{FGN}, \cite{FGDN}, \cite{F-N_CPDE}, \cite{F-Lu-N}, \cite{F_MA}, \cite{K-M-N}. We refer to
\cite{F} for additional details and further references, as well as for recent developments.
On the contrary, not so many results are available for density-dependent incompressible fluids. To the best of our knowledge, the only work in this direction is \cite{FG}, treating the case
of the non-homogeneous Navier-Stokes equations in two-dimensional domains. The reason for such a gap between the compressible and the (non-homogeneous) incompressible cases is that the coupling
between the mass and the momentum equation is weaker in the latter situation than in the former one. As a consequence, less information is available on the limit points of the sequences of solutions,
and taking the limit in the equations becomes a harder task. This explains also the lack of results for $3$-D incompressible flows with variable densities.

The case of rotating MHD equations has also recieved some attention in the past years. Once again, most of the available results concern the case of homogeneous flows: for instance, we mention papers
\cite{DDG} and \cite{R}, concerning the stability of boundary layers in homogenous rotating MHD, and \cite{N}, about the stabilising effect the rotation has on solution lifespan.
See also references therein, as well as Chapter 10 of \cite{CDGG}, for further references.
On the density-dependent side, fast rotating asymptotics has recently been conducted in \cite{KLS} for compressible flows, in two space dimensions. We point out that the approach of \cite{KLS} is based
on relative entropy estimates; if on the one hand this method enables to consider also a vanishing viscosity and resistivity regime, on the other hand it requires to assume well-prepared initial data.

\subsection{Overview of the main results of the paper}
Our main motivation here is to extend the results of \cite{FG} to the case of the MHD equations. Therefore, we choose to work with incompressible density-dependent
fluids, see system \eqref{intro_eq:MHD}.
More precisely, we study the fast rotation asymptotics in two different regimes: the \emph{quasi-homogeneneous} regime (meaning that the initial densities are small variations of a constant state)
and the \emph{fully non-homogeneous} regime (when the initial densities are perturbations of a fixed \emph{non-constant} profile, in the sense of relation \eqref{NDCP} below).
In the former case, the limit dynamics is identified as a homogeneous incompressible MHD system, coupled (\textsl{via} a lower order term)
with a pure transport equation for the limit density variation function. In the latter case, we show convergence to an underdetermined equation, expressed in terms of the vorticity of the limit
velocity field and the limit density oscillation function. The fact that the limit system is underdetermined can be viewed as an expression of the weaker coupling we mentioned above, between
the mass and momentum equations.
In order to prove our results, we will resort to the techniques of \cite{FG}, based on a compensated compactness argument, which allows us to consider general \emph{ill-prepared}
initial data. 
Roughly speaking, compensated compactness consists in exploiting the structure of the equations (wirtten in the form of a wave system governing oscillations, which propagate
in the form of Poincar\'e-Rossby waves), in order to find special algebraic cancellations and relations which allow to pass to the limit in the non-linear terms, even in absence of strong convergence.
That technique goes back to the pioneering work \cite{L-M} by P.-L. Lions and Masmoudi,
where the authors dealt with the incompressible limit of the compressible Navier-Stokes equations; it was later adapted by Gallagher and Saint-Raymond in \cite{GSR}
to the context of fast rotating fluids, and then broadly exploited in similar studies (see e.g. \cite{FGDN}, \cite{F_JMFM} and \cite{F}).

We remark that, in the fully non-homogeneous regime, the limit equation (combining the mass and momentum equations of the primitive system) is \emph{linear} in the unknowns.
This is a remarkable property, which is however now quite well-understood (see e.g. \cite{GSR}, \cite{FGDN}): let us give an insight of it.
We will be able to show that, if the initial densities are small perturbations around a non-constant state (say) $\rho_0$, then, at any later time, the solutions stay close
(in a suitable topology, but quantitatively, in powers of $\eps$) to the same state $\rho_0$. Notice that this property
is not obvious at all, as the densities satisfy a pure transport equation by the velocity fields. Anyhow, as a consequence the limit density profile is exactly the initial
reference state $\rho_0$. Roughly speaking, this fact restricts much more the limit motion than in the case when the target density is constant, since the kernel of the singular perturbation
operator is smaller. The additional constraint implies that the average process (convergence in the weak formulation of the equations) tends to kill the convective term
in the limit $\eps\ra0^+$.

As a last comment, let us point out that system \eqref{intro_eq:MHD} differs from the Navier-Stokes-Coriolis equations of \cite{FG} in a crucial way: namely, the density-dependent viscosity
and resistivity terms, respectively $\D \big( \nu(\rho) \nabla u \big)$ and $\nabla^\perp \big( \mu(\rho)\, \curl (b) \big)$, introduce new difficulties in the analysis.
Indeed, the methods of \cite{FG}, which rely on compactness of the densities in spaces of negative index of regularity, are insufficient to take the limit in those non-linear terms:
overcoming this obstacle requires almost everywhere convergence of the densities. Now, the sought almost everywhere convergence is implied by strong convergence in suitable Lebesgue spaces,
which we achieve by using well-posedness results on linear transport equations proved by Di Perna and P.-L. Lions \cite{DL}.
Besides, we point out that strong convergence of the densities has the additional advantage of providing simpler proofs: where the analysis of \cite{FG} relies on paradifferential calculus to obtain
convergence of some quadratic terms, we can often replace it by plain Hölder inequalities.

\medbreak

We conclude this introduction by giving a short overview of the paper.

In the next section we fix our assumptions and state our main results, both for the quasi-homogeneous and the fully non-homogeneous regimes.
Section \ref{s:ub} is devoted to the derivation of uniform bounds for the sequence of weak solutions, which enable us to infer first convergence properties and to identify weak-limit points.
There, we will also derive constraints those limit points have to satisfy, and establish strong convergence of the density functions.
In Section \ref{s:convergence} we complete the proof of the convergence; the main part of the analysis will be devoted to passing to the limit in the convective term.
Section \ref{s:limit-system} focuses on the well-posedness of the limit system obtained in the quasi-homogeneous regime. An appendix about Littlewood-Paley theory and
paradifferential calculus will end the manuscript.

\subsubsection*{Notation and conventions}

Before starting, let us introduce some useful notation we use throughout this text.

The space domain will be denoted by $\Omega\subset\R^d$: throughout the text, we will always work in the case $d=2$.
All derivatives are (weak) derivatives, and the symbols $\nabla$, $\D$ and $\Delta$ are, unless specified otherwise, relative to the space variables.
Given a subset $U\subset\Omega$ or $U\subset\R_+\times\Omega$,
we note $\mathcal{D}(U)$ the space of compactly supported $C^\infty$ functions on $U$.
If $f$ is a tempered distribution, we note $\mathcal{F}[f] = \what{f}$ the Fourier transform of $f$ with respect to the space variables.

For $1 \leq p \leq +\infty$, we will note $L^p(\Omega) = L^p$ when there is no ambiguity regarding the domain of definition of the functions. Likewise, we omit the dependency on $\Omega$ in functional spaces when no mistake can be made.
If $X$ is a Fréchet space of functions, we note $L^p(X) = L^p(\mathbb{R}_+ ; X)$. For any finite $T > 0$, we note $L^p_T(X) = L^p([0, T] ; X)$ and $L^p_T = L^p[0, T]$.

Let $\big(f_\epsilon\big)_{\epsilon > 0}$ be a sequence of functions in a normed space $X$. If this sequence is bounded in $X$,  we use the notation $\big(f_\epsilon\big)_{\epsilon > 0} \subset X$.
If $X$ is a topological linear space, whose (topological) dual is $X'$, we note $\langle \, \cdot \, | \, \cdot \, \rangle_{X' \times X}$ the duality brackets.

Any constant will be generically noted $C$, and, whenever deemed useful, we will specify the dependencies by writing $C = C(a_1, a_2, a_3, ...)$.
In all the text, $M_p(t) \in L^p(\R_+)$ will be a generic globally $L^p$ function; on the other hand, we will use the notation $N_p(t)$ to denote a generic function in $L^p_{\rm loc}(\R_+)$.

\subsection*{Acknowledgements}

{\small
The work of the second author has been partially supported by the LABEX MILYON (ANR-10-LABX-0070) of Universit\'e de Lyon, within the program ``Investissement d'Avenir''
(ANR-11-IDEX-0007),  and by the projects BORDS (ANR-16-CE40-0027-01) and SingFlows (ANR-18-CE40-0027), all operated by the French National Research Agency (ANR).

The second author is deeply grateful to M. Paicu for interesting discussions on the MHD system.
}

\section{Main assumptions and results} \label{s:results}

Let us fix the initial domain $\Omega$ to be either the whole space $\R^2$ or the torus $\T^2$. In $\R_+\times\Omega$, we consider the following \emph{non-homogeneneous
incompressible MHD system}:
\begin{equation}\label{MHD1}
\begin{cases}
\d_t\rho\,+\,\div\big(\rho\,u\big)\,=\,0 \\[1ex]
\partial_t\big(\rho u\big)\,+\,\D\big(\rho\, u \otimes u\big)\,+\,\dfrac{1}{\epsilon}\, \nabla \pi\,+\,\dfrac{1}{\epsilon}\, \rho\, u^\perp\,=\,\D\big(\nu(\rho)\, \nabla u\big)\,+\,
\D\big(b \otimes b\big)\,-\,\dfrac{1}{2}\nabla |b|^2 \\[1ex]
\partial_t b\, +\, \D\big(u \otimes b\big) \,-\,\div\big(b \otimes u\big)\, =\,  \nabla^\perp \big( \mu(\rho)\, \curl (b) \big) \\[1ex]
\div u\,=\,\div b\, =\, 0\,.
\end{cases}
\end{equation}

In the previous system, the viscosity coefficient $\nu$ and the resistivity coefficient $\mu$ are assumed to be continuous and non-degenerate: more precisely, they satisfy 
$$
\nu\,,\;\mu\,\in\,C^0(\R_+)\,,\quad\qquad\mbox{ with, }\;\forall\;\rho\geq0\,,\qquad\quad \nu(\rho)\geq\nu_*>0\quad\mbox{ and }\quad\mu(\rho)\geq\mu_*>0\,,
$$
for some positive real numbers $\nu_*$ and $\mu_*$.

Our main goal here is to perform the limit for $\eps\ra0^+$ in equations \eqref{MHD1} for general \emph{ill-prepared} initial data. Let us then specify the assumptions on
the initial density, velocity field and magnetic field.

\subsection{Initial data}\label{SecID}

We supplement system \eqref{MHD1} with general ill-prepared initial data. Let us be more precise, and start by considering the density function: for any $0<\eps\leq1$, we take
\begin{equation*}
\rho_{0, \epsilon}\, =\, \rho_0\, +\, \epsilon\, r_{0, \epsilon}\,, \qquad\qquad \text{with} \quad\qquad \rho_0\,\in\, C^2_b(\Omega)\quad \text{ and }\quad
\big(r_{0, \epsilon}\big)_{\epsilon>0}\,\subset\, \big(L^2 \cap L^\infty\big)(\Omega)\,. 
\end{equation*}
Here above, we have denoted $ C^2_b\,:=\, C^2\cap W^{2,\infty}$. We also assume that there is a constant $\rho^* > 0$ such that,
for any $\eps>0$, one has
\begin{equation*}
0 \leq \rho_0 \leq \rho^*\qquad \qquad \text{ and }\qquad\qquad 0 \leq \rho_{0, \epsilon} \leq 2\rho^*\,.
\end{equation*}
In the case $\Omega = \mathbb{R}^2$, we require the initial densities $\rho_{0, \epsilon}$ to fulfill an extra integrability assumption. Namely,
we suppose that one of the two following (non-equivalent) conditions is satisfied: either
\begin{align}
\exists\; \delta > 0\quad\big| \qquad \left( \frac{1}{\rho_{0, \epsilon}}\,\mds{1}_{\left\{\rho_{0, \epsilon} < \delta\right\}} \right)_{\!\!\epsilon > 0}\, \subset\, L^1(\Omega)\,, \label{Extra1}
\end{align}
where the symbol $\mds{1}_A$ stands for the characteristic function of a set $A\subset\Omega$, or
\begin{align}
\exists\; p_0 \in\,]1,+\infty[\,,\quad\exists\;\oline{\rho}>0\quad \big|\qquad\left(\big(\oline{\rho}\,-\,\rho_{0, \epsilon}\big)^+ \right)_{\!\!\epsilon > 0}\,\subset\,L^{p_0}(\Omega)\,. \label{Extra2}
\end{align}
These two conditions allow for a low frequency control on the fluid velocity, uniformly with respect to $\eps$. We will comment more about them at the end of the next subsection,
and refer to Chapter 2 of \cite{L} (see conditions (2.8), (2.9) and (2.10) therein) for more details.

For the velocity field, in order to avoid the trouble of defining the speed of the fluid in a vacuum zone $\{ \rho = 0 \}$, we work instead on the momentum $m = \rho u$.
For any $\eps>0$, we take an initial momentum $m_{0, \epsilon}$ such that
\begin{equation*}
\big(m_{0, \epsilon}\big)_{\epsilon > 0}\,\subset\, L^2(\Omega)\,,\quad \qquad \left( \frac{|m_{0, \epsilon}|^2}{\rho_{0,\epsilon}} \right)_{\epsilon > 0}\, \subset\, L^1(\Omega)\,,
\end{equation*}
where we agree that $m_{0, \epsilon} = 0$ and $|m_{0, \epsilon}|^2\,/\,\rho_{0, \epsilon} = 0$ wherever $\rho_{0, \epsilon} = 0$.

Finally, let us now consider the magnetic field: we choose initial data
$$
\big(b_{0, \epsilon}\big)_{\epsilon > 0}\, \subset\, L^2(\Omega)\,,\qquad\qquad\mbox{ with, }\,\forall\,\eps>0\,,\qquad \div b_{0,\eps}\,=\,0\,.
$$

\medbreak
Because of the previous uniform bounds, we deduce that, up to an extraction, one has the weak convergence properties
\begin{equation}\label{EQcv}
m_{0, \epsilon} \wtend m_0 \quad \text{ in }\; L^2(\Omega)\,, \qquad r_{0, \epsilon} \wtend^* r_0 \quad \text{ in }\; \big(L^2 \cap L^\infty\big)(\Omega)\,,\qquad
 b_{0, \epsilon} \wtend b_0 \quad \text{ in }\; L^2(\Omega)\,,
\end{equation}
for suitable functions $m_0$, $r_0$ and $b_0$ belonging to the respective functional spaces. Notice that we obviously have the strong convergence
property $\rho_{0, \epsilon}-\rho_0 \tend 0 \text{ in } L^2\cap L^\infty $ for $\eps\ra0^+$.

\subsection{Finite energy weak solutions} \label{ss:weak}

For smooth solutions of \eqref{MHD1} related to the initial data $(\rho_0, m_0, b_0)$, we can multiply the momentum equation by $u$ and the magnetic field equation by $b$ and
integrate both equations. We have, after integration by parts,
\begin{align*}
\frac{1}{2}\, \frac{\rm d}{{\rm d} t}\, \int_\Omega \rho |u|^2\, \text{d} x\, +\,  \int_\Omega \nu(\rho)\, \left|\nabla u\right|^2\, \text{d}x\, &=\, \int_\Omega (b\cdot \nabla) b \cdot u \,\dx \\
\frac{1}{2}\, \frac{\rm d}{{\rm d} t}\, \int_\Omega |b|^2\, \text{d} x\, +\,  \int_\Omega \mu(\rho)\, \left|\curl (b) \right|^2\, \text{d}x\, &=\, \int_\Omega (b\cdot \nabla) u \cdot b\, \dx\,.
\end{align*}
One more integration by parts show that the right-hand side of both equations are opposite. So, by summing the equations and integrating over $t \in [0, T]$ (for any fixed $T>0$)
and using the non-degeneracy hypothesis on the viscosity and resistivity coefficients, we get
\begin{equation}\label{EN}
\int_\Omega \bigg( \rho(T) |u(T)|^2 + |b(T)|^2 \bigg) \dx +  \int_0^T \int_\Omega \bigg( \nu_* |\nabla u|^2 + c \mu_* |\nabla b|^2 \bigg) \dx \text{d}t \leq \int_\Omega \left( \frac{|m_0|^2}{\rho_0} + |b_0|^2 \right) \dx.
\end{equation}
In the above inequality, we have used the fact that $b$ is divergence free, which implies that for almost all times $t \geq 0$ the norms
$\| \curl b(t) \|_{L^2}$ and $\|\nabla  b(t) \|_{L^2}$ are equivalent\footnote{In fact, in dimension $2$ and for the $L^2$ norm, we have the exact equality: namely,
$\| \curl b \|_{L^2}=\| \nabla b \|_{L^2}$. Indeed, since $\xi\cdot\what b(\xi)=0$, by Cauchy-Schwarz one has that $\left|\xi^\perp\cdot\what b(\xi)\right|=\left|\xi\right|\,\left|\what b(\xi)\right|$.}.

On the other hand, $\rho$ is simply transported by the divergence-free velocity field $u$. Hence, for all $t\geq0$, one formally has
\begin{equation} \label{est:rho_L^p}
\forall\;p\,\in\,[1,+\infty]\,,\qquad\qquad
 \left\|\rho(t)\right\|_{L^p}\,=\,\left\|\rho_0\right\|_{L^p}\,.
\end{equation}

The previous inequalities give us grounds to define the notion of finite energy weak solution.

\begin{defi}\label{DEFSol}
Let $T > 0$ and  let $\big(\rho_0, m_0, b_0\big)$ be initial data fulfilling the assumptions described in Section \ref{SecID} above. We say that $\big(\rho, u, b\big)$ is a
\emph{finite energy weak solution} to system \eqref{MHD1} in $[0, T] \times \Omega$ related to the previous initial data if the following conditions are verified:
\begin{enumerate}[(i)]
\item $\rho \in L^\infty\big([0, T] \times \Omega\big)$ and $\rho \in C^0 \big([0,T];L^q_{\rm loc}(\Omega)\big)$ for all $1 \leq q < +\infty$;
\item $\rho |u|^2 \in L^\infty\big([0,T];L^1(\Omega)\big)$, with $u \in L^2\big([0,T];H^1(\Omega)\big) \cap C_w^0 ([0, T] ; L^2(\Omega))$, where the index $w$ refers to continuity with respect to the weak topology;
\item $b \in L^\infty\big([0,T];L^2(\Omega)\big) \cap C_w^0 ([0, T] ; L^2(\Omega))$, with $\nabla b\in L^2\big([0,T];L^2(\Omega)\big)$;
\item the mass equation is satisfied in the weak sense: for any $\psi \in \mathcal{D}\big([0, T[\, \times \Omega\big)$, one has
\begin{equation*}
\int_0^T \int_\Omega \bigg\{  \rho \partial_t \psi + \rho u \cdot \nabla \psi  \bigg\} \text{d}x \text{d}t =- \int_\Omega \rho_0 \psi_{|t = 0} \text{d}x\,;
\end{equation*}
\item the divergence-free conditions $\D(u) = \D(b) = 0$ are satisfied in $\mathcal{D}'\big(\,]0, T[ \,\times \Omega\big)$;
\item the momentum equation is satisfied in the weak sense: for any $\phi \in \mathcal{D}\big([0, T[ \,\times \Omega ; \mathbb{R}^2\big)$ such that $\D(\phi) = 0$,
one has
\begin{multline*} 
\hspace{-0.8cm}\int_0^T\!\!\!\int_\Omega \left\{ \rho u \cdot \partial_t \phi + \big(\rho u \otimes u - b \otimes b\big) : \nabla \phi 
- \frac{1}{\epsilon} \rho u^\perp \cdot \phi - \nu(\rho) \nabla u :\nabla \phi  \right\} \text{d}x \text{d}t=- \int_\Omega m_0 \phi_{|t = 0} \text{d}x \,;
\end{multline*}
\item the equation for the magnetic field is satisfied in the weak sense: for all $\phi \in \mathcal{D}([0, T[ \times \Omega;\R^2)$, one has
\begin{equation*}
\int_0^T \int_\Omega \bigg\{  b \cdot \partial_t \phi + \big(u \otimes b - b \otimes u\big):\nabla \phi - \mu(\rho) \curl (b) \curl (\phi)   \bigg\} \dx \text{d}t =- \int_\Omega b_0 \cdot \phi_{|t = 0} \dx \,;
\end{equation*}
\item for almost every $t \in [0, T]\,$, the energy inequality \eqref{EN} is satisfied.
\end{enumerate}
The solution $\big(\rho, u,b\big)$ is said to be \emph{global} if the above conditions hold for all $T > 0$.
\end{defi}

Existence of such finite energy weak solutions has been shown for fluids with density dependent viscosities in the case where there is no magnetic field
(namely $b \equiv 0$ in system \eqref{MHD1} above) by P.-L. Lions. His result even allows the initial density to vanish, under conditions
\eqref{Extra1} and \eqref{Extra2}. We refer to Chapter 2 of \cite{L} for more details and references.

Concerning conductive fluids, more limited results are available. Gerbeau and Le Bris prove in \cite{GB} existence of finite energy weak solutions in a bounded domain of $\mathbb{R}^3$
(their proof can be extended to $\mathbb{R}^2$ or $\mathbb{T}^2$ with standard modifications), but only for fluids with non-vanishing initial densities.
Desjardins and Le Bris do so in \cite{DLe} for cylindrical or toroidal domains based on bounded subsets of $\mathbb{R}^2$, and for flows with translation invariance.

Even if we do not have a full existence result for flows presenting vacuum patches, as described above, our arguments are robust enough to consider possible (mild)
vanishing of the density, in the same spirit of \eqref{Extra1} and \eqref{Extra2}. Therefore, we will work under those conditions, in order to accommodate possible future existence results.

\subsection{Statement of the results}

We consider a sequence of initial data $\big(\rho_{0, \epsilon}, m_{0, \epsilon}, b_{0, \epsilon}\big)_{\epsilon}$ satisfying all the assumptions and uniform bounds described in
Section \ref{SecID} above. We further consider a sequence $\big(\rho_\epsilon, u_\epsilon, b_\epsilon\big)_{\epsilon}$ of finite energy weak solutions (in the sense of Definition \ref{DEFSol})
related to those initial data.
We aim at proving some kind of convergence of the solutions $\big(\rho_\epsilon, u_\epsilon, b_\epsilon\big)_{\epsilon}$ and identify the limit dynamics for
$\epsilon \rightarrow 0^+$ in the form of a PDE solved by the limit points of the sequence. We consider two cases.

Firstly, we consider the case of a \emph{quasi-homogeneous} density, meaning that the initial density $\rho_{0, \epsilon}$ is supposed to be a perturbation of a constant density state,
say $1$ for simplicity. We then write $\rho_{0, \epsilon}\, =\, 1\, +\, \epsilon\, r_{0, \epsilon}$: this assumption simplifies the equations very much. Indeed, on the one hand, at any later time
we still have $\rho_\epsilon\, =\, 1\, +\, \epsilon\, r_\epsilon$, with $r_\epsilon$ solving a linear transport equation
\begin{equation} \label{eq:r_e}
\begin{cases}
\partial_t r_\epsilon + \D(r_\epsilon u_\epsilon) = 0\\[1ex]
\big(r_\epsilon\big)_{|t = 0} = r_{0, \epsilon}\,,
\end{cases}
\end{equation}
thanks to the divergence-free condition $\D (u_\epsilon) = 0$. On the other hand, the momentum equation can be written (in a suitable sense)
\begin{align*}
&\partial_t (\rho_\epsilon u_\epsilon) + \D \big(\rho_\epsilon u_\epsilon \otimes u_\epsilon - b_\epsilon \otimes b_\epsilon\big) + \frac{1}{\epsilon} \nabla \pi_\epsilon + 
\frac{1}{2} \nabla \big( |b_\epsilon|^2 \big) + \frac{1}{\epsilon} \rho_\epsilon u_\epsilon ^\perp - \div\big(\nu(\rho_\eps)\,\nabla u_\eps\big) \\
&\quad= \partial_t u_\epsilon + \D\big( u_\epsilon \otimes u_\epsilon - b_\epsilon \otimes b_\epsilon\big)  - \nu(1) \Delta u_\epsilon + r _\epsilon u_\epsilon^\perp+
\bigg\{ \frac{1}{\epsilon} \nabla \pi_\epsilon + \frac{1}{2} \nabla \big( |b_\epsilon|^2 \big) + \frac{1}{\epsilon} u_\epsilon ^\perp \bigg\} + O(\epsilon)\,,
\end{align*}
where the terms in the brackets (which are singular in $\eps$) are gradient terms, hence do not appear in the weak form of the equations.
Therefore, taking the limit in this case will not be too complicated. In the end, we can prove the following result.

\begin{thm}\label{THQH}
Suppose that $\rho_0 = 1$ and consider a sequence $\big(\rho_{0, \epsilon}, m_{0, \epsilon}, b_{0, \epsilon}\big)_{\epsilon > 0}$ of initial data satisfying the assumptions fixed in Section \ref{SecID}.
Let $\big(\rho_\epsilon, u_\epsilon, b_\epsilon\big)_{\epsilon > 0}$ be a sequence of corresponding finite energy weak solutions to \eqref{MHD1}.
Finally, let $m_0$, $b_0$ and $r_0$ be as in \eqref{EQcv} and define $r_\epsilon = (\rho_\epsilon - 1)/\eps$.

Then, there exists a triplet $\big(r, u, b\big)$  in the space $L^\infty\big(\R_+;L^2(\Omega) \cap L^\infty(\Omega)\big) \times L^\infty\big(\R_+;L^2(\Omega)\big)\times L^\infty\big(\R_+;L^2(\Omega)\big)$,
with $\nabla u$ and $\nabla b$ belonging to $L^2\big(\R_+;L^2(\Omega)\big)$ and $\D u=\div b = 0$, such that, up to the extraction of a subsequence, the following convergence properties hold:
for any $T > 0$, we have
\begin{enumerate}[(1)]
\item $r_\epsilon \stackrel{*}{\rightharpoonup} r\;$ in $\;L^\infty_T (L^2 \cap L^\infty)$, and $\;r_\epsilon \tend r$ in $L^2_T(L^2_{\rm loc})$;
\item $u_\epsilon \stackrel{*}{\rightharpoonup} u\;$ in $\;L^\infty(L^2) \cap L^2_T(H^1)$;
\item $b_\epsilon \stackrel{*}{\rightharpoonup} b\;$ in $\;L^\infty(L^2) \cap L^2_T(H^1)$, and $\;b_\epsilon \tend b$ in  $L^2_T(H^s_{\rm loc})$ for any $ s < 1$.
\end{enumerate}
The limit dynamics is described by a homogeneous MHD-type system, which the triplet $\big(r, u, b\big)$ solves in the weak sense: namely,
\begin{equation}\label{EQLim}
\begin{cases}
\partial_t r \,+\, \D (r\,u)\, =\, 0 \\
\partial_t u \,+\, \D (u \otimes u)\, +\, \nabla \pi\, +\, \dfrac{1}{2} \nabla \big( |b|^2 \big)\, +\, r\,u^\perp\, =\, \nu(1) \,\Delta u \,+\, \D (b \otimes b) \\
\partial_t b \,+\, \D (u \otimes b \,-\, b \otimes u)\, =\, \mu(1)\, \Delta b \\
\D (u) \,=\, \D(b) \,=\, 0\,,
\end{cases}
\end{equation}
for some pressure function $\pi$ and with initial data $\big(r_0,m_0,b_0\big)$. \\
In addition, if $\big(r_0, u_0, b_0\big) \in H^{1 + \beta} \times H^1 \times H^1$, for some $\beta\in\,]0,1[\,$, then the solution $(r,u,b)$ to system \eqref{EQLim}
is unique. As a consequence, the whole sequence $\big(r_\epsilon, u_\epsilon, b_\epsilon\big)_{\epsilon>0}$ converges.
\end{thm}

We will study the limit system \eqref{EQLim} in Section \ref{s:limit-system}; there, we will also specify better in which functional class the uniqueness of solutions holds.

The second case we consider is the case of a \emph{fully non-homogeneous} density, in the sense that the reference density profile $\rho_0$ is non-constant.
Moreover, we need an extra technical assumption on $\rho_0$: we suppose that 
\begin{equation}\label{NDCP}
\forall\,K\subset\Omega\,,\quad K\,\mbox{ compact},\qquad\quad
\text{meas} \left\{x\,\in\,K\;\Bigl|\;\bigl|\nabla\rho_0(x)\bigr|\,\leq\,\delta\right\}\,
\tend_{\delta \rightarrow 0^+}\, 0\,.
\end{equation}
Roughly speaking, we are requiring that, for any fixed compact $K\subset\Omega$, the set of critical points of $\rho_0$ in $K$ is of zero measure.
We remark that this condition is a relaxed version of the condition imposed in \cite{FG} (see also \cite{GSR}, \cite{F_JMFM}), which involves no difficulties in the proof, and has the advantage
of allowing for more general reference densities (for instance, profiles which exponentially decay to some positive constant at $|x|\sim +\infty$).

This case is understandably more difficult, since none of the two previous simplifications can be made.
However, we will see that, by using the structure of system \eqref{MHD1}, we can find an analogous decomposition $\rho_\epsilon\, =\, \rho_0 \,+\, \epsilon\, \sigma_\epsilon$,
where $\big(\sigma_\epsilon\big)_\eps$ is bounded in a low regularity space. Unfortunately, this does not simplify much the singular term, as $\frac{1}{\epsilon}\,\rho_0\,u_\epsilon^\perp$
is not a gradient term.
Another problem is that the bounds we will find on $\big(\sigma_\eps\big)_\eps$ are in such a low regularity space ($H^{-3-\delta}$ in fact) that taking the limit $\epsilon \rightarrow 0^+$ directly in the momentum
equation is impossible. We will need to, instead, work on the vorticity and take the $\curl$ of the momentum equation.

The result in the fully non-homogeneous case is contained in the next statement.
\begin{thm}\label{THNH}
Assume that $\rho_0$ satisfies condition \eqref{NDCP} and (when $\Omega=\R^2$) either \eqref{Extra1} or \eqref{Extra2}.
Consider a sequence $\big(\rho_{0, \epsilon}, m_{0, \epsilon}, b_{0, \epsilon}\big)_{\epsilon > 0}$ of initial data satisfying the assumptions fixed in Section \ref{SecID}, and
let $\big(\rho_\epsilon, u_\epsilon, b_\epsilon\big)_{\epsilon > 0}$ be a sequence of corresponding weak solutions to \eqref{MHD1}.
Finally, let $m_0$, $b_0$ and $r_0$ be as in \eqref{EQcv} and define $\sigma_\epsilon\,:=\,\big(\rho_\epsilon - 1\big)/\eps$.

Then, there exist $\sigma \in L^\infty_{\rm loc}\big(\R_+;H^{-3-\delta}(\Omega)\big)$ for any $\delta > 0$ arbitrarily small, $u \in L^\infty\big(\R_+;L^2(\Omega)\big)$ and
$b \in L^\infty\big(\R_+;L^2(\Omega)\big)$, with $\nabla u$ and $\nabla b$ in $L^2\big(\R_+;L^2(\Omega)\big)$ and $\D(\rho_0 u)\,=\,\div u\,=\,\div b\,=\,0$, such that,
up to the extraction of a subsequence, the following convergence properties hold true: for any $T > 0$, one has
\begin{enumerate}[(1)]
\item $\rho_\epsilon \tend \rho_0\;$ in $\;L^2_T(L^2_{\rm loc})$;
\item $\sigma_\epsilon \stackrel{*}{\rightharpoonup} \s\;$ in $\;L^\infty_T(H^{-3-\delta})$, for all arbitrarily small $\de>0$;
\item $u_\epsilon \stackrel{*}{\rightharpoonup} u\;$ in $\;L^\infty(L^2) \cap L^2_T(H^1)$;
\item $b_\epsilon \stackrel{*}{\rightharpoonup} b\;$ in $\;L^\infty(L^2) \cap L^2_T(H^1)$, and $\;b_\epsilon \tend b$ in  $L^2_T(H^s_{\rm loc})$ for any $ s < 1$.
\end{enumerate}
Moreover, there exists a distribution $\Gamma\in\mc D'\big(\R_+\times\Omega\big)$ of order at most $1$ such that 
\begin{equation*}
\begin{cases}
\partial_t \Big( \curl(\rho_0\,u)\, -\, \sigma \Big)\,-\,\curl\Big(\div\big(\nu(\rho_0)\,\nabla u\big)\Big)\,+\,\curl \Big(\rho_0 \nabla \Gamma\, -\, \D(b \otimes b) \Big)\, =\, 0 \\
\partial_t b \,+\, \D \big(u \otimes b\, -\, b \otimes u\big)\, =\,    \nabla^\perp \big ( \mu(\rho_0)\, \curl (b) \big) \\
\D(\rho_0\, u)\, =\, 0 \\
\D (u)\, =\, \D(b)\, =\, 0\,,
\end{cases}
\end{equation*}
with initial data $\Big(\curl \big(\rho_0\,u\big)\,-\,\s\Big)_{|t=0}\,=\,\curl (m_0)\,-\,r_0\;$ and $\;b_{|t=0}\,=\,b_0$.
\end{thm}

\begin{rmk}
The summand $\rho_0 \nabla \Gamma$ can be interpreted as a Lagrange multiplier associated to the constraint $\D(\rho_0 u) = 0$, just as the pressure term $\nabla \pi$ in \eqref{MHD1}
and \eqref{EQLim} can be seen as a Lagrange multiplier for the incompressibility constraint $\D(u) = 0$.
\end{rmk}

\section{Uniform bounds and convergence properties} \label{s:ub}

The next three sections are devoted to the proofs of Theorems \ref{THQH} and \ref{THNH}. In all that follows, $\big(\rho_{0, \epsilon}, m_{0, \epsilon}, b_{0, \epsilon}\big)_{\epsilon}$
is a sequence of initial data satisfying all the assumptions and uniform bounds described in Section \ref{SecID} above, and $\big(\rho_\epsilon, u_\epsilon, b_\epsilon\big)_{\epsilon}$
is an associated sequence of finite energy weak solutions related to those initial data, in the sense of Definition \ref{DEFSol} above.

In this section, we first use uniform bounds on the solutions to prove their weak convergence. Then, we focus on convergence results for the density, which we will need later.

\subsection{Uniform bounds}\label{SecUB}

In this section, we establish uniform bounds (i.e. bounds independent of $\epsilon$) on the sequence of solutions $\big(\rho_\epsilon, u_\epsilon, b_\epsilon\big)_{\epsilon > 0}$,
thus enabling us to extract weakly converging subsequences.

First of all, we notice that the solutions satisfy the energy inequality \eqref{EN} (this is point (viii) of Definition \ref{DEFSol}): for almost every $t>0$ fixed, we have
\begin{equation*}
\int_\Omega \bigg( \rho_\epsilon(t) |u_\epsilon(t)|^2 + |b_\epsilon(t)|^2 \bigg) \dx +  \int_0^t \int_\Omega \bigg( \nu_* |\nabla u_\epsilon|^2 +  \mu |\nabla b_\epsilon|^2 \bigg) \dx \text{d}s
\leq \int_\Omega \left( \frac{|m_{0, \epsilon}|^2}{\rho_{0, \epsilon}} + |b_{0, \epsilon}|^2 \right) \dx\,.
\end{equation*}
In view of our assumptions on the initial data, the right-hand side of the previous inequality is uniformly bounded. Thus we get
\begin{align}
&\big( \sqrt{\rho_\epsilon}\, u_\epsilon \big)_{\epsilon > 0}\,,\;\big(b_\epsilon\big)_{\epsilon > 0}\; \subset\; L^\infty\big(\R_+;L^2(\Omega)\big)\,, \label{ub:u-b} \\
&\qquad\qquad\qquad\qquad\big(\nabla u_\epsilon\big)_{\epsilon > 0}\,,\; \big(\nabla b_\epsilon\big)_{\epsilon > 0}\; \subset\; L^2\big(\R_+;L^2(\Omega)\big)\,. \label{ub:nabla_u-b}
\end{align}

Secondly, because both $\rho_\epsilon$ and (in the quasi-homogeneous case) $r_\epsilon\, :=\,\big(\rho_\epsilon - 1\big)/\eps$ solve a pure transport equation by the divergence-free vector
field $u_\epsilon$ (keep in mind \eqref{eq:r_e} above), we see that, for all $\eps>0$, one has
\begin{align*}
\forall \; 0 \leq \alpha \leq \beta < +\infty\,, \qquad  \text{meas}\big\{ \alpha \leq \rho_\epsilon \leq \beta \big\}\, =\, \text{meas}\big\{ \alpha \leq \rho_{0, \epsilon} \leq \beta \big\}\,,
\end{align*}
and the same holds for $r_\eps$ (this is the same property as in Theorem 2.1, Chapter 2, p. 23 of \cite{L}).
On the one hand, this of course implies that $\big(\rho_\epsilon\big)_{\epsilon > 0} \subset L^\infty(L^\infty)$ and $\big(r_\epsilon\big)_{\epsilon > 0} \subset L^\infty(L^2 \cap L^\infty)$,
together with the bounds
\begin{equation} \label{ub:rho_e}
0\,\leq\,\rho_\eps\,\leq\,2\,\rho^*\qquad\qquad\mbox{ almost everywhere in }\quad \R_+\times\Omega\,,
\end{equation}
in view of the assumptions on the initial datum $\rho_{0,\eps}$.
Therefore, up to extracting a subsequence, we gather the convergences
\begin{equation}\label{conv:weak-rho-r}
\rho_\epsilon \wtend^* \rho \quad \text{ in }\; L^\infty (L^\infty)\qquad\qquad \mbox{ and }\qquad\qquad  r_\epsilon \wtend^* r \quad \text{ in }\; L^\infty (L^2 \cap L^\infty)\,,
\end{equation}
for some $\rho\in L^\infty(L^\infty)$ and $r\in L^\infty(L^2\cap L^\infty)$.
On the other hand, the same property also shows that, for almost every $t>0$, the density $\rho_\epsilon(t)$ satisfies the extra regularity properties 
which we had required if $\Omega = \mathbb{R}^2$, and it does so independently of $t>0$ and uniformly with respect to $\epsilon>0$:
\begin{align}
\left( \frac{1}{\rho_{\epsilon}(t)}\,\mds{1}_{\left\{\rho_{\epsilon}(t) < \delta\right\}} \right)_{\!\!\epsilon > 0}\, \subset\, L^1(\Omega)\qquad \mbox{ or }\qquad 
\left(\big(\oline{\rho}\,-\,\rho_{\epsilon}(t)\big)^+ \right)_{\!\!\epsilon > 0}\,\subset\,L^{p_0}(\Omega)\,, \label{EQEx1}
\end{align}
where $\de>0$, $p_0\in\,]1,+\infty[\,$ and $\oline\rho>0$ are the same as in conditions \eqref{Extra1} and \eqref{Extra2}.
%
%

Finally, we also see that $\big(u_\epsilon\big)_{\epsilon > 0}$ is in fact uniformly bounded in $L^2_T(L^2)$ for any finite time $T > 0$.
Indeed, if $\Omega = \mathbb{R}^2$, this is a consequence of either one of the two previous conditions in \eqref{EQEx1} (see also \cite{L}, point 8 in Remark 2.1 pp. 24-25).
If $\Omega = \mathbb{T}^2$ instead, the same can be shown without the extra assumptions, by means of the Poincaré-Wirtinger inequality (see again \cite{L}, Subsection 2.3 p. 37).
Therefore, up to an extraction, we deduce that
\begin{equation} \label{conv:u-b}
u_\epsilon \wtend u \quad \text{ in }\; L^2_{\rm loc}\big(\R_+;H^1(\Omega)\big)\qquad\mbox{ and }\qquad b_\epsilon \wtend b \quad \text{ in }\; L^2_{\rm loc}\big(\R_+;H^1(\Omega)\big)\,,
\end{equation}
for suitable functions $u$ and $b$ belonging to $L^2_{\rm loc}\big(\R_+;H^1(\Omega)\big)$.

Remark that, in fact, we have a more precise convergence property for the magnetic fields: in view of \eqref{ub:u-b}-\eqref{ub:nabla_u-b}, we know that
$b\in L^\infty\big(\R_+;L^2(\Omega)\big)$, with $\nabla b\in L^2\big(\R_+;L^2(\Omega)\big)$ and, up to an extraction, we have
the convergences  $b_\eps\,\stackrel{*}{\rightharpoonup}\,b$  in $L^\infty(L^2)$ and 
$\nabla b_\eps\,\rightharpoonup\,\nabla b$ in $L^2(L^2)$.
We will resort to those precise features when needed.

\subsection{Strong convergence of the densities}\label{SecStrDen}

This section is dedicated to the quest of pointwise convergence for the $\rho_\epsilon$. This will be useful for two reasons.
Firstly, strong convergence makes the proofs simpler: on many occasions, the use of paradifferential calculus in \cite{FG} can be replaced by more elementary arguments.
Secondly, pointwise convergence is necessary to deal with the viscosity and resistivity terms: since we only have weak convergence of the velocity fields, strong convergence 
of both $\nu(\rho_\epsilon)$ and $\mu(\rho_\epsilon)$ is required to achieve convergence of the product terms $\D\big(\nu(\rho_\epsilon)\nabla u_\epsilon\big)$ and
$\nabla^\perp\big(\mu(\rho_\epsilon)\curl u_\epsilon\big)$.

However, the uniform bounds alone are insufficient to prove the strong convergence we seek. So far, we have only obtained mere weak convergence
$\rho_\epsilon \stackrel{*}{\rightharpoonup} \rho$ in $L^\infty(L^\infty)$ (recall uniform bound \eqref{conv:weak-rho-r} above). We now resort to the arguments of Di Perna and P.-L. Lions \cite{DL}:
if somehow we proved that 
\begin{equation}\label{EQ1}
\rho_\epsilon^2\, \wtend^*\, \rho^2 \qquad\qquad \text{ in }\quad L^\infty\big(\R_+;L^\infty(\Omega)\big)\, ,
\end{equation}
then, by using the characteristic function $\mds{1}_K$ of a compact subset $K \subset\Omega$ as a test function, we would recover convergence of the $L^2$ norms.
Using the euclidean structure of $L^2_T (L^2(K))$, we would then deduce local strong convergence, hence pointwise convergence, after extraction.

Therefore, the argument boils down to proving \eqref{EQ1}. The quadratic non-linearity is the main challenge as, by the uniform bounds $\big(\rho_\epsilon\big)_{\epsilon > 0} \subset L^\infty(L^\infty)$,
we only know that there exists some function $g \in L^\infty(L^\infty)$ such that $\rho_\epsilon^2\, \wtend^* \,g$ in $L^\infty(L^\infty)$,
and this function $g$ need not be $\rho^2$. The trick is that both $g$ and $\rho^2$ are (weak) solutions of the transport PDE
\begin{equation}\label{EQ4}
\begin{cases}
\partial_t a + u \cdot \nabla a = 0\\[1ex]
a_{|t = 0}= a_0\,,
\end{cases}
\end{equation}
with same initial datum and divergence-free velocity field. After the work \cite{DL}, problem \eqref{EQ4} is well-posed, so $g$ and $\rho^2$ must be equal.
To sum up, we get the following statement.

\begin{prop}\label{PropStrCnv}
The convergence property \eqref{EQ1} holds true. In particular, in the limit $\veps\ra0^+$, we have the strong convergence
\begin{equation*}
\rho_\epsilon\,\longrightarrow\, \rho\qquad\qquad \text{ in }\qquad L^2_{\rm loc}(\mathbb{R_+} \times \Omega)\,,
\end{equation*}
and, up to the extraction of a suitable subsequence, the convergence holds also almost everywhere in $\mathbb{R}_+ \times \Omega$.
\end{prop}

Before proving the previous proposition, some preliminary lemmas are in order.
First of all, we establish that all the $\rho_\epsilon^2$ are solutions to the continuity equation. Since this fact is to be shown for all $\epsilon > 0$ independently,
we drop the $\epsilon$ indices for more clarity.

\begin{lemma}\label{LemmaR2}
Let $\rho_0\in L^\infty$ and $u \in L^2_{\rm loc}\big(\R_+;H^1\big)$ be a divergence-free vector field. Let $\rho \in L^\infty(L^\infty)$ be a weak solution to the Cauchy problem
\eqref{EQ4}, with initial datum $\rho_0$.

Then $\rho^2$ is also a weak solution of the same equation, related to the initial datum $\rho_0^2$. 
\end{lemma}

\begin{proof}[Proof of Lemma \ref{LemmaR2}]
We wish to prove that $\rho^2$ is a weak solution of \eqref{EQ4}, with initial datum $\rho_0^2$: this means that, for all $T>0$ and all $\psi\in\mathcal{D}([0, T[ \times\Omega)$, one has
\begin{equation*}
\int_0^T \int_\Omega \rho^2 \left( \partial_t \psi + u \cdot \nabla \psi \right) \text{d}x \text{d}t + \int_\Omega \rho_0^2 \,\psi_{|t = 0} \text{d}x\, =\, 0\,.
\end{equation*}

We consider a smoothing kernel $\big(\mu_\alpha\big)_{\alpha>0}$ such that, if $\Omega = \mathbb{R}^2$, we have $\mu_\alpha(x) = \mu\big(x/\alpha\big)/\alpha^d$, where
$\mu \in C^\infty(\mathbb{R}^2)$ is such that $\Supp\mu \subset B(0, 1)$ and  $\mu(x) = \mu(-x)$.
For $\Omega = \mathbb{T}^2$, instead, we set $\mu'_\alpha(x) = \sum_{k \in \mathbb{Z}^2} \mu_\alpha(x + k)$.

For all $\alpha>0$, we define $\rho_\alpha = \mu_\alpha * \rho$ (use $\mu_\alpha'$ instead of $\mu_\alpha$ if $\Omega=\T^2$). Then, using also
the divergence-free condition for $u$ and the evenness of $\mu_\alpha$, we deduce that $\rho_\alpha$ solves (in the weak sense) the following approximate equation: 
\begin{equation*}
\partial_t \rho_\alpha\, +\, u \cdot \nabla \rho_\alpha \,=\,\big[ u\cdot\nabla\,,\, \mu_\alpha *\big] \rho\,,\qquad\big(\rho_\alpha\big)_{|t=0}\,=\,\mu_\alpha\,*\,\rho_{0}\,,
\end{equation*}
where we have denoted $[ u\cdot\nabla\,,\, \mu_\alpha *]$ the commutator between $u\cdot\nabla$ and the convolution by $\mu_\alpha$.
Multiplying this equation by $2\rho_\alpha$ shows that
\begin{equation}\label{EQ3}
\partial_t \left( \rho_\alpha^2 \right) + u \cdot \nabla \left( \rho_\alpha^2 \right) = 2\rho_\alpha\,\big[ u\cdot\nabla\,,\, \mu_\alpha *\big] \rho\,,
\qquad\big(\rho_\alpha^2\big)_{|t=0}\,=\,\left(\mu_\alpha\,*\,\rho_{0}\right)^2\,.
\end{equation}
The space differentiation $2\rho_\alpha \nabla \rho_\alpha = \nabla (\rho_\alpha^2)$ is justified because $\rho_\alpha(t) \in C^\infty$ for almost all times $0 \leq t \leq T$ (for any fixed $T>0$),
and the time differentiation $2 \rho_\alpha \partial_t \rho_\alpha = \partial_t (\rho_\alpha^2)$ is justified because $\rho_\alpha \in W^{1, 2}_T(H^s_{\rm loc})$ for every $\alpha>0$ and $s\geq0$.
This comes from the property $\rho u \in L^2_T(L^2)$ and the relation $\partial_t \rho_\alpha\, =\, -\, \D \big(   \mu_\alpha * ( \rho\, u) \big)$
(which follows from \eqref{EQ4} and $\div u=0$), which implies $\d_t\rho_\alpha\in L^2_T(H^s_{\rm loc})$.

Our next goal is to take the limit $\alpha\ra0^+$ in the weak formulation of \eqref{EQ3}, namely in the relation
\begin{equation}\label{EQ9}
\int_0^T\!\!\!\int_\Omega \rho_\alpha^2 \bigg\{  \partial_t \psi + u\cdot \nabla \psi \bigg\}\dx\dt + \int_\Omega \big(\mu_\alpha * \rho_0\big)^2 \psi_{|t = 0}\,\dx +  
\int_0^T\!\!\!\int_\Omega 2 \psi \rho_\alpha\,\big[ u\cdot\nabla, \mu_\alpha * \big]\rho\,\dx\dt \, = 0\,,
\end{equation}
for any arbitrary test function $\psi \in \mathcal{D}\big([0, T[\, \times\Omega\big)$.
On the one hand, we remark that, since $u \in L^2_T(H^1)$ and $\rho \in L^\infty(L^\infty)$, we can apply Lemma II.1 in \cite{DL} to get
\begin{equation*}
\big[u \cdot \nabla\,,\, \mu_\alpha *\big]\rho\, \tend_{\alpha \rightarrow 0^+}\, 0 \qquad\qquad \text{ in }\quad L^1_T(L^2_{\rm loc})\,.
\end{equation*}
Using this, we find that the commutator term in \eqref{EQ3} cancels in the limit $\alpha \rightarrow 0^+$, as, for every compact $K \subset \Omega$, we have
$$ 
\big\|   2\rho_\alpha \left[ u\cdot \nabla, \mu_\alpha * \right] \rho   \big\|_{L^1_T(L^2(K))} \leq
2 \| \rho \|_{L^\infty_{t, x}} \,\big\|\left[ u\cdot \nabla, \mu_\alpha * \right] \rho  \big\|_{L^1_T(L^2(K))} \tend_{\alpha \rightarrow 0^+} 0\,.
$$ 
On the other hand, since $\rho \in L^\infty(\R_+\times\Omega)$, by standard properties of mollification kernels we gather the strong convergence
$\rho_\alpha \tend \rho$ in e.g. $L^2_T(L^2_{\rm loc})$ when $\alpha\ra0^+$, for any fixed $T>0$. By the same token, we also have $\mu_\alpha * \rho_0\tend\rho_0$
in $L^2_{\rm loc}(\Omega)$, in the limit $\alpha\ra0^+$. Thanks to those properties, it is easy to take the limit in the first and second term in \eqref{EQ9}.

The proof of the lemma is now completed.
\end{proof}

We also need the following result.
\begin{lemma} \label{l:g}
Let $g\in L^\infty(\R_+\times\Omega)$ be any weak-$*$ limit point of the sequence $\big(\rho_\eps^2\big)_{\eps > 0}$ with respect to the $L^\infty(\R_+\times\Omega)$ topology.
Let $\rho_0$ be the limit density profile and $u$ the limit velocity field identified in \eqref{conv:u-b}.

Then $g$ is a solution of the linear transport equation 
\begin{equation} \label{EQ11}
\d_tg\,+\,u\cdot\nabla g\,=\,0\,,\qquad\qquad\mbox{ with }\qquad g_{|t=0}\,=\,\rho_0^2\,.
\end{equation}
\end{lemma}

\begin{proof}[Proof of Lemma \ref{l:g}]
By Lemma \ref{LemmaR2} above, we know that the $\rho_\eps^2$ solve the transport equation
\begin{equation} \label{EQ10}
\partial_t \left( \rho_\epsilon^2 \right) + u_\eps \cdot \nabla \left( \rho_\epsilon^2 \right) = 0\,, 
\end{equation}
related to the initial datum $\big(\rho_\eps^2\big)_{|t=0}\,=\,\rho_{0,\eps}^2$. Therefore, for proving our claim, it is enough to take the limit in the weak formulation
of the previous equation.
The main issue is showing that $\rho_\epsilon^2\, u_\epsilon$ converges in (say) $\mc D'$ to $gu$: for doing so, we resort to some arguments of \cite{FG} (see Paragraph 3.1.2. therein).
The basic idea is to use the transport equation solved by the $\rho_\epsilon^2$ to
trade space regularity for time compactness.

We start by remarking that, in view of \eqref{EQ10}, one has $\partial_t (\rho_\epsilon^2) = - \D (\rho_\epsilon^2 u_\epsilon)$.
Since $|\rho_\epsilon^2 u_\epsilon| \leq \left( \rho^*  \right)^{3/2} |\sqrt{\rho_\epsilon} u_\epsilon|$ by use of \eqref{ub:rho_e} and \eqref{ub:u-b}, we infer that
$\big( \partial_t (\rho_\epsilon^2) \big)_{\epsilon>0} \subset L^\infty(H^{-1})$, hence
\begin{equation*}
\big( \rho_\epsilon^2\big)_{\epsilon>0} \subset W^{1, \infty}_T(H^{-1}_{\rm loc})\,,
\end{equation*}
where the localisation in space comes from the fact that the initial data $\rho_0^2$ is just $L^\infty(\Omega)$.
Let now $\theta \in\, ]0, 1[\,$: standard Sobolev interpolation gives, for almost all $0 \leq s, t \leq T$, the estimate
\begin{equation*}
\left\|\big( \rho_\epsilon(t) - \rho_\epsilon(s) \big)\chi \right\|_{H^{-\theta}}\, \leq\, \left\|\big( \rho_\epsilon(t) - \rho_\epsilon(s) \big) \chi \right\|_{H^{-1}}^{\theta}\,
\left\| \big( \rho_\epsilon(t)-\rho_\epsilon(s)\big) \chi \right\|_{L^2}^{1 - \theta},
\end{equation*}
where $\chi \in \mathcal{D}(\Omega)$ is an arbitrary compactly supported function. This shows that $(\rho_\epsilon^2)_{\epsilon>0}$ is bounded in every space
$C^{0,\theta}_T (H^{-\theta}_{\rm loc})$. Therefore, by the Ascoli-Arzel\`a theorem, we gather the strong convergence
\begin{equation*}
\rho_\epsilon^2 \,\tend\, g \qquad\qquad \text{ in }\quad C^{0, \theta}\big([0,T];H^{-\theta}_{\rm loc}(\Omega)\big)\,,
\end{equation*}
for all $0 < \theta < 1$ and all fixed $T>0$. Combining this property with the Lemma \ref{Para}, which provides continuity of the function product $(a, b) \mapsto ab$ in the
$H^{-\theta} \times H^1 \rightarrow H^{-\theta -\delta}$ topology (for $\delta > 0$ arbitrarily small), we get 
\begin{equation*}
\rho_\epsilon^2 \,u_\epsilon\, \wtend\, g\,u \qquad\qquad \text{ in }\quad \mathcal{D}'\big(\,]0, T[\,\times \Omega\big)\,.
\end{equation*}
It is now possible to take the limit $\epsilon \rightarrow 0^+$ in the weak form of equation \eqref{EQ10}, thus recovering equation \eqref{EQ11}.
\end{proof}

We can now complete the proof to Proposition \ref{PropStrCnv}.

\begin{proof}[Proof of Proposition \ref{PropStrCnv}]
In view of Lemmas \ref{l:g} and \ref{LemmaR2}, it follows that both $\rho^2$ and $g$ are weak solutions to the initial value problem \eqref{EQ11}, and they both belong to
$L^\infty\big(\R_+\times\Omega\big)$.
But problem \eqref{EQ11} is in fact well-posed in the previous space, as a consequence of Di Perna and P.-L. Lions theory. More precisely, in order to apply their uniqueness result
(see Theorem II.2 of \cite{DL}),
we have to make sure that the limit velocity field $u\in L^2_{\rm loc}\big(\R_+;H^1(\Omega)\big)$ fulfills the following condition: for any  fixed $T>0$,
\begin{equation}\label{EQ13}
\frac{u(t, x)}{1 + |x|}\; \in\; L^1_T(L^1) + L^1_T(L^\infty)\,.
\end{equation}
To see this, let us set an arbitrary $R > 0$ and decompose according to whether $|u| < R$ or not:
\begin{equation*}
\frac{|u(t, x)|}{1 + |x|} \, = \,\mds{1}_{\{|u| < R\}}\, \frac{|u(t, x)|}{1 + |x|}\,  +\, \mds{1}_{\{|u| \geq R\}} \,\frac{|u(t, x)|}{1 + |x|}\,.
\end{equation*}
On the one hand, the measure of the set $A_R(t) := \{ |u(t)| \geq R \}$ is bounded by the Bienaym\'e-Chebyshev inequality, as
\begin{equation*}
\text{meas}\,A_R(t)
\,\leq\, \frac{1}{R^2} \int_\Omega |u(t, x)|^2 \,\text{d}x\,.
\end{equation*}
Observe that the term on the right-hand side of the previous estimate belongs to $L^1_T$ for all fixed $T>0$.
Therefore, H\"older's inequality yields
\begin{equation*}
\int_0^T \int_{A_R} \frac{|u(t, x)|}{1 + |x|}\, \text{d}x \dt\, \leq\, \int_0^T \| u(t) \|_{L^1(A_R)}\, \dt \,\leq\, \int_0^T \| u(t) \|_{L^2}\, \big(\text{meas}\,A_R(t)\big)^{1/2}\,\dt \,<\, +\infty\,,
\end{equation*}
implying that $\mds{1}_{A_R(t)}\, u(t, x)\,(1 + |x|)^{-1} \in L^1_T(L^1)$.
On the other hand, we obviously have 
\begin{equation*}
\mds{1}_{\{|u(t)| < R\}}\, \frac{|u(t, x)|}{1 + |x|}\, \leq\, R\, \in\, L^1_T(L^\infty)\,.
\end{equation*}

Therefore, \eqref{EQ13} is indeed satisfied by $u$, so we can apply Theorem II.2 of \cite{DL}. This result implies that we do have $\rho^2 \equiv g$ almost everywhere.
In particular, we also deduce the weak convergence \eqref{EQ1}.
As already remarked above, that property in turn yields local strong convergence of the $\rho_\epsilon$ to $\rho$.
Indeed, let $K \subset\Omega$ be a compact set; then, using \eqref{EQ1} we get
\begin{equation*}
\langle \rho_\epsilon^2, \mds{1}_K \rangle_{L^\infty_{t,x}\times L^1_{t,x}}\, =\, \| \rho_\epsilon \|_{L^2_T{L^2(K)}}^2 \,\tend\,
\langle \rho^2, \mds{1}_K \rangle_{L^\infty_{t,x}\times L^1_{t,x}}\,=\,\| \rho \|_{L^2_T{L^2(K)}}^2\,.
\end{equation*}
At this point, the fact that $L^2_T\big(L^2(K)\big)$ has a Euclidean structure gives strong convergence: 
because of the weak-$*$ convergence \eqref{conv:weak-rho-r} of the $\rho_\epsilon$, we infer that $\langle \rho_\epsilon,  \rho \rangle_{L^\infty(\,]0,T[\,\times K) \times L^1(\,]0,T[\,\times K)}$
tends to $\| \rho \|_{L^2_T(L^2(K))}$ and hence
\begin{equation*}
\| \rho_\epsilon  - \rho\|^2_{L^2_T(L^2(K))}\,=\,\|\rho_\epsilon \|^2_{L^2_T(L^2(K))}\, +\, \|\rho \|^2_{L^2_T(L^2(K))}\,-\,2 \int_0^T \int_K \rho_\epsilon \rho \, \text{d}x \text{d}t\,
\tend_{\epsilon \rightarrow 0^+}\,0\,.
\end{equation*}
In particular, after extracting one more time, we deduce the pointwise convergence $\rho_\epsilon\, \tend\, \rho$, in the limit $\eps\ra0^+$.
\end{proof}

\subsection{The singular part of the equations}\label{SecSg}

In this part, we focus our attention on the singular part of system \eqref{MHD1}, namely on the term $\epsilon^{-1}\big( \nabla \pi_\epsilon + \rho_\epsilon u_\epsilon^\perp \big)$
in the momentum equation.
Note that any singular gradient term disappears in its weak formulation, due to the divergence-free condition on the test functions.

\begin{prop}\label{PropDebut}
Let $\bigl(\rho_\eps,u_\eps,b_\eps\bigr)_{\eps>0}$ be a sequence of weak solutions to system \eqref{MHD1}, associated with the sequence of initial data $\bigl(\rho_{0,\eps},u_{0,\eps}, b_{0,\eps}\bigr)_{\eps > 0}$ satisfying
the assumptions fixed in Subsection \ref{SecID}. Let $(\rho,u,b)$ be a limit point of the sequence $\bigl(\rho_\eps,u_\eps,b_\eps\bigr)_{\eps>0}$, as identified in Subsection \ref{SecUB}.
\begin{enumerate}[1)]
\item In the case of a quasi-homogeneous density, for all test function $\phi \in \mathcal{D}\big(\mathbb{R}_+ \times \Omega ; \mathbb{R}^2\big)$ such that $\div \phi=0$, we have
\begin{equation*}
\frac{1}{\epsilon} \int_0^{+ \infty} \int_\Omega \rho_\epsilon\, u_\epsilon^\perp\cdot\phi \dx \dt\, \tend_{\epsilon \rightarrow 0^+}\, \int_0^{+ \infty} \int_\Omega r\,u^\perp \cdot \phi \dx \dt\,.
\end{equation*}

\item In the fully non-homogeneous case, the limit density satisfies $\rho(t, x) = \rho_0(x)$ for almost every $(t,x)\in \mathbb{R}_+ \times\Omega$. Moreover, we have the relations
$\D (\rho_0 \,u)\, =\, \D u\, =\, 0$ almost everywhere in $\mathbb{R}_+ \times \Omega$. In particular, $\nabla \rho_0 \cdot u = 0$ almost everywhere in $\mathbb{R}_+  \times \Omega$.
\end{enumerate}
\end{prop}

\begin{proof}

We start by attending to the quasi-homogeneous setting, where the singularity \textsl{de facto} disappears. Indeed, we can write
\begin{equation*}
\frac{1}{\epsilon} \left( \nabla \pi_\epsilon + \rho_\epsilon u_\epsilon^\perp \right)\, =\,\frac{1}{\epsilon} \left( \nabla \pi_\epsilon + u_\epsilon^\perp \right) + r_\epsilon u_\epsilon\,,
\end{equation*}
and the terms in the brackets are perfect gradients.
Therefore, if $T > 0$ and $\phi \in \mathcal{D}\big( [0, T[ \, \times \Omega ; \mathbb{R}^2\big)$ is a divergence-free test function, one gets
\begin{equation*}
\frac{1}{\epsilon} \int_0^T \int_\Omega \rho_\epsilon\, u_\epsilon^\perp\cdot \phi \dx \dt \,=\, \int_0^T \int_\Omega r_\epsilon\, u_\epsilon^\perp \cdot \phi \dx \dt\,.
\end{equation*}
To take the limit $\epsilon \rightarrow 0^+$ in this last integral, we observe that the functions $r_\epsilon$ solve the linear transport equation \eqref{eq:r_e}; hence,
Proposition \ref{PropStrCnv} applies to the sequence $\big(r_\epsilon\big)_{\epsilon>0}$, yielding the strong convergence
\begin{equation*}
r_\epsilon\, \tend_{\epsilon \rightarrow 0^+}\, r \qquad\qquad \text{ in }\qquad L^2_T(L^2_{\rm loc})\,,
\end{equation*}
for any fixed $T>0$. Using the weak convergence \eqref{conv:u-b} for $\big(u_\eps\big)_\eps$, we finally infer that
\begin{equation*}
r_\epsilon\, u_\epsilon\, \tend\, r\,u \qquad\qquad \text{ in }\qquad \mathcal{D}'\big(\,]0, T[\, \times \Omega\big)\,.
\end{equation*}
This completes the proof of the first property of the proposition.

\medbreak
The study of the fully non-homogeneous case is very much similar, with the exception that the singularity does not disappear.
We start by remarking that, exactly as done above, as a consequence of Proposition \ref{PropStrCnv} we deduce that
$\rho_\epsilon\, u_\epsilon\, \tend\, \rho\,u$ in the sense of distributions.
Now, multiplying the momentum equation in its weak form by $\epsilon$, we see that, for any divergence-free $\phi \in \mathcal{D}\big([0, T[\, \times \Omega ; \mathbb{R}^2\big)$, one has
\begin{equation*}
\int_0^T \int_\Omega \rho_\epsilon\, u_\epsilon^\perp \cdot \phi\, \dx \dt\, =\, O(\epsilon)\,.
\end{equation*}
Indeed, the uniform bounds established in Section \ref{SecUB} show that, for any $T>0$, one gets $\big(\rho_\epsilon\, u_\epsilon\big)_{\epsilon>0}\,\subset\,L^2_T(L^2)$,
$\big(\rho_\epsilon\, u_\epsilon \otimes u_\epsilon\big)_{\epsilon>0}$ and $\big(b_\epsilon \otimes b_\epsilon\big)_{\epsilon>0}$ uniformly bounded in $L^\infty_T(L^1)$, and
$\big( \nu(\rho_\epsilon) \nabla u_\epsilon \big)_{\epsilon>0}\,\subset\,L^2_T(L^2)$, while by assumption $\big(m_{0, \epsilon} \big)_{\epsilon>0}$ is bounded in $L^2$. Therefore, we can take
the limit $\eps\ra0^+$ and get, for any test function $\phi$ as above, that
\begin{equation*}
\int_0^T \int_\Omega \rho\, u^\perp \cdot \phi \,\dx \dt\, =\, 0\,.
\end{equation*}
This means that $\rho u^\perp = \nabla p$ for some suitable function $p$, which implies, after taking the $\curl$, the constraint $\D(\rho u) = 0$. With this latter relation at hand,
we look at the mass equation: using again that $\rho_\epsilon\, u_\epsilon\,\ra\, \rho\, u$ in $\mc D'$, we have no trouble in taking the limit $\epsilon \rightarrow 0^+$, and we obtain
\begin{equation*}
\partial_t \rho\, +\, \D(\rho\, u)\, =\,0\,,\qquad\quad\mbox{ which implies }\qquad\quad \partial_t \rho = 0\,.
\end{equation*}
The consequence is that $\rho(t,x)= \rho_0(x)$ for almost every $(t,x)$. In particular, the relation $\D(\rho_0 u) = \nabla \rho_0 \cdot u = 0$ is satisfied almost everywhere in $\mathbb{R}_+ \times \Omega$.
\end{proof}

\subsection{Quantitative convergence properties for the density}\label{SecSig}

This paragraph centers on the density functions in the fully non-homogeneous case.
In Subsection \ref{SecStrDen} above, we have shown strong convergence of the densities in $L^2_{\rm loc}\big(\R_+\times\Omega\big)$. However, this convergence is not enough for the
convergence in the fully non-homogeneous case, since neither quantitative nor uniform with respect to time.

As we have seen, there is no obvious way to write $\rho_\epsilon = \rho_0 + \epsilon \sigma_\epsilon$, with $\big(\sigma_\epsilon\big)_\eps$ being bounded in some Banach space.
Nonetheless, it turns out that the previous decomposition holds true thanks to the structure of the system, but will yield uniform bounds for $\big(\sigma_\epsilon\big)_\eps$
in the very low regularity space $H^{-3-\delta}$.
We refer to Subsection 3.3 of \cite{FG} for the original proofs (the presence of the magnetic field introduces only minor modifications).

\begin{prop}\label{PropSigma}
Define the functions $\sigma_\epsilon\,:=\,\big(\rho_\epsilon - \rho_0\big)/\eps$. Then the sequence $\big(\s_\eps\big)_{\eps >0}$ is uniformly bounded in
$L^\infty_{\rm loc}\big(\R_+;H^{-3-\delta}(\Omega)\big)$
for all $\delta > 0$. In particular, up to an extraction, it weakly-$*$ converges to some $\sigma$ in that space.
\end{prop}

\begin{proof}
For notational convenience, set
\begin{equation} \label{def:V-f}
V_\epsilon\,=\,\rho_\epsilon\, u_\epsilon\qquad\mbox{ and }\qquad 
f_\epsilon\, =\, \D \big(\nu (\rho_\epsilon)\, \nabla u_\epsilon\big)\, +\,\D \big(b_\epsilon \otimes b_\epsilon - \rho_\epsilon\, u_\epsilon \otimes u_\epsilon\big)\,.
\end{equation}
Because of the Sobolev embedding $H^{1 + \delta} \subset L^\infty$ (see the note following Proposition \ref{p:embed}), which holds for any $\delta > 0$, we see that $L^1 \subset H^{-1-\delta}$. Hence,
arguing as in the proof of Proposition \ref{PropDebut}, we see that, for any $T>0$ and any arbitrarily small $\de>0$, one has
\begin{equation} \label{ub:f_e}
\big(f_\epsilon\big)_{\epsilon}\, \subset\, L^2_T(H^{-1}) + L^\infty_T(H^{-2-\delta})\, \subset\, L^2_T(H^{-2-\delta})\,.
\end{equation}

Now, because $\rho_0$ is time-independent, we can write the mass and momentum equations as
\begin{equation}\label{Cor1}
\left\{\begin{array}{l}
\epsilon\, \partial_t \sigma_\epsilon\, +\, \D V_\epsilon\, =\, 0 \\[1ex]
\epsilon\, \partial_t V_\epsilon\, + \,\nabla \pi_\epsilon\, +\, V_\epsilon^\perp\, =\, \epsilon\, f_\epsilon\,.
\end{array}
\right.
\end{equation}
Taking the curl of the second equation and computing the difference with the first one leads to 
\begin{equation}\label{Cor3}
\big( \partial_t (\eta_\epsilon - \sigma_\epsilon) \big)_{\epsilon>0}\, =\, \big( \curl (f_\epsilon) \big)_{\epsilon>0}\, \subset\, L^2_T(H^{-3-\delta})\,,
\end{equation}
where we have set $\eta_\epsilon = \curl V_\epsilon$. Notice that $\big(\eta_\eps\big)_\eps\,\subset\,L^\infty_T(H^{-1})$ for all $T>0$ and all $\de>0$, in view of \eqref{ub:u-b} and \eqref{ub:rho_e},
whence the uniform bound $(\sigma_\epsilon)_{\epsilon} \subset L^\infty_T(H^{-3-\delta})$.
In particular, there exists $\sigma \in L^\infty_T(H^{-3-\delta})$ such that $\sigma_\epsilon \stackrel{*}{\wtend}\sigma$ in this space.
\end{proof}

Next, for any $\eps>0$ let us set
$$
s_\epsilon\,:=\, \rho_\epsilon\, -\, \rho_0\,=\,\eps\,\s_\eps\,.
$$
By interpolating between the regularity of $s_\eps$ and the one of $\s_\eps$, we want to show convergence of the densities in better spaces, at the cost of losing the linear convergence speed $O(\epsilon)$,
which we will have to replace by $O(\epsilon^\theta)$, for some $0 < \theta < 1$.

\begin{prop}\label{PropSeps}
Given $0 < \gamma < 1$, there exists $\big(\theta,\beta, k\big)\in\R^3$, satisfying
\begin{equation*}
0 < \beta < \gamma < k < 1\qquad\qquad\mbox{ and }\qquad\qquad 0 < \theta < 1\,,
\end{equation*}
such that the following uniform embeddings,
\begin{equation*}
\left( \epsilon^{-\theta}\, s_\epsilon \right)_{\epsilon > 0} \,\subset\, C^{0, \beta}\big([0,T];H^{-k}(\Omega)\big)\qquad \text{ and } \qquad
\left( \epsilon^{-\theta}\, s_\epsilon\, u_\epsilon \right)_{\epsilon > 0} \,\subset \,L^2\big([0,T];H^{-k-\delta}(\Omega)\big)\,,
\end{equation*}
hold true for any $T>0$ and any arbitrarily small $\delta > 0$. In addition,
\begin{equation*}
\epsilon^{-\theta}\, s_\epsilon\,\tend\, 0 \; \text{ in }\; L^\infty\big([0,T];H^{-k-\delta}_{\rm loc}(\Omega)\big) \qquad \text{ and } \qquad
\epsilon^{-\theta}\, s_\epsilon\, u_\epsilon\, \wtend\, 0 \;\text{ in }\; L^2\big([0,T];H^{-k-\delta}_{\rm loc}(\Omega)\big)
\end{equation*}
for any $T>0$ and any small $\delta > 0$.
\end{prop}

\begin{proof}
The function $s_\epsilon$ solves a transport equation with a second member:
\begin{equation*}
\partial_t s_\epsilon + \D (s_\epsilon u_\epsilon) = -u_\epsilon \cdot \nabla \rho_0.
\end{equation*}
with initial datum $\big(s_{\eps}\big)_{|t = 0} = \epsilon r_{0, \epsilon}$. Because we have assumed that $\rho_0 \in C^2_b$, Sobolev embeddings show that, for any fixed $T>0$, the sequence
$\big(u_\epsilon \cdot \nabla \rho_0\big)_\eps\subset L^2_T(H^1)$ is bounded in every space $L^2_T(L^q)$, for $2 \leq q < +\infty$. In addition, $s_\epsilon = \rho_\epsilon - \rho_0$
is trivially uniformly bounded in $L^\infty(L^\infty)$; therefore, we finally infer, for all fixed $T>0$, the uniform bounds
\begin{equation*}
\big(s_\epsilon\big)_{\epsilon > 0}\, \subset \,L^\infty_T(L^2 \cap L^\infty)\,.
\end{equation*}
Furthermore, writing $\partial_t s_\epsilon = -\D (\rho_\epsilon u_\epsilon)$ and reasoning as in the proof of Lemma \ref{l:g}, we see that
$\big(s_\epsilon\big)_{\epsilon > 0} \subset W^{1, \infty}_T(H^{-1})$ for all $T>0$; after interpolation between Sobolev spaces, we get,
for every $0 \leq \gamma \leq 1$, the embedding
\begin{equation*}
\big(s_\epsilon\big)_{\epsilon > 0}\, \subset\, C^{0, \gamma }_T(H^{-\gamma})\,.
\end{equation*}

On the other hand, from Proposition \ref{PropSeps} we know that $\big(\sigma_\epsilon\big)_\eps\,\subset\, L^\infty_T (H^{-3-\delta})$ for $T>0$ and arbitrarily small $\delta > 0$.
Therefore, for $0 \leq t_1, t_2 \leq T$, and for $0 < \theta < 1$ such that $k = \gamma(1 - \theta) + (3+\delta) \theta$, we have
\begin{align*}
\left\| s_\epsilon (t_2) - s_\epsilon(t_1) \right\|_{H^{-k}}\,& \leq\, \| s_\epsilon (t_2) - s_\epsilon(t_1) \|_{H^{-\gamma}}^{1-\theta}\,\| s_\epsilon (t_2) - s_\epsilon(t_1) \|^\theta_{H^{-3 - \delta}} \\
& \leq\, 2\, \| s_\epsilon \|_{C^{0, \gamma}_T(H^{-\gamma})}^{1 - \theta} \, |t_2 - t_1|^{\gamma(1 - \theta)}\, \epsilon^{\theta}\, \| \sigma_\epsilon \|_{L^\infty_T(H^{-3-\delta})}^\theta\,.
\end{align*}
By setting $\beta = (1 - \theta) \gamma$, we get $\big(\eps^{-\theta}\,s_\eps\big)_\eps\subset C_T^{0,\beta}(H^{-k})$, as claimed. We deduce from the Ascoli-Arzel\`a theorem that
the sequence $\big(\epsilon^{-\theta} \,s_\epsilon\big)_{\epsilon}$ is compact in $L^\infty_T(H^{-k-\delta}_{\rm loc})$, for arbitrarily small $\de>0$,
so that it converges strongly to some $s$ in that space. Finally, we remark that we must have $s = 0$, because $\epsilon^{-\theta}s_\epsilon = \epsilon^{1 - \theta} \sigma_\epsilon \tend 0$
in $\mathcal{D}'$.

Next, Corollary \ref{c:product} gives continuity of the function product in the $H^{-k} \times H^1 \longrightarrow H^{-k-\delta}$ topology, for arbitrarily small $\delta > 0$, whence the uniform bound
\begin{equation*}
\left( \epsilon^{-\theta}\, s_\epsilon\, u_\epsilon \right)_{\epsilon > 0}\, \subset\, L^2_T (H^{-k-\delta})\,,
\end{equation*}
for all fixed $T>0$.
In addition, using the strong convergence $\epsilon^{-\theta}\,s_\epsilon\,\tend\, 0$ in  $L^\infty_T(H^{-k-\delta}_{\rm loc})$, we get the weak convergence
$\epsilon^{-\theta}\,s_\epsilon\,u_\epsilon \,\wtend\, 0$ in $L^2_T (H^{-k-\delta}_{\rm loc})$.
\end{proof}

\section{Convergence} \label{s:convergence}

In this section we complete the proof of Theorems \ref{THQH} and \ref{THNH}, up to the uniqueness part of the former statement (which will be considered in Section \ref{s:limit-system} below).
Subsections \ref{SecMag} and \ref{SecVisc} are common to both the quasi-homogeneous and the fully non-homogeneous case.
In Subsection \ref{SecQH} we take care of the convective term in the quasi-homogeneous case.
Subsections \ref{ss:fully-nh} and \ref{ss:fully-conclusion} are dedicated to the fully non-homogeneous case. Discussion first bears on the convergence of the convective term,
which is the most involved part of the proof; lastly, we handle the Coriolis term. In both parts, we thoroughly exploit the decomposition
$\rho_\epsilon = \rho_0 + \epsilon \sigma_\epsilon$ and the vorticity form of the momentum equation.

\subsection{The magnetic field} \label{SecMag}

In this section, we take care of all the terms containing the magnetic fields $b_\veps$, except for the resistivity term $\nabla^\perp\big(\mu(\rho_\eps)\,\curl b_\eps\big)$.
As in the previous section with the continuity equation, we use the magnetic field equation to trade space regularity against time compactness.

\begin{prop}\label{PropMag}
We have the following strong convergence for the magnetic fields: up to the extraction of a suitable subsequence, for any $0 < s < 1$ we have
\begin{equation*}
b_\epsilon\, \tend_{\epsilon \rightarrow 0^+}\,b \qquad \text{ in }\quad L^2_{\rm loc}\big(\R_+;H^s_{\rm loc}(\Omega)\big)\,.
\end{equation*}
In particular, we deduce the convergence of all bilinear terms involving the magnetic field:
\begin{equation*}
\D \big(b_\epsilon \otimes b_\epsilon\big)\, \tend\, \D \big(b \otimes b\big) \qquad \text{ in }\quad \mathcal{D}'\big(\R_+\times \Omega\big)\,,
\end{equation*}
and analogous convergence properties hold for $\div\big(u_\epsilon \otimes b_\epsilon\big)$ and $\div\big(b_\epsilon \otimes u_\epsilon\big)$.
\end{prop}

\begin{proof}
Let $T > 0$ be a fixed positive time. Recall that we have $(b_\epsilon)_{\epsilon} \subset L^\infty(L^2)$, as well as $(u_\epsilon)_{\epsilon}, (b_\epsilon)_{\epsilon} \subset L^2_T(H^1)$.
The magnetic field equation reads
\begin{equation*}
\partial_t b_\epsilon\, =\, \D\big(b_\epsilon \otimes u_\epsilon\, -\, u_\epsilon \otimes b_\epsilon\big)\, +\, \nabla^\perp \big( \mu(\rho_\epsilon)\, \curl (b_\epsilon) \big)\,.
\end{equation*}
Gagliardo-Nirenberg inequality of Lemma \ref{GN}
gives $\big(b_\epsilon\big)_{\epsilon} \subset L^4_T(L^4)$ and $\big(u_\epsilon\big)_{\epsilon} \subset L^2_T(L^4)$.
Hence $\big(u_\epsilon \otimes b_\epsilon\big)_{\epsilon} \subset L^{4/3}_T(L^2)$ and $\big(\partial_t b_\epsilon\big)_{\epsilon}\subset L^{4/3}_T (H^{-1})$.
As a result, we get the H\"older bound $\big(b_\epsilon\big)_{\epsilon} \subset C^{0, 3/4}_T(H^{-1})$, and the Ascoli-Arzel\`a theorem gives compactness of
$\big(b_\epsilon\big)_{\epsilon}$ in e.g. $C^0_T(H^{-1 - \delta}_{\rm loc})$, for any small $\delta > 0$. Using the uniform bound in $L^2_T(H^1)$ and Sobolev interpolation gives strong convergence
of the magnetic fields $\big(b_\epsilon\big)_\eps$ in $L^2_T(H^s_{\rm loc})$, for any $0\leq s<1$.

As a consequence, we gather strong convergence of the tensor products $\big(b_\epsilon \otimes b_\epsilon\big)_\eps$ in, say, $L^1_T(L^1_{\rm loc})$, which implies in particular that
$\D \big(b_\epsilon \otimes b_\epsilon\big)\, \tend\, \D \big(b \otimes b\big)$ in $\mathcal{D}'\big(\,]0, T[ \,\times \Omega\big)$.
In an analogous way, we can achieve weak convergence of the mixed tensor products $b_\epsilon \otimes u_\epsilon$ and $u_\epsilon \otimes b_\epsilon$:
for the sake of brevity, we omit the details here.
\end{proof}

Note that this proposition does not complete the study of the magnetic field equation: the convergence of the resistivity term $\nabla^\perp \big( \mu(\rho_\epsilon)\, \curl (b_\epsilon) \big)$ still remains.
This is the goal of the next paragraph.

\subsection{The viscosity and resistivity terms}\label{SecVisc}

In this section, we take care of the convergence of the viscosity and resistivity terms, namely $\D\big(\nu(\rho_\epsilon) \nabla u_\epsilon\big)$ and
$\nabla^\perp \big( \mu(\rho_\epsilon)\, \curl (b_\epsilon) \big)$ respectively.
Remember that we have taken $\nu$ and $\mu$ to be continuous on $\mathbb{R}_+$.

\begin{prop}\label{PropVisc}
The following convergence of the viscosity and resistivity terms holds true, in the sense of $\mathcal{D}'(\mathbb{R}_+ \times \Omega)$: 
\begin{equation*}
\D\big(\nu(\rho_\epsilon) \nabla u_\epsilon\big)\, \tend\, \D\big(\nu(\rho_0) \nabla u\big)\qquad\mbox{ and }\qquad
\nabla^\perp \big( \mu(\rho_\epsilon)\, \curl (b_\epsilon) \big)\, \tend\, \nabla^\perp \big( \mu(\rho_0)\, \curl (b) \big)\,,
\end{equation*}
where $\rho_0$ is either $1$ (in the quasi-homogeneous case) or the truly variable profile satisfying the assumptions of Theorem \ref{THNH} (in the fully-non-homogeneous case).
\end{prop}

\begin{proof}
We only prove the convergence of the viscosity term, the one of the resistivity term being, in all that matters, identical.

In the case where the density is quasi-homogeneous, we already have strong convergence
$\rho_\epsilon \tend 1$ in $L^\infty(L^\infty)$, because $\big(r_\epsilon\big)_{\epsilon}$ is bounded in $L^\infty(L^2 \cap L^\infty)$. This makes the viscosity term easy to handle:
let $T > 0$ and $\phi \in \mathcal{D}\big([0, T[ \, \times \Omega ; \mathbb{R}^2\big)$ be a test function, then
\begin{align*}
&\int_0^{T} \int_\Omega \nu (\rho_\epsilon) \nabla u_\epsilon : \nabla \phi\, \dx \dt - \int_0^{T} \int_\Omega \nu (1) \nabla u : \nabla \phi\, \dx \dt \\
&\qquad\qquad =\int_0^{T} \int_\Omega \big( \nu (\rho_\epsilon) - \nu(1) \big) \nabla u_\epsilon : \nabla \phi\, \dx \dt +
\int_0^{T} \int_\Omega \nu(1) \big( \nabla u_\epsilon  - \nabla u \big) : \nabla \phi\, \dx \dt\,.
\end{align*}
The second integral has limit zero, because of the weak convergence of the $u_\epsilon$ in $L^2_T(H^1)$. As for the first integral,
uniform convergence of the $\rho_\epsilon$ and continuity of $\nu$ gives
\begin{align*}
\left|   \int_0^{T} \int_\Omega \big( \nu (\rho_\epsilon) - \nu(1) \big) \nabla u_\epsilon : \nabla \phi \right|&\leq\left\| \nu(\rho_\epsilon) - \nu(1)\right\|_{L^\infty(L^\infty)}\,
\| \nabla u_\epsilon \|_{L^2_T(L^2)}\, \| \nabla \phi \|_{L^2_T(L^2)}\,
\tend_{\epsilon \rightarrow 0^+}\, 0\,.
\end{align*}

Obviously, this does not work as well in the fully non-homogeneous case. Hence, we will need to use the strong convergence result of Proposition \ref{PropStrCnv}: namely,
in the limit $\eps\ra0^+$, one has
$\rho_\epsilon \tend \rho_0$ in $L^2_{\rm loc}\big(\R_+\times\Omega\big)$ and almost everywhere in $\mathbb{R}_+ \times \Omega$, where we have used also the information coming from
the second item of Proposition \ref{PropDebut}.

The uniform bounds $\big(\rho_\epsilon\big)_{\epsilon > 0} \subset L^\infty(L^\infty)$ and the continuity of $\nu$ show that the $\nu(\rho_\epsilon)$ are also uniformly bounded in
$L^\infty(L^\infty)$: namely, there is a constant $\nu^* > 0$ such that $ \nu(\rho_\epsilon) \leq \nu^*$ for all $\epsilon > 0$. The dominated convergence theorem then gives
strong convergence of the viscosities: 
\begin{equation} \label{conv:strong_nu}
\nu(\rho_\epsilon)\, \tend_{\epsilon \rightarrow 0^+}\, \nu(\rho_0) \qquad\qquad \text{ in }\qquad L^2_{\rm loc}\big(\R_+\times\Omega\big)\,.
\end{equation}
Now, let $T>0$ and $K \subset\Omega$ be a compact set. For any $\phi \in \mathcal{D}\big([0, T[\, \times K\big)$, we can estimate
\begin{align*}
&\left| \int_0^T \int_\Omega \nu(\rho_\epsilon) \nabla u_\epsilon \cdot \nabla \phi\, \text{d}x \text{d}t -\int_0^T \int_\Omega \nu(\rho_0) \nabla u \cdot \nabla \phi\,\text{d}x \text{d}t \right| \\
&\qquad \leq \left|\int_0^T \int_\Omega \nu(\rho_0)\, \big(\nabla u_\epsilon - \nabla u\big) \cdot \nabla \phi \, \text{d}x \text{d}t \right|\,+\,
\left|  \int_0^T \int_\Omega \big( \nu(\rho_\epsilon) - \nu(\rho_0) \big)\, \nabla u_\epsilon \cdot \nabla \phi \, \text{d}x \text{d}t  \right|\,.
\end{align*}
The first intergral obviously tends to zero as $\epsilon \rightarrow 0^+$, because of weak convergence $\nabla u_\epsilon \wtend \nabla u$ in $L^2_T(L^2)$.
The second integral, instead, can be bounded as follows:
\begin{equation*}
\left|  \int_0^T \int_\Omega \big( \nu(\rho_\epsilon) - \nu(\rho_0) \big)\, \nabla u_\epsilon \cdot \nabla \phi \, \text{d}x \text{d}t  \right|\,\leq\,
\| \nabla \phi \|_{L^\infty(L^\infty)}\, \| \nu(\rho_\epsilon) - \nu(\rho_0) \|_{L^2_T(L^2(K))}\, \| \nabla u_\epsilon \|_{L^2(L^2)}\,,
\end{equation*}
where the quantity on the right-hand side goes to $0$ in view of \eqref{conv:strong_nu}.

The proposition is now proved.
\end{proof}

\subsection{The convective term: the quasi-homogeneous case}\label{SecQH}

In the slightly non-homogeneous case, the convective term $\D (\rho_\epsilon u_\epsilon \otimes u_\epsilon)$ is the last one we have to study.
The argument is in three steps. First of all, we reduce the problem to the study of $\D(u_\epsilon \otimes u_\epsilon)$, taking advantage of the approximation
$\rho_\epsilon \approx 1$. Then, we use the uniform $H^1$ regularity to find an approximation of $u_\epsilon$ by smooth functions, which we will need for the last step,
a compensated compactness argument.

\begin{prop}
For all divergence-free $\phi \in \mathcal{D}\big(\mathbb{R}_+ \times \Omega ; \mathbb{R}^2\big)$, one has
\begin{equation*}
\int_0^{+ \infty} \int_\Omega \rho_\epsilon\, u_\epsilon \otimes u_\epsilon : \nabla \phi\, \dx\dt\, \tend_{\epsilon \rightarrow 0^+}\, \int_0^{+ \infty} \int_\Omega u \otimes u : \nabla \phi\, \dx \dt\,.
\end{equation*}
\end{prop}

\begin{proof}

Let $T > 0$ be a fixed positive time. To start the proof, recall that we can write $\rho_\epsilon = 1 + \epsilon r_\epsilon$, with $(r_\epsilon)_{\epsilon} \subset L^\infty (L^2 \cap L^\infty)$.
Then, by virtue of the the $L^2_T(L^2)$ uniform boundedness of $\big(u_\eps\big)_\eps$, for all divergence-free $\phi \in \mathcal{D}([0, T[ \times \Omega ; \mathbb{R}^2)$ we have
\begin{equation}\label{App1}
\left| \int_0^T \int_\Omega \big(  \rho_\epsilon u_\epsilon \otimes u_\epsilon - u_\epsilon \otimes u_\epsilon \big) : \nabla \phi\, \dx \dt  \right|\, \tend_{\epsilon \rightarrow 0^+} \,0\,.
\end{equation}

Next, we seek a uniform approximation of $u_\epsilon$ by a smooth function (that is, smooth in the space variable). Let $S_j$ be the low-frequency cut-off operator from the
Littlewood-Paley decomposition given by \eqref{eq:S_j} in the appendix. Then, using the uniform bound $(u_\epsilon)_{\epsilon} \subset L^2_T(H^1)$ and the characterisation \eqref{eq:LP-Sob}
of Sobolev spaces, it immediately follows that
\begin{equation*}
\big\| \big(I - S_j\big) u_\epsilon  \big\|_{L^2_T(L^2)}\, \leq\, C\, 2^{-j}\,,
\end{equation*}
for some constant $C > 0$ possibly depending on $T$, but independent of both $j$ and $\epsilon$.
Using the previous bound, we can thus estimate
\begin{multline}\label{Unif}
\left| \int_0^T \int_\Omega \bigg\{  u_\epsilon \otimes u_\epsilon - S_j u_\epsilon \otimes S_j u_\epsilon  \bigg\} : \nabla \phi\, \dx \dt \right| \\
 \leq\, \| \nabla \phi \|_{L^\infty(L^\infty)}  \bigg( \big\| (I - S_j)u_\epsilon \otimes u_\epsilon \big\|_{L^1_T(L^1)} + \big\| S_j u_\epsilon \otimes (I - S_j)u_\epsilon  \big\|_{L^1_T(L^1)}\bigg)
\, \leq \,C\, 2^{-j}\,.
\end{multline}

Therefore, the problem is now to prove convergence of the term
$$
\int_0^T \int_\Omega S_j u_\epsilon \otimes S_j u_\epsilon: \nabla \phi\, \dx \dt\,=\,-\int_0^T \int_\Omega \div\big(S_j u_\epsilon \otimes S_j u_\epsilon\big)\cdot\phi\, \dx \dt
$$
when $\eps\ra0^+$, for any fixed $j\geq-1$. Notice that the integration by parts is justified, since all the quantities are now smooth with respect to the space variable.
For convenience purposes, we henceforth note $u_{\epsilon, j} = S_j u_\epsilon$. Because the operator $S_j$ is a Fourier-multiplier, it commutes with all the partial derivatives: in particular,
$\D (u_{\epsilon, j}) = 0$ and therefore, after denoting $\omega_\eps\,:=\,\curl u_\eps$ and $\omega_{\eps,j}\,:=\,S_j\omega_\eps\,=\,\curl u_{\eps,j}$,  we deduce
\begin{equation}\label{QH1}
\D\big(u_{\epsilon, j} \otimes u_{\epsilon, j}\big)\, =\, \frac{1}{2} \nabla \big| u_{\epsilon, j} \big|^2\, +\, \omega_{\epsilon, j}\, u_{\epsilon, j}^\perp\,.
\end{equation}
The first term in the right-hand side disappears when tested against a divergence-free function, hence we deduce
\begin{equation*}
\int_0^T \int_\Omega u_{\epsilon, j} \otimes u_{\epsilon, j} : \nabla \phi\, \dx \dt\, =\, -\int_0^T \int_\Omega \omega_{\epsilon, j}\, u_{\epsilon, j}^\perp \cdot \phi\, \dx \dt\,.
\end{equation*}
As for the vorticity term, we resort to the reformulation \eqref{Cor1} of the momentum equation, which can be rewritten, in the quasi-homogeneous case, in the following way:
\begin{equation*}
\epsilon\, \partial_t V_\epsilon + \nabla \pi_\epsilon + \frac{1}{2}\,\epsilon\, \nabla |b_\epsilon|^2 + u_\epsilon^\perp\, =\, \epsilon\, \big(f_\epsilon - r_\epsilon u_\epsilon^\perp \big)\,,
\end{equation*}
where $V_\eps$ and $f_\epsilon$ have been defined in \eqref{def:V-f}. Then, applying the operator $S_j$ to the previous equation and taking the $\curl$ gives
\begin{equation*}
\partial_t \eta_{\epsilon, j} \,=\, \curl \bigg(f_{\epsilon, j} - S_j\left(r_\epsilon\, u_\epsilon^\perp\right) \bigg)\,,
\end{equation*}
where we have set $\eta_{\epsilon, j}\,:=\, S_j \curl (V_\epsilon)$ and $f_{\epsilon, j} = S_j f_\epsilon$.
From \eqref{ub:f_e}, we known that $\big(f_\epsilon\big)_\eps$ is bounded in $L^2_T(H^{-2 - \delta})$ for any $\de>0$, and so, for any fixed $j$, the sequence $\big( f_{\epsilon, j} \big)_\epsilon$ is bounded in every $L^2_T(H^m)$, with $m \in \mathbb{R}$. Likewise, because  the sequence$(r_\epsilon u_\epsilon)_\epsilon$ is bounded in $L^2_T (L^2)$, we also  see that $\big( r_\epsilon u_\epsilon^\perp \big)_\epsilon \subset L^2_T(H^m)$ for every $m \in \mathbb{R}$. We deduce a uniform bound for the $\eta_{\epsilon, j}$: for $m\in \mathbb{R}$,
\begin{equation}
\big( \eta_{\epsilon, j} \big)_\epsilon \subset C^{0, 1/2}_T(H^m)\,.
\end{equation}
Hence, the Ascoli-Arzel\`a theorem provides strong convergence (up to the extraction of a subsequence) to some $\eta_j \in L^\infty_{\rm loc}(H^m_{\rm loc})$: more precisely,
for all fixed $T>0$ and $m\in \mathbb{R}$,
\begin{equation*}
\forall j \geq 1, \quad \qquad \eta_{\epsilon, j}\, \tend_{\epsilon \rightarrow 0^+}\, \eta_j \qquad \text{ in }\;L^\infty_T(H^m_{\rm loc})\,.
\end{equation*}
But since we already know that $V_\epsilon\, \wtend\, u$ in $L^2_T(L^2)$, it follows that $\eta_j = \omega_j$. Thanks to the previous strong convergence property,
for fixed $j \geq 1$ and for every $m \in \mathbb{R}$, we get also the strong convergence of $\big(\omega_{\eps,j}\big)$: namely,
\begin{equation*}
\omega_{\epsilon, j}\,=\, \eta_{\epsilon, j}\, -\, \epsilon\, S_j\curl\left(r_\epsilon u_\epsilon\right)\, \tend_{\epsilon \rightarrow 0^+}\,\omega_j \qquad \text{ in }\quad
L^\infty_T(H^m_{\rm loc})\,.
\end{equation*}
Combining this information with the weak convergence $u_\epsilon\, \wtend\, u$  in  $L^2_T(H^1)$ finally yields
\begin{equation*}
\omega_{\epsilon, j}\, u_{\epsilon, j}^\perp\, \tend_{\epsilon \rightarrow 0^+}\, \omega_j\, u_j^\perp \qquad \text{ in }\quad  \mathcal{D}'\big([0, T[\, \times \Omega\big)\,.
\end{equation*}

Thus, we have proved that, for any divergence-free test function $\phi\in\mc D\big([0, T[\, \times \Omega\big)$, we have
\begin{equation*}
\big\langle \D \big(u_{\epsilon, j} \otimes u_{\epsilon, j}\big), \phi \big\rangle_{\mathcal{D}' \times \mathcal{D}}\, =\,
\langle \omega_{\epsilon, j}\, u_{\epsilon, j}^\perp, \phi \rangle_{\mathcal{D}' \times \mathcal{D}}\, \tend_{\epsilon \rightarrow 0^+} \,
\langle \omega_j\, u_j^\perp, \phi \rangle_{\mathcal{D}' \times \mathcal{D}}\, =\, \big\langle \D\big(u_{j} \otimes u_{j}\big), \phi \big\rangle_{\mathcal{D}' \times \mathcal{D}}\,.
\end{equation*}
Now, keeping in mind \eqref{App1}, using the uniform approximation property \eqref{Unif} yields the claimed convergence result.
\end{proof}

\subsection{The convective term: the fully non-homogeneous case} \label{ss:fully-nh}

The main ideas for handling the convective term in the fully non-homogeneous case are very similar to those used for the quasi-homogeneous case, although many complications occur.
Because $\rho_0$ is not constant, the equivalent of decomposition \eqref{QH1} will instead be (omitting for the time being the regularisation argument)
\begin{equation*}
\D \big(\rho_\epsilon u_\epsilon \otimes u_\epsilon\big)\, \approx\, \D \big(\rho_0 u_\epsilon \otimes u_\epsilon\big)\, =\,
\frac{1}{2} \rho_0\, \nabla |u_\epsilon|^2\, +\, \rho_0\, \omega_\epsilon\, u_\epsilon^\perp\, +\, (u_\epsilon \cdot \nabla \rho_0)\, u_\epsilon\,.
\end{equation*}
To simplify those terms, we can no longer rely on the fact that we use divergence-free test functions:
$\langle \rho_0 \nabla |u_\epsilon|^2, \phi \rangle \neq 0$ even when $\D(\phi) = 0$. However, any term of the form $\rho_0 \nabla \Lambda_\epsilon$ or 
$\Lambda_\epsilon \nabla \rho_0$ will give rise to a term of the form $\rho_0 \nabla \Gamma$ in the limit (see below), which can be considered as a ``pressure'' term associated to the constraint
$\div\big(\rho_0\,u\big)=0$.

In the end, we can prove the next statement.
\begin{prop}\label{PropFin}
There is a distribution $\Gamma$ (of order at most one) such that, for all $\phi \in \mathcal{D}\big(\mathbb{R}_+ \times \Omega ; \mathbb{R}^2\big)$ such that $\D(\phi) = 0$, one has
\begin{equation*}
\int_0^{+\infty} \int_\Omega \rho_\epsilon\, u_\epsilon \otimes u_\epsilon : \nabla \phi\, \dx \dt\, \tend_{\epsilon \rightarrow 0^+}\, \big\langle \rho_0 \nabla \Gamma, \phi \big\rangle_{\mc D'\times\mc D}\,.
\end{equation*}
\end{prop}

The rest of this subsection is devoted to the proof of the previous proposition. Our argument, mainly borrowed from \cite{FG}, consists of several steps.
We start by taking a positive time $T > 0$, a compact set $K\subset\Omega$ and a divergence-free test function $\phi\in\mc D\big([0,T[\,\times K;\R^2\big)$, which we keep fixed throughout all the proof.

\medbreak
\textbf{Step 1: approximation of the densities.} First of all, we justify the approximation $\rho_\epsilon u_\epsilon \otimes u_\epsilon \approx \rho_0 u_\epsilon \otimes u_\epsilon$. 
Note that, because we have accounted for the presence of vacuum, the best uniform bound we have for the velocity field is $\big(u_\epsilon\big)_{\epsilon}\subset L^2_T(H^1)$.
This means that, when estimating the difference $(\rho_\epsilon - \rho_0) u_\epsilon \otimes u_\epsilon$, we must use strong convergence of the densities
\emph{uniformly with respect to time}. In particular, we cannot benefit of the convergence properties proved in Proposition \ref{PropVisc}, which only provides strong convergence
$\rho_\epsilon \tend \rho_0$ in the spaces $L^p_T(L^q_{\rm loc})$ for $1 \leq p, q < +\infty$ (thanks to the uniform bound $0 \leq \rho_\epsilon \leq \rho^*$ and dominated convergence).
Instead, from Proposition \ref{PropSeps} we know that $\rho_\epsilon = \rho_0 + s_\epsilon$, with $s_\epsilon \tend 0$ in $C^{0, \beta}_T(H^{-k}_{\rm loc})$, where $\beta, k \in ]0, 1[$
are as in that proposition (relatively to some $\gamma \in ]0, 1[\,$). 

Firstly, using Corollary \ref{c:product}, we see that, for every $\delta > 0$, the function product is continuous in the $H^1 \times H^1 \longrightarrow H^{1-\delta}$ topology.
This implies that the tensor product $(u_\epsilon \otimes u_\epsilon)_\epsilon$ is bounded in the space $L^1_T(H^{1 - \delta})$, for every $\delta > 0$.
Next, using Lemma \ref{Para}, we get continuity of the product in the $H^{-k} \times H^{1 - \delta} \longrightarrow H^{-k - \delta}$ topology provided that $\delta > 0$ be small enough
(i.e. small enough for $1 - k - \delta$ to be positive). We therefore gather that
\begin{equation*}
\big\| s_\epsilon u_\epsilon \otimes u_\epsilon  \big\|_{L^1_T(H^{-k - \delta})} \leq \big\| s_\epsilon \big\|_{L^\infty_T(H^{-k})} \big\| u_\epsilon \otimes u_\epsilon \big\|_{L^1_T(H^{1 - \delta})}
\tend_{\epsilon \rightarrow 0^+}  0
\end{equation*}
and are reconducted to taking the limit $\epsilon \rightarrow 0^+$ in the integral
$$
\int_0^{T} \int_\Omega \rho_0\, u_\epsilon \otimes u_\epsilon : \nabla \phi\, \dx \dt\,.
$$

At this point, the idea is to resort once again to a compensated compactness argument: for this, we need first to smooth out the velocity fields.

\medbreak
\textbf{Step 2: regularisation.}
We use the same regularisation procedure than in the quasi-homogeneous case: here $S_j$ still is the Littlewood-Paley operator defined by \eqref{eq:S_j} below.
We continue to note $S_j g_\epsilon = g_{\epsilon, j}$ for any sequence of functions $(g_\epsilon)_{\epsilon>0}$, whenever we feel it make things more legible.

We set as above $\eta_\epsilon = \curl V_\epsilon = \curl\big(\rho_\epsilon u_\epsilon\big)$. We now state a simple approximation lemma: its proof is simple, hence omitted. It is enough to recall that
$\big(\eta_\eps)_\eps$ is bounded in $L^\infty_T(H^{-1})$ and that $\big(\sigma_\eps\big)_\eps$
is bounded in every $L^\infty_T(H^{-3-\delta})$, with $\delta > 0$.
\begin{lemma}
The following uniform properties hold, in the limit for $j \rightarrow + \infty$:
\begin{enumerate}[1)]
\item for all $s > 3$, we have $\sup_{\epsilon > 0} \left\| \sigma_\epsilon - S_j \sigma_\epsilon \right\|_{L^\infty_T (H^{-s})}\, \tend\, 0$;
\item for all $s > 1$, we have $\sup_{\epsilon > 0} \left\| \eta_{\epsilon} - S_j \eta_{\epsilon} \right\|_{L^\infty_T(H^{-s})}\, \tend\, 0$.
\end{enumerate}
\end{lemma}

The first new problem we face, after introducing $S_j$, is that $S_j(\rho_\epsilon u_\epsilon) \neq \rho_0 S_j u_\epsilon$, because $\rho_0$ is no longer constant.
With the same notation introduced in Section \ref{SecSig}, we write
\begin{equation} \label{eq:dec_rho-u}
S_j(\rho_\epsilon u_\epsilon) =  S_j(\rho_0 u_\epsilon) + \epsilon^\theta S_j(\epsilon^{-\theta} s_\epsilon u_\epsilon) =
\rho_0 u_{\epsilon, j} + \big[ S_j, \rho_0 \big]u_\epsilon + \epsilon^\theta S_j(\epsilon^{-\theta} s_\epsilon u_\epsilon).
\end{equation}
In the above, $\big[ S_j, \rho_0 \big]$ is the commutator between $S_j$ and the multiplication by $\rho_0$ operator. By Proposition \ref{PropSeps}, we already know that
$\Big(S_j\big(\epsilon^{-\theta} s_\epsilon u_\epsilon\big)\Big)_{\eps} \subset L^2_T(H^s)$ for any $s \geq 0$. To deal with the second summand, we use Lemma \ref{lemma:Comm} on
$\big[ S_j, \rho_0 \big]u_\epsilon$. On the one hand, we have
\begin{equation*}
\left\| \big[ S_j, \rho_0 \big]u_\epsilon \right\|_{L^2_T(L^2)}\, \leq\, \frac{C}{2^j}\, \| \nabla \rho_0 \|_{L^\infty}\, \| u_\epsilon \|_{L^2_T(L^2)}\,\leq\,C\,2^{-j}\,.
\end{equation*}
On the other hand, by differentiating the commutator, we get, for $i \in \{1, 2\}$,
\begin{equation*}
\partial_i \big[ S_j, \rho_0 \big]u_\epsilon = \big[ S_j, \partial_i \rho_0 \big]u_\epsilon + \big[ S_j, \rho_0 \big] \partial_i u_\epsilon\,.
\end{equation*}
Therefore, from Lemma \ref{lemma:Comm} we infer
\begin{equation*}
\left\|   \partial_i \big[ S_j, \rho_0 \big]u_\epsilon  \right\|_{L^2_T(L^2)} \leq
\frac{C}{2^j} \bigg\{  \| \nabla^2 \rho_0 \|_{L^\infty} \| u_\epsilon \|_{L^2_T(L^2)} + \| \nabla \rho_0 \|_{L^\infty} \| \nabla u_\epsilon \|_{L^2_T(L^2)}\bigg\}\,\leq\,C\,2^{-j}\,.
\end{equation*}

Thus, from \eqref{eq:dec_rho-u}, we have obtained the following decomposition of $S_j(\rho_\epsilon u_\epsilon)$:
\begin{equation}\label{Vort}
S_j\big(\rho_\epsilon u_\epsilon\big)\, =\, \rho_0\, u_{\epsilon, j}\, +\, \epsilon^{\theta} \,\zeta_{\epsilon, j}\, + \,h_{\epsilon, j}\,,
\end{equation}
where we have defined $\zeta_{\epsilon, j}\,:=\, S_j\big(\epsilon^{-\theta} s_\epsilon u_\epsilon\big)$ and $h_{\epsilon, j}\, :=\, \big[ S_j, \rho_0 \big]u_\epsilon$. Notice that
\begin{equation}\label{H1Bound}
\forall\,j\geq-1\,,\;\forall\,s\geq0\,,\quad\big(\zeta_{\epsilon, j}\big)_{\eps}\,\subset\, L^2_T(H^s)\qquad\quad\mbox{ and }\qquad\quad
\sup_{\eps>0}\left\| h_{\epsilon, j}\right \|_{L^2_T(H^1)}\,\leq\,C\,2^{-j}\,.
\end{equation}

We make a couple of remarks before going onwards. Firstly, we have seen in Proposition \ref{PropSeps} that $\left( \epsilon^{-\theta} s_\epsilon u_\epsilon \right)_\eps$ is bounded in
$L^2_T(H^{-k-\delta})$, where $0<k<1$ has been fixed in that proposition and $\de>0$ is arbitrarily small. Therefore, we can further write
\begin{equation}\label{EtaDec}
\eta_{\epsilon, j}\, =\, S_j \curl \left[  \rho_0 u_\epsilon\, +\, \epsilon^\theta(\epsilon^{-\theta} s_\epsilon u_\epsilon) \right]\,=\,\eta_{\epsilon, j}^{(1)}\,+\,\epsilon^\theta\,\eta_{\epsilon, j}^{(2)}\,,
\end{equation}
with the uniform bounds (with respect to $\epsilon$)
\begin{equation*}
 \left\| \eta_{\epsilon, j}^{(1)} \right\|_{L^2_T(L^2)}\,\leq\,C_1 \qquad\qquad \text{ and } \qquad\qquad \left\| \eta_{\epsilon, j}^{(2)} \right\|_{L^2_T(H^s)}\, \leq\, C(s, j)\,,
\end{equation*}
for any given $s \geq 0$. Note that the constant $C_1$ does not depend on $\epsilon$ nor on $j$.

Finally, exactly as in the quasi-homogeneous case (keep in mind estimate \eqref{Unif} above), thanks to the uniform approximation properties of $S_j$,
we note that it is enough to prove the convergence of the term
\begin{equation*}
\int_0^T \int_\Omega \rho_0\,u_{\epsilon, j} \otimes u_{\epsilon, j}\,:\, \nabla \phi \, \dx \dt\,=\,
-\int_0^T \int_\Omega \div\big(\rho_0\,u_{\epsilon, j} \otimes u_{\epsilon, j}\big)\cdot \phi\,\dx\dt\,,
\end{equation*}
where the integration by parts is now well-justified, since all the quantities in the integral are smooth in the space variable.

\medbreak
\textbf{Step 3: reformulation.} In view of the previous discussion, we are left with $\D(\rho_0 u_{\epsilon, j} \otimes u_{\epsilon, j})$. Since all functions are smooth, we can write
\begin{equation*}
\D\big(\rho_0 u_{\epsilon, j} \otimes u_{\epsilon, j}\big)\, =\, \big(u_{\epsilon, j} \cdot \nabla \rho_0 \big)\, u_{\epsilon, j}\, +\, \rho_0\, \omega_{\epsilon, j}\, u_{\epsilon, j}^\perp\, +\,
\frac{1}{2} \rho_0 \nabla |u_{\epsilon, j}|^2\,,
\end{equation*}
with $\omega_{\epsilon, j} = S_j \curl(u_\epsilon)$. We remark that the last term in the righthand side of this equation contributes $\rho_0 \nabla \Gamma$ in the limit,
for some distribution $\Gamma$ of order at most one. In the same way, any term of the form $\langle \Lambda_{\epsilon, j} \nabla \rho_0, \phi \rangle$
has a limit of the same form $\langle \rho_0 \nabla \Gamma, \phi \rangle$: since $\div\phi=0$, an integration by parts gives
\begin{equation*}
\int_0^T \int_\Omega \Lambda_{\epsilon, j}\, \nabla \rho_0 \cdot \phi\, \dx \dt\, =\, \int_0^T \int_\Omega  \Lambda_{\epsilon, j}\, \D (\rho_0 \phi)\, \dx \dt\, =\,
- \int_0^T \int_\Omega \rho_0\, \nabla \Lambda_{\epsilon, j} \cdot \phi\, \dx \dt\,.
\end{equation*}
Since all terms of the form $\rho_0 \nabla \Lambda_{\epsilon, j}^{(1)} + \Lambda_{\epsilon, j}^{(2)} \nabla \rho_0$ can be treated in this way, we will generically note any of them by $\Gamma_{\epsilon, j}$.
Likewise, we note $R_{\epsilon, j}$ any remainder term, that is any term such that
\begin{equation*}
\lim_{j\ra+\infty}\,\limsup_{\epsilon\ra0^+} \left|  \int_0^T \int_\Omega R_{\epsilon, j} \cdot \phi\, \dx \,\dt \right|\,=\, 0\,.
\end{equation*}
With that notation, we can write
\begin{equation*}
\D\big(\rho_0\, u_{\epsilon, j} \otimes u_{\epsilon, j}\big)\, =\, \big(u_{\epsilon, j} \cdot \nabla \rho_0 \big)\, u_{\epsilon, j}\, +\, \rho_0\,\omega_{\eps, j}\, u_{\eps, j}^\perp\,+\,\Gamma_{\eps, j}\,.
\end{equation*}
We are now going to deal with the vorticity term, namely the second term in the right-hand side of the previous equation.

\medbreak
\textbf{Step 4: the vorticity term.} By use of \eqref{Vort}, we get
\begin{align*}
\eta_{\epsilon, j}\, & = \,\curl\big(\rho_0\, u_{\epsilon, j}\big)\, +\, \epsilon^\theta\, \curl\zeta_{\epsilon, j} \,+\, \curl h_{\epsilon, j} \\
& = \rho_0\, \omega_{\epsilon, j}\, +\, \nabla^\perp \rho_0 \cdot u_{\epsilon, j}\, +\, \epsilon^\theta\, \curl\zeta_{\epsilon, j}\, +\, \curl h_{\epsilon, j}\,.
\end{align*}
Therefore, by virtue of \eqref{H1Bound}, we gather that
$\rho_0\,\omega_{\epsilon, j}\, u_{\eps, j}^\perp\, =\,\eta_{\eps, j}\, u_{\epsilon, j}^\perp\, -\,\big(u_{\eps, j} \cdot \nabla^\perp \rho_0 \big)\, u_{\eps, j}^\perp\, +\, R_{\epsilon, j}$,
which in turn gives
\begin{equation*}
\D\big(\rho_0\, u_{\epsilon, j} \otimes u_{\epsilon, j}\big)\,=\,\eta_{\epsilon, j}\,u_{\epsilon, j}^\perp\,+\,\big(u_{\eps, j}\cdot\nabla\rho_0\big)\,u_{\eps, j}\,-\,
\big(  u_{\epsilon, j} \cdot \nabla^\perp \rho_0 \big)\, u_{\epsilon, j}^\perp\, +\, \Gamma_{\epsilon, j}\, +\, R_{\epsilon, j}\,.
\end{equation*}

\textbf{Step 5: a geometric property.} Now we focus on the term
$$
X_{\epsilon, j}\,:=\,\big(u_{\epsilon, j} \cdot \nabla \rho_0 \big)\, u_{\epsilon, j}\, -\, \big(  u_{\epsilon, j} \cdot \nabla^\perp \rho_0 \big)\,u_{\epsilon, j}^\perp\,.
$$
The main idea to treat this term is to decompose $u_{\epsilon, j}(t, x)$ in the orthonormal basis of $\mathbb{R}^2$ given by
$\left(  \frac{\nabla \rho_0(x)}{|\nabla \rho_0(x)|}, \frac{\nabla^\perp \rho_0(x)}{|\nabla \rho_0(x)|} \right)$. However, to avoid complications, we have first to deal with those
$x\in\Omega$ for which $|\nabla \rho_0(x)|$ is small. More precisely, let $B \in \mathcal{D}(\mathbb{R}^2)$ be such that
\begin{equation*}
0 \leq B \leq 1 \qquad \text{ and } \qquad
\begin{cases}
B(y) = 1 \quad \text{for } |y| \leq 1 \\[1ex]
B(y) = 0 \quad \text{for } |y| \geq 2\,,
\end{cases}
\end{equation*}
and let $B_j(x) = B \big(  2^{j/2} \nabla \rho_0(x) \big)$. The function $B_j$ is so chosen that $|\nabla \rho_0| \geq 2^{-j/2}$ on $\Supp(1 - B_j)$.
Recall that $\Supp\phi\subset[0,T]\times K$, where $\phi$ is the divergence-free test function we have fixed at the beginning of the argument.
Then, for any $2 < q < +\infty$, H\"older's inequality with $1 = (2/q) + (q-2)/q$ yields
\begin{equation*}
\left\| B_j\, X_{\epsilon, j} \right\|_{L^1_T(L^1(K))}\,\leq\, C\,\left\| X_{\eps, j}\right\|_{L^1_T(L^{q/2})}\, \text{meas}\left\{x\in K\,\big|\;| \nabla \rho_0(x) | \leq 2^{1 - j/2}\right\}^{\!(q-2)/q}\,.
\end{equation*}
Using the Sobolev embedding $H^1 \subset L^q$, we get
\begin{equation*}
\left\| X_{\epsilon, j} \right\|_{L^1_T(L^{q/2})}\, \leq\, C\, \left\| \nabla \rho_0 \right\|_{L^\infty}\, \left\| u_{\epsilon, j} \right\|_{L^2_T(H^1)}^2\,\leq\,C\,.
\end{equation*}
Hence we see that $B_j\,X_{\epsilon, j} = R_{\epsilon, j}$ is a remainder term, thanks to the assumption \eqref{NDCP} on $\rho_0$.

Next, we look at $\left(1 - B_j\right)\,X_{\epsilon, j}$: following the idea explained here above, we get
\begin{align*}
\left(1 - B_j\right)\,u_{\epsilon, j}\,&=\,\frac{1 - B_j}{|\nabla \rho_0|^2}\,\bigg\{\big(u_{\eps, j} \cdot \nabla \rho_0\big)\,\nabla \rho_0\,+\,
\big(u_{\eps, j}\cdot \nabla^\perp \rho_0\big)\, \nabla^\perp \rho_0  \bigg\} \\
\left(1 - B_j\right)\,u_{\epsilon, j}^\perp\,&=\,\frac{1 - B_j}{|\nabla \rho_0|^2}\,\bigg\{-\big(u_{\epsilon, j} \cdot \nabla^\perp \rho_0\big)\,\nabla \rho_0\, +\,
\big(u_{\epsilon, j} \cdot \nabla \rho_0\big)\, \nabla^\perp \rho_0  \bigg\}\,.
\end{align*}
Putting this in $(1 - B_j)X_{\epsilon, j}$, we see that the two terms parallel to $\nabla^\perp \rho_0$ cancel out. All that remains is
\begin{equation*}
\left(1 - B_j\right)\,X_{\eps, j}\,=\,\frac{1 - B_j}{|\nabla \rho_0|^2}\,\bigg\{ \big(u_{\eps, j} \cdot \nabla \rho_0\big)^2\,+\,\big(u_{\eps, j}\cdot\nabla^\perp \rho_0\big)^2\bigg\}\,\nabla\rho_0\,=\,
\Gamma_{\epsilon, j}\,.
\end{equation*}
We have thus proved that $X_{\epsilon, j}\, =\, \Gamma_{\epsilon, j}\, +\, R_{\epsilon, j}$, therefore
\begin{equation*}
\D\big(\rho_0\,u_{\epsilon, j} \otimes u_{\epsilon, j}\big)\,=\,\eta_{\epsilon, j}\,u_{\epsilon, j}^\perp\, +\,  \Gamma_{\epsilon, j}\, +\, R_{\epsilon, j}\,.
\end{equation*}

\textbf{Step 6: the new vorticity term.} It remains us to deal with the new vorticity term $\eta_{\epsilon, j}\,u_{\epsilon, j}^\perp$.
First we prove that $B_j\, \eta_{\epsilon, j}\, u_{\epsilon, j}^\perp$ is a remainder term $R_{\epsilon, j}$. Writing
$\eta_{\epsilon, j}\, =\, \eta_{\epsilon, j}^{(1)}\, +\, \epsilon^\theta\, \eta_{\epsilon, j}^{(2)}$ as in \eqref{EtaDec}, we see that, for $2 < q < +\infty$, one has
\begin{align*}
\left\| B_j\, \eta_{\eps, j}\, u_{\eps, j}^\perp\right\|_{L^1_T(L^1(K))}\,&\leq\, C\, \left\| \eta_{\eps, j}^{(1)} \right\|_{L^2_T(L^2)}\,\left\|u_{\eps, j} \right\|_{L^2_T(L^q)}\,
\text{meas}\left\{x\in K\,\big|\; |\nabla \rho_0(x)| \leq 2^{1 - j/2}\right\}^{\!(q-2)/(2q)} \\
&\quad+\,C\,\epsilon^\theta\,\left\| \eta_{\eps, j}^{(2)} \right\|_{L^2_T(L^2)}\,\left\| u_{\epsilon, j} \right\|_{L^2_T(L^2)}\,,
\end{align*}
which implies the bound
\begin{equation*}
\left\|B_j\,\eta_{\epsilon,j}\,u_{\epsilon,j}^\perp\right\|_{L^1_T(L^1(K))}\,\leq\,C\,\text{meas}\left\{x\in K\,\big|\;|\nabla \rho_0(x)|\leq 2^{1 - j/2}\right\}^{(q-2)/(2q)}\,+\,C(j)\,\epsilon^\theta\,.
\end{equation*}
Therefore, we do in fact see that $B_j\, \eta_{\epsilon, j}\, u_{\epsilon, j}^\perp\, =\, R_{\epsilon, j}$, as claimed.

Next, we use one last time the decomposition of $u_{\epsilon, j}$ on the basis
$\left(  \frac{\nabla \rho_0}{|\nabla \rho_0|}, \frac{\nabla^\perp \rho_0}{|\nabla \rho_0|} \right)$ to get
\begin{align*}
\left(1 - B_j\right)\,\eta_{\eps, j}\,u_{\eps, j}^\perp\, &=\,\frac{1 - B_j}{|\nabla \rho_0|^2}\, \eta_{\eps, j}\,\bigg\{\big(u _{\eps, j}^\perp\cdot\nabla\rho_0\big)\,\nabla\rho_0\,+\,
\big(u_{\eps, j} \cdot \nabla \rho_0\big)\, \nabla^\perp \rho_0 \bigg\} \\
&=\,\frac{1 - B_j}{|\nabla \rho_0|^2}\,\eta_{\epsilon, j}\,\big(u_{\epsilon, j}\cdot \nabla \rho_0\big)\, \nabla^\perp \rho_0\, +\, \Gamma_{\epsilon, j}\,.
\end{align*}
At this point, observe that, by taking the divergence of equation \eqref{Vort}, we can write $u_{\epsilon, j} \cdot \nabla \rho_0\,=\,\D\big(\rho_0\,u_{\eps, j}\big)\,=\,\D V_{\epsilon, j}\,-\,
\D\left(\epsilon^\theta \zeta_{\epsilon, j} + h_{\epsilon, j}\right)$, so
\begin{equation*}
\left(1 - B_j\right)\,\eta_{\eps, j}\,u_{\eps,j}^\perp\,=\,\frac{1 - B_j}{|\nabla \rho_0|^2}\,\eta_{\eps, j}\,\D V_{\eps, j}\,\nabla^\perp\rho_0\,+\, \Gamma_{\eps, j}\, +\, R_{\epsilon, j}\,.
\end{equation*}
Indeed, on the one hand $\epsilon^\theta \D\zeta_{\epsilon, j}\,\eta_{\epsilon, j}\,= \,R_{\epsilon, j}$; on the other, we can estimate
\begin{equation*}
\left\|\frac{1 - B_j}{|\nabla \rho_0|^2}\,\eta_{\eps, j}\, \D h_{\epsilon, j} \right\|_{L^1_T(L^1)}\,\leq\,C\,\left(\left\| \eta_{\eps,j}^{(1)}\right\|_{L^2_T(L^2)}\,\left\|h_{\eps,j}\right\|_{L^2_T(H^1)}
\,+\,\epsilon^\theta\,\left\| \eta_{\epsilon, j}^{(2)} \right\|_{L^2_T(L^2)}\, \| h_{\epsilon, j} \|_{L^2_T(H^1)}\right)\,,
\end{equation*}
so also that term is a remainder, thanks to the uniform bound \eqref{H1Bound} on $\big(h_{\eps, j}\big)_{\eps}$.

So far, we have thus obtained
\begin{equation*}
\D\big(\rho_0\, u_{\epsilon, j} \otimes u_{\epsilon, j}\big)\,=\,\frac{1 - B_j}{|\nabla \rho_0|^2}\;\eta_{\eps, j}\,\D V_{\epsilon, j}\;\nabla^\perp \rho_0\, +\, \Gamma_{\epsilon, j}\,+\,R_{\eps, j}\,.
\end{equation*}

\textbf{Step 7: end of the proof.} Recalling equation \eqref{Cor1}, we compute
\begin{align*}
\eta_{\eps, j}\,\D V_{\epsilon, j}\,&=\,-\,\eps\,\eta_{\epsilon, j}\,\partial_t\sigma_{\eps, j}\,=\,-\,\frac{1}{2}\,\eps\,\partial_t\!\left(\left|\sigma_{\eps, j}\right|^2\right)\,-\,
\eps\,\left(\eta_{\epsilon, j}\, -\, \sigma_{\epsilon, j}\right)\,\partial_t \sigma_{\epsilon, j} \\
&=\,-\,\frac{1}{2}\,\eps\,\partial_t\!\left(\left|\sigma_{\eps, j}\right|^2\right)\,-\,\eps\,\d_t\Big\{\big(\eta_{\epsilon, j}\,-\,\sigma_{\eps, j}\big)\,\sigma_{\epsilon, j}\Big\}\,+\,
\eps\,\curl f_{\epsilon, j}\, \sigma_{\epsilon, j}\,,
\end{align*}
where we have used also equation \eqref{Cor3} in passing from the first to the second line. By virtue of the uniform bounds we have already obtained on $\big(\sigma_{\epsilon, j}\big)_{\epsilon}$,
$\big(\eta_{\epsilon, j}\big)_{\epsilon}$ and $\big(f_{\epsilon, j}\big)_{\eps}$, we see that
\begin{equation*}
\frac{1 - B_j}{|\nabla \rho_0|^2}\; \eta_{\epsilon, j}\, \D V_{\epsilon, j}\; \nabla^\perp \rho_0\, =\, R_{\epsilon, j}\,.
\end{equation*}

In the end, we have shown that 
\begin{equation*}
\D\big(\rho_0\, u_{\epsilon, j} \otimes u_{\epsilon, j}\big)\,=\,\Gamma_{\epsilon, j}\, +\, R_{\epsilon, j}\,,
\end{equation*}
which finally achieves the proof of Proposition \ref{PropFin}.

\subsection{Conclusion: taking the limit} \label{ss:fully-conclusion}

For the quasi-homogeneous case, the results of Sections \ref{SecSg}, \ref{SecMag}, \ref{SecVisc} and \ref{SecQH} combine to show convergence of the sequence
$(r_\epsilon, u_\epsilon, b_\epsilon)_{\epsilon > 0}$ to a solution $(r, u, b)$ of system \eqref{EQLim}. Notice that we still have to prove the uniqueness part of Theorem \ref{THQH}.

In the fully non-homogeneous case, it remains to take care of the singular Coriolis force term: for this, we use the linear convergence properties for the density of Proposition \ref{PropSigma}.

%
\begin{prop}\label{PropCor}
The Coriolis force term satisfies the following convergence property: for any divergence-free $\phi \in \mathcal{D}\big(\mathbb{R}_+ \times \Omega ; \mathbb{R}^2\big)$,
write $\phi = \nabla^\perp \psi$; then
\begin{equation*}
\frac{1}{\epsilon} \int_0^{+ \infty} \int_\Omega \rho_\epsilon\, u_\epsilon^\perp \cdot \phi\, \dx \dt\,\tend_{\eps\rightarrow 0^+}\, -\int_0^{\infty}\int_\Omega\sigma\, \partial_t \psi\,\dx \dt\, +\,
\int_\Omega r_0\, \psi_{|t=0}\, \dx\,,
\end{equation*}
where $\s$ is the limit density oscillation function identified in Proposition \ref{PropSigma}, and we have improperly written the integral in space for the duality product in $H^{-3-\de}\times H^{3+\de}$.
\end{prop}

\begin{proof}
Let $T>0$ be a fixed positive time and let $\phi = \nabla^\perp \psi$ be a divergence-free test function, with $\psi \in \mathcal{D}\big([0, T[\, \times \Omega\big)$.
We use the decomposition $\rho_\epsilon = \rho_0 + \epsilon \sigma_\epsilon$ to write
\begin{align*}
\frac{1}{\epsilon} \int_0^T \int_\Omega \rho_\epsilon\, u_\epsilon^\perp \cdot \phi\, \dx \dt\, & =\,-\,\frac{1}{\epsilon}\,\big\langle \D\big(\rho_\eps\,u_\eps\big), \psi\big\rangle_{\mc D'\times\mc D}\, =\,
\frac{1}{\epsilon}\, \big\langle  \partial_t \rho_\epsilon, \psi  \big\rangle_{\mc D'\times\mc D} \\
& =\,\big\langle  \partial_t \sigma_\epsilon, \psi  \big\rangle_{\mc D'\times\mc D}\,=\,-\int_0^T \int_\Omega \sigma_\epsilon\,\partial_t \psi\,\dx \dt\,-\,\int_\Omega r_{0, \epsilon}\,\psi_{|t = 0}\,\dx\,.
\end{align*}
It is now matter of applying the convergence properties of \eqref{EQcv} and Proposition \ref{PropSigma}.
\end{proof}

This last trick of using $\psi$ as a test function rather than the divergence-free $\phi$ allows us to get rid of the singularity in the case of non-constant densities.
However, it forces us to take the $\curl$ of the whole equation.
Propositions \ref{PropDebut} and \ref{PropMag} to \ref{PropCor} show that, for any divergence-free $\phi \in \mathcal{D}\big(\mathbb{R}_+ \times \Omega ; \mathbb{R}^2\big)$,
which we can write $\phi = \nabla^\perp \psi$ with $\psi \in \mathcal{D}\big(\mathbb{R}_+ \times \Omega\big)$, one has
\begin{multline*}
\int_0^{+\infty}\!\!\!\int_\Omega \sigma\,\partial_t \psi\,\dx\dt\,+\int_0^{+\infty}\!\!\!\int_\Omega \bigg( \rho_0 u\cdot\partial_t \nabla^\perp \psi - b\otimes b: \nabla \nabla^\perp \psi-
\nu(\rho_0) \nabla u: \nabla \nabla^\perp \psi  \bigg) \dx \dt \\
+ \big\langle \Gamma, \D(\rho_0 \nabla^\perp \psi) \big\rangle_{\mc D'\times\mc D}\,=\, \int_\Omega \bigg(  r_0\, \psi_{|t = 0} - m_0 \nabla^\perp \psi_{|t = 0}  \bigg)\,\dx \dt\,.
\end{multline*}
Integration by parts show that we have indeed the weak form of the sought equation. The proof of Theorem \ref{THNH} is then completed.

\section{The quasi-homogeneous case: study of the limit system} \label{s:limit-system}

In this section, we focus our attention on the limit system for the quasi-homogeneous case, which we recall for the reader's convenience:
\begin{equation}\label{EQL}
\begin{cases}
\partial_t r \,+\, \D (r\,u)\, =\, 0 \\
\partial_t u \,+\, \D (u \otimes u)\, +\, \nabla \pi\, +\, \dfrac{1}{2} \nabla \big( |b|^2 \big)\, +\, r\,u^\perp\, =\, \nu(1) \,\Delta u \,+\, \D (b \otimes b) \\
\partial_t b \,+\, \D (u \otimes b \,-\, b \otimes u)\, =\, \mu(1)\, \Delta b \\
\D (u) \,=\, \D(b) \,=\, 0\,,
\end{cases}
\end{equation}
for some pressure function $\pi$. For simplicity, from now on we set $\nu(1)=\mu(1)=1$.

We shall prove that the solutions to system \eqref{EQL} are unique in the energy space, given regular enough initial data.
In particular, this completes the proof of Theorem \ref{THQH}, showing that the whole sequence of solutions $(r_\epsilon, u_\epsilon, b_\epsilon)_{\epsilon}$ weakly converges to the limit
point $(r, u, b)$, without the need to extract a subsequence.
We proceed in four steps. First, we find energy estimates for \eqref{EQL} at the order of regularity suited to prove uniqueness with stability estimates.
Then, we show rigorously the existence of solutions at this level of regularity. Finally, we prove uniqueness for system \eqref{EQL}.
The end result of this section is a well-posedness theorem for the limit system \eqref{EQL}.

\begin{thm}
Let $0 < \beta < 1$ and let $(r_0, u_0, b_0)\in H^1 \times H^1 \times H^{1 + \beta}$ be a set of initial data. There exists a (unique) solution $(r, u, b)$ of \eqref{EQL} related to
those initial data such that
\begin{enumerate}[(i)]
\item for all $0< \gamma < \beta$, we have $r \in C^0(\mathbb{R}_+ ; H^{1 + \gamma})$;
\item we have $u, b \in L^\infty_{\rm loc}(H^1) \cap L^2_{\rm loc}(H^2)$.
\end{enumerate}
Moreover, such a solutions is unique in the energy space $L^\infty_{\rm loc}(L^2 \cap L^\infty) \times L^\infty_{\rm loc}(L^2) \times L^\infty_{\rm loc}(L^2)$.
\end{thm}

\begin{rmk} \label{r:sol}
\begin{enumerate}[(i)]
\item It goes without saying that the previous regularity properties for the solution $(r,u,b)$ have also a quantitative counterpart, for which we refer to the estimates of Proposition \ref{PropO2}.
\item The uniqueness of solutions is a consequence of a stability estimate and a \emph{weak-strong} uniqueness result, which are stated in Proposition \ref{PropUni} below.
\end{enumerate}
\end{rmk}

In what follows, we will make extensive use of the Gagliardo-Nirenberg inequality (GN inequality for short, see Lemma \ref{GN}) as well as the Young inequality in the following form:
if $1/p + 1/q = 1$ then, for any $\eta > 0$ and $a, b \geq 0$, we have $ab \leq \eta a^p + \eta^{-q/p} b^q$. From now on, $\eta > 0$ will always note a small positive
constant to be fixed in the later parts of the proofs.
In addition, we try to find inequalities that are as precise as (reasonably) possible in order to highlight which terms have the most impact on the final estimates.

\subsection{Order 2 energy estimates} \label{ss:en-est}

In this section, we focus on finding order 2 \textsl{a priori} estimates for the limit system; one way to do this (see e.g. Section 4.4.1 of \cite{FG})
is to use $\Delta u$ and $\Delta b$ as test functions in \eqref{EQL}.
In addition, we attempt to optimise the estimates as far as growth in time is concerned: we show that
the higher order energy 
grows slower than any polynomial function of $T \geq 0$ of positive degree $h > 0$.

\begin{prop}\label{PropO2}
Let $(r, u, b)$ be a regular enough solution of \eqref{EQL} related to the (regular) initial datum $(r_0, u_0, b_0)$. Then, we have the following properties:
\begin{enumerate}[(i)]
\item we have $u, b \in L^\infty (L^2)$ and $\nabla u, \nabla b \in L^2(L^2)$, with the standard energy estimate \eqref{EQEn} below;
\item for any fixed $T>0$, we have $\nabla u, \nabla b \in L^\infty_T (L^2)$ and $\Delta u, \Delta b, \partial_t u, \partial_t b \in L^2_T(L^2)$, with explicit bounds:
for all $h > 0$, there is a constant $C = C\big(\|r_0\|_{L^2\cap L^\infty}, \| (u_0, b_0) \|_{H^1},h\big)>0$ such that
\begin{equation} \label{est:u-b_higher}
\| \nabla u \|_{L^\infty_T(L^2)} + \| (\Delta u, \partial_t u) \|_{L^2_T(L^2)} + \| \nabla b \|_{L^\infty_T(L^2)} + \| (\Delta b, \partial_t b) \|_{L^2_T(L^2)} \leq C (1 + T^h);
\end{equation}
\item for all $2 \leq p \leq +\infty$ and all $t\geq0$, $\| r(t) \|_{L^p} = \| r_0 \|_{L^p}$, and for all $0 < \gamma < \beta < 1$ and any fixed $T>0$, we also have
\begin{equation} \label{est:r_higher}
\| r \|_{L^\infty_T(H^{1+ \gamma})}\,\leq\,C(\beta, \gamma)\, \exp\left\{ C(\beta, \gamma) \left(  \int_0^T \| \nabla u \|_{H^1} \dt \right)^{\!\!2}  \right\}\, \| r_0 \|_{H^{1+\beta}}\,.
\end{equation}
\end{enumerate}
\end{prop}

\begin{proof}
First, testing the momentum equation with $u$ and the magnetic field equation with $b$ gives a basic energy estimate similar to \eqref{EN}: namely,
\begin{equation}\label{EQEn}
\frac{1}{2}\,\Big(\|u(t)\|^2 + \|b(t)\|^2 \Big) + \int_0^t\Big(\|\nabla u(s)\|^2 + \|\nabla b(s)\|^2\Big)\,\text{d}s\,\leq\,\frac{1}{2}\,\Big(\|u_0\|^2 + \|b_0\|^2 \Big)\,.
\end{equation}
Next, we use the fact that $r$ solves a pure transport equation with a divergence-free flow $u$ to see that the $L^p$ norms of $r(t)$ are preserved: $\| r(t)\|_{L^p} = \| r_0 \|_{L^p}$
for all $p\in[2,+\infty]$ and all $t\geq0$.

Now, we consider order 2 energy estimates. We test the momentum equation with $-\Delta u$ and the magnetic field equation with $-\Delta b$; summing and integrating by parts gives
\begin{align} \label{O2I}
&\frac{1}{2}\,\frac{\rm d}{ \dt}\Big(\|\nabla u\|^2 + \|\nabla b\|^2 \Big) + \Big(\|\Delta u\|^2+ \|\Delta b\|^2\Big)+\int_\Omega (b \cdot \nabla ) b \cdot \Delta u \dx +
\int_\Omega (b \cdot \nabla) u \cdot \Delta b\dx \\
&\qquad\qquad\qquad\qquad\qquad\qquad
= \int_\Omega (u \cdot \nabla) u \cdot \Delta u \dx + \int_\Omega r u^\perp \Delta u \dx + \int_\Omega (u \cdot \nabla) b \cdot \Delta b \dx\,. \nonumber
\end{align}

We start by handling the two inegrals which do not involve the magnetic field. On the one hand, using H\"older's inequality and then Proposition \ref{PropOP} yields
\begin{align*}
\left| \int_\Omega (u \cdot \nabla)u \cdot \Delta u \dx \right|\,\leq\,\| \Delta u\|_{L^2} \| u \|_{L^\infty} \| \nabla u \|_{L^2}\,\leq\, C \| \Delta u\|_{L^2}^{3/2} \| u \|_{L^2}^{1/2} \| \nabla u \|_{L^2}\,.
\end{align*}
Young's inequality with exponents $\frac{3}{4} + \frac{1}{4} = 1$ gives in turn
\begin{equation*}
\left| \int_\Omega (u \cdot \nabla)u \cdot \Delta u \dx \right|\,\leq\,\eta\,\| \Delta u\|_{L^2}^2\,+\,C(\eta)\,\| u \|_{L^2}^2\, \| \nabla u \|_{L^2}^4\,=\,\eta\,\| \Delta u\|_{L^2}^2\, +\,
M_1(t)\,\| \nabla u \|_{L^2}^2\,,
\end{equation*}
where $M_1(t)\,:= C(\eta)\| u(t) \|_{L^2} \| \nabla u(t) \|_{L^2}^2$. Remark that, by \eqref{EQEn}, $M_1 \in L^1(\mathbb{R}_+)$ is a globally integrable function, with
$\| M_1 \|_{L^1(\mathbb{R}_+)}$ depending only on $\eta$, $\|u_0\|_{L^2}$ and $\|b_0\|_{L^2}$.

On the other hand, we use H\"older's inequality with exponents $\frac{1}{p} + \frac{1}{q} + \frac{1}{2} = 1$ which we will choose later, followed by the GN inequality: for all $\eta > 0$,
\begin{align*}
\left| \int_\Omega ru^\perp \cdot \Delta u \dx \right|\,&\leq\,\|r\|_{L^p}\,\| u \|_{L^q}\,\| \Delta u \|_{L^2}\;\leq\;C(q)\,\| r \|_{L^2}^{2/p}\,\| r \|_{L^\infty}^{1 - 2/p}\,
\| u \|_{L^2}^{2/q}\,\| \nabla u \|_{L^2}^{1 - 2/q}\,\| \Delta u \|_{L^2} \\
&\leq\, \eta\,\| \Delta u \|_{L^2}^2\, +\, C\big(\eta, \| r_0 \|_{L^2\cap L^\infty}, q\big)\,\| u \|_{L^2}^{4/q}\, \| \nabla u \|_{L^2}^{2\big(1-2/q\big)}\,.
\end{align*}
Since $\|u\|_{L^2}\in L^\infty(\R_+)$ and $\| \nabla u \|_{L^2}^2 \in L^1(\mathbb{R}_+)$ by the energy inequality \eqref{EQEn}, we see that, for any arbitrary $h > 0$, we can chose $q$ so large that
the function $M_{1+h}(t) := C \| u(t) \|_{L^2}^{4/q} \| \nabla u(t) \|_{L^2}^{2\big(1 -2/q \big)}$ belongs to the space $L^{1+h}(\mathbb{R}_+)$. Therefore
\begin{equation} \label{est:ru}
\left| \int_\Omega r\,u^\perp \cdot \Delta u\, \dx \right|\, \leq\, \eta\, \| \Delta u \|_{L^2}^2\, +\, M_{1+h}(t)\,,
\end{equation}
with the norm $\| M_{1+h} \|_{L^{1+h}(\mathbb{R}_+)}$ only depending on the quantities $(\eta, \| r_0 \|_{L^2\cap L^\infty}, \|(u_0,b_0) \|_{L^2}, h)$.

Now we take care of the three remaining integrals in \eqref{O2I} involving the magnetic field. As with the first integral, Proposition \ref{PropOP} yields
\begin{align*}
\left|  \int_\Omega (b \cdot \nabla)b \cdot u \dx \right| & \leq \| \Delta u \|_{L^2} \| b \|_{L^\infty} \| \nabla b \|_{L^2}\;\leq\;C
\| \Delta u \|_{L^2} \| \Delta b \|_{L^2}^{1/2} \| b \|_{L^2}^{1/2} \| \nabla b \|_{L^2}\\
& \leq \eta \| \Delta u \|_{L^2}^2 + C(\eta) \| \Delta b \|_{L^2} \| b\|_{L^2} \| \nabla b \|_{L^2}^2\\
& \leq 
\eta \bigg\{ \| \Delta u \|_{L^2}^2 + \| \Delta b \|_{L^2}^2 \bigg\} + M_1(t)\| \nabla b \|_{L^2}^2\,,
\end{align*}
where we have set $M_1(t) := C(\eta) \| b(t) \|_{L^2}^2 \| \nabla b(t) \|_{L^2}^2$. Notice that $M_1\in L^1(\mathbb{R}_+)$, with $\| M_1 \|_{L^1(\mathbb{R}_+)}$ depending only
on $\eta$, $\|u_0\|_{L^2}$ and $\|b_0\|_{L^2}$.

Next, similarly as above, we use H\"older's inequality, the GN inequality and then Young's inequality twice, to obtain
\begin{align*}
\left|  \int_\Omega (b \cdot \nabla )u \cdot \Delta b \dx   \right| & \leq \| b \|_{L^4} \| \nabla u \|_{L^4} \| \Delta b \|_{L^2} \leq C \| b \|_{L^2}^{1/2} \| \nabla b \|_{L^2}^{1/2} \| \nabla u \|_{L^2}^{1/2} \| \Delta u \|_{L^2}^{1/2} \| \Delta b \|_{L^2}\\
& \leq \eta \| \Delta b \|_{L^2}^2 + C(\eta) \| b \|_{L^2} \| \nabla b \|_{L^2} \| \nabla u \|_{L^2} \| \Delta u \|_{L^2}\\
& \leq \eta \bigg\{ \| \Delta b \|_{L^2}^2 + \| \Delta u \|_{L^2}^2  \bigg\} +  M_1(t) \| \nabla b \|_{L^2}^2,
\end{align*}
where this time $M_1 = C(\eta) \|b\|^2_{L^2} \| \nabla u \|_{L^2}^2 \in L^1(\mathbb{R}_+)$, with $\| M_1 \|_{L^1(\mathbb{R}_+)} = C(\eta, \| u_0\|_{L^2}, \| b_0 \|_{L^2})$.
Exactly the same computations \textsl{mutatis mutandi} yield
\begin{equation*}
\left|  \int_\Omega (u \cdot \nabla) b \cdot \Delta b \dx   \right| \leq \eta \bigg\{ \| \Delta b \|_{L^2}^2 + \| \Delta u \|_{L^2}^2  \bigg\} +  M_1(t) \| \nabla u \|_{L^2}^2\,,
\end{equation*}
with $M_1 = C(\eta) \|u\|^2_{L^2} \| \nabla b \|_{L^2}^2 \in L^1(\mathbb{R}_+)$ and $\| M_1 \|_{L^1(\mathbb{R}_+)} = C(\eta, \| u_0\|_{L^2}, \| b_0 \|_{L^2})$.

\medbreak
In the end, putting all these estimates together and choosing $\eta>0$ small enough, we obtain the differential inequality
\begin{equation*}
\frac{1}{2}\,\frac{\rm d}{ \dt}\Big(\|\nabla u\|^2 + \|\nabla b\|^2 \Big)\, +\,\frac{1}{2}\, \Big(\|\Delta u\|^2+ \|\Delta b\|^2\Big)\,\leq\, M_{1+h}(t)\,+\,M_1(t)\,\Big(\|\nabla u\|^2 + \|\nabla b\|^2 \Big)\,.
\end{equation*}
An application of Gr\"onwall's lemma finally yields, for all fixed $T>0$, the inequality
\begin{multline} \label{est:2order_u-b}
\sup_{t\in[0,T]}\Big(\|\nabla u(t)\|^2 + \|\nabla b(t)\|^2 \Big)\,+\,\int^T_0\Big(\|\Delta u(t)\|^2+ \|\Delta b(t)\|^2\Big)\dt \\
\leq\, 2 \left( \| \nabla u_0 \|_{L^2}^2 + \| \nabla b_0 \|_{L^2}^2 + \int_0^T M_{1+h}(t)\,\dt \right)\,\exp\left\{ \int_0^{+\infty} M_1(t)\,\dt \right\}\,.
\end{multline}
Observe that
$$
\int^T_0M_{1+h}(t)\,\dt\,=\,\left\|M_{1+h}\right\|_{L^{1+h}(\R_+)}\,T^{h/(1+h)}\,\leq\,C\,\left(1+T^h\right)\,,
$$
which obviously implies estimate \eqref{est:u-b_higher}, up to the time derivative terms, which will be dealt with in a while.
For the moment, remark that, since 
we have shown $L^2_T(H^2)$ bounds on the velocity field, we can apply Proposition 5.2. of \cite{D} (which is also expressed in a much more thorough form in Theorem 3.33 of \cite{BCD})
to the transport equation for $r$: for all $0 < \gamma < \beta < 1$, we get propagation of the $L^\infty_T(H^{1+\g})$ regularity of $r$, together with estimate \eqref{est:r_higher}.

Let us come back to the proof of \eqref{est:u-b_higher}: it remains us to bound the time derivative terms.
We start by looking at $\partial_t u$. In order to get rid of the pressure term, we apply the Leray projector $\mathbb{P}$, which is the $L^2$-orthogonal projector on the subspace of divergence-free
functions. It can be defined as a Fourier multiplier:
\begin{equation*}
\forall f \in L^2, \quad \what{ \mathbb{P} f } (\xi) = \what{f}(\xi) - \left( \frac{\xi}{|\xi|^2} \cdot \what{f}(\xi) \right) \xi = \left( \frac{\xi^\perp}{|\xi|^2} \cdot \what{f}(\xi) \right) \xi^\perp.
\end{equation*}
In particular, it commutes with differential operators (both on time and space variables).
Applying the Leray projector to the momentum equation in \eqref{EQL} gives
\begin{equation}\label{Leray}
\partial_t u + \mathbb{P} \big[  \D(u \otimes u - b \otimes b) + ru^\perp \big] = \Delta u\,,
\end{equation}
since $\div u=0$. In fact, since $\mathbb{P}$ is a Fourier multiplier associated to a homogeneous function of degree zero, it is a continuous operator over $L^2$;
so we need only estimates on $\Delta u$, $ru$ and $\D (u \otimes u - b \otimes b)$ to conclude \eqref{est:u-b_higher}.

First of all, for $1/p + 1/q = 1/2$, H\"older's and the GN inequalities, combined with \eqref{EQEn}, give
\begin{equation*}
\| r\,u \|_{L^2}\, \leq\, \| r \|_{L^p}\, \| u \|_{L^q}\, \leq\, C\big(q, \|r_0\|_{L^2\cap L^\infty}, \|u_0\|_{L^2}, \| b_0 \|_{L^2}\big)\, \| \nabla u \|_{L^2}^{1 - 2/q}\,=\,C\,M_{1+h}^{1/2}\,,
\end{equation*}
where $M_{1+h}$ has been defined in \eqref{est:ru}. Hence, by taking $q$ large enough, we see that $ru \in L^{2(1+h)}(L^2)$.
Secondly, using that $\div(u\otimes u)\,=\,(u\cdot\nabla)u$, an application of Proposition \ref{PropOP} 
gives
\begin{align}\label{EQQuadT}
\left\| \D \big(u \otimes u\big) \right\|^2_{L^2}\,&\leq\,C\,\| u \|^2_{L^\infty}\,\| \nabla u \|^2_{L^2}\,\leq\, C\,\| u \|_{L^2}\,\| \Delta u \|_{L^2} \| \nabla u \|^2_{L^2}\,.
\end{align}
Integrating over $t \in [0, T]$ and using estimate \eqref{EQEn}, after an application of Young inequality we obtain, for all $h > 0$, the bound
\begin{align*}
\left\| \D \big(u \otimes u \big) \right\|_{L^2_T(L^2)}^2\,&\leq\,C\int^T_0\big(\| u \|^2_{L^2}\,\|\nabla u \|^4_{L^2}+\|\Delta u \|^2_{L^2}\big)\dt \\
&\leq\,C\big(\|(u_0,b_0)\|_{L^2}\big)\,\left(\sup_{[0,T]}\left(\|\nabla u \|^2_{L^2}\right)\,+\,\|\Delta u \|^2_{L^2_T(L^2)}\right)\,.
\end{align*}
In view of \eqref{est:2order_u-b}, this term can be finally bounded by the right-hand side of that inequality.
The same computations with the magnetic field yield
\begin{equation*}
\left\| \D \big(b \otimes b \big) \right\|_{L^2_T(L^2)}^2\,\leq\,C\big(\|(u_0,b_0)\|_{L^2}\big)\,\left(\sup_{[0,T]}\left(\|\nabla b \|^2_{L^2}\right)\,+\,\|\Delta b \|^2_{L^2_T(L^2)}\right)\,,
\end{equation*}
which again can be bounded by the right-hand side of \eqref{EQEn}.
In the end, the combination of all those estimates implies that also $\left\|\d_tu\right\|_{L^2_T(L^2)}$ satisfies \eqref{est:u-b_higher}.

It only remains to find the estimate on $\left\|\d_tb\right\|_{L^2_T(L^2)}$, for which we use the magnetic field equation, namely the third equation in \eqref{EQL}. Estimate \eqref{est:2order_u-b}
already provides us with a bound for $\Delta b$ in the sought space. On the other hand, the quadratic terms can be estimated exactly as the corresponding ones appearing in the momentum equation.
We finally deduce that also $\left\|\d_tb\right\|_{L^2_T(L^2)}$ fulfills \eqref{est:u-b_higher}, and this completes the proof of the proposition.
\end{proof}

\subsection{Existence result}

In this section, we explain how solutions of \eqref{EQL}, whether they be weak solution in the energy space or with the level of regularity described in Proposition \ref{PropO2}, can be constructed.

\begin{prop}\label{PropEx}
Assume that $r_0 \in L^2 \cap L^\infty$ and that $u_0, b_0 \in L^2$ are two divergence-free functions. Then there exists a weak solution $(r, u, b)$ of system \eqref{EQL} related
to those initial data, such that $r \in L^\infty (L^2 \cap L^\infty)$ and\footnote{Given a Banach space $X$, we note $C^0_b(X)$ the space fo continuous and globally bounded functions
on $\mathbb{R}_+$ with values in $X$.} $u, b \in C^0_b (L^2)$, with $\nabla u$ and $\nabla b$ belonging to $L^2(L^2)$. Moreover
this solution satisfies the basic energy inequality \eqref{EQEn} and the conservation of the $L^p$ norms of $r$ in time, for any $p\in[2+\infty]$. 

Suppose now that $r_0 \in H^{1 + \beta}$, for some $\beta > 0$, and that $u_0, b_0 \in H^1$. Then there exists a weak solution $(r, u, b)$ of \eqref{EQL} such that, in addition to the previous properties,
one also has, for all $T > 0$, $r \in C_T^0(H^{1+\gamma})$ for all $0<\g<\min\{1,\beta\}$, and $u, b \in C_T^0(H^1) \cap L^2_T(H^2)$. Moreover, this solution satisfies inequalities \eqref{est:u-b_higher}
and \eqref{est:r_higher}.
\end{prop}

The proof of the previous proposition is standard (see e.g. \cite{CDGG} or \cite{DeA-F}); we give here most of the details for reader's convenience.
First of all, we implement Friedrichs scheme and construct smooth solutions to an approximate system. After deriving uniform bounds (of first and second order),
we show that those solutions tend (weakly) to a solution of \eqref{EQL}.

\medbreak
\textbf{Step 1: approximate system.} Let $j \geq 2$ and $A_j$ be the spectral projection operator defined in the following way:
\begin{equation*}
\forall f \in L^2\,, \qquad \mathcal{F}\big[ A_j f \big] (\xi) = \mds{1}_{|\xi| \leq j} \what{f}(\xi)\,.
\end{equation*}
Recall the Leray projector $\mathbb{P}$ from \eqref{Leray}, and that we have set $\mu(1)=\nu(1)=1$.
Set $r_1(t, x) = S_1 r_0(x)$, where $S_1$ is the low frequency cut-off operator defined in \eqref{eq:S_j}. For $j\geq2$,
we consider the sequence of approximate systems
\begin{equation}\label{ApprMHD}
\begin{cases}
\partial_t u_j + \mathbb{P} A_j \D \big(u_j \otimes u_j - b_j \otimes b_j\big) + \mathbb{P} A_j \big[ r_{j-1} u_j^\perp \big] =  \Delta u_j\\[1ex]
\partial_t b_j +  A_j \D \big( u_j \otimes b_j - b_j \otimes u_j\big) =  \Delta b_j\\[1ex]
\D (u_j) = \D(b_j) = 0\,,
\end{cases}
\end{equation}
which we equip with the initial datum $\big(u_j,b_j\big)_{|t=0}\,=\,\big(A_ju_0,A_jb_0\big)$. Given $r_{j-1}$, applying the Cauchy-Lipschitz theorem in the Banach space
\begin{equation*}
\mc X_j = \left\{ f\,\in\,L^2\;\big| \qquad \Supp\what{f} \subset B(0, j)  \right\}
\end{equation*}
gives the existence of a unique  $C^\infty\big(\,]T_-(j), T_+(j)[\,; \mc X_j\big)$ maximal solution.
Then, we compute the function $r_{j}$ by solving the linear transport equation
\begin{equation*}
\begin{cases}
\partial_t r_{j} + \D\big(r_{j} u_{j}\big) = 0 \\[1ex]
\big(r_{j}\big)_{|t = 0} = S_jr_0\,.
\end{cases}
\end{equation*}

We now look for uniform bounds for $\big(r_j,u_j,b_j\big)_j$ in suitable spaces. First of all, using the fact that $r_{j}$ solves a pure transport equation by a divergence-free velocity field $u_j$,
one gathers that
\begin{equation*}
\left\|r_{j}(t)\right\|_{L^2}\,+\,\left\|r_{j}(t)\right\|_{L^\infty}\,=\,\left\| S_j r_0 \right\|_{L^2}\,+\,\left\| S_j r_0 \right\|_{L^\infty}\,\leq\,C\,\|r_0 \|_{L^2\cap L^\infty}\,.
\end{equation*}
for all $t\in[0,T_+(j)[\,$. This implies that $\big(r_j\big)_j\subset L^\infty\big(L^2\cap L^\infty\big)$. Next, we easily see that the maximal solution $\big(u_j,b_j\big)$ satisfies
the basic energy estimate \eqref{EQEn}; it also satisfies the order 2 estimates \eqref{est:u-b_higher}.
Indeed, testing the momentum equation in \eqref{ApprMHD} with $-\Delta A_j u$ (for example), which is both in $\mc X_j$ and divergence-free, gives, for all $0 \leq t < T_+(j)$,
the equality
\begin{equation*}
\frac{1}{2} \frac{\rm d}{ \dt} \int_\Omega\left|\nabla u_j\right|^2 \dx + \int_\Omega \left|\Delta u_j\right| \dx = \int_\Omega \D \big(u_j \otimes u_j - b_j \otimes b_j\big) \cdot \Delta u_j \dx +
\int_\Omega r_{j-1} u_j^\perp \cdot \Delta u_j \dx\,.
\end{equation*}
An analogous relation holds for the magnetic field $b_j$. Then, one can repeat the same computations as in Subsection \ref{ss:en-est} to finally get \eqref{est:u-b_higher}, as claimed.

We now show that the approximate solutions do not blow-up in finite time.
For this, fix $j \geq 2$: the basic energy estimates state that $\| u_j(t) \|_{L^2}$ and $\| b_j(t) \|_{L^2}$ are bounded for $0 \leq t < T_+(j)$, therefore (using the bounds on $\big(r_j\big)_j$ and the
spectral localisation) so are the norms of the time derivatives $\left\| \partial_t u_j (t) \right\|_{L^2}$ and $\left\| \partial_t b_j (t) \right\|_{L^2}$.
Hence, the solution of the ODE system \eqref{ApprMHD} satisfies the Cauchy criterion for $t < T_+(j)$, and necessarily $T_+(j) = +\infty$.

\medbreak
\textbf{Step 2: convergence.} 
For the sake of generality, we prove convergence of the sequence of approximate solutions relying only on the basic energy estimates \eqref{EQEn} and the conservation of $L^p$ norms
for the $r_j$. Specifically, we will be using only the following uniform bounds:
\begin{equation*}
\big(u_j\big)_{j\geq 2}\,, \; \big(b_j\big)_{j\geq 2}\; \subset\;L^\infty(L^2) \cap L^2_{\rm loc}(H^1) \qquad \text{ and } \qquad \big(r_j\big)_{j \geq 2}\, \subset\, L^\infty (L^2 \cap L^\infty)\,.
\end{equation*}
Those bounds yield the following weak convergence properties (up to an extraction): there exists a triplet $(r, u, b) \in L^\infty(L^2 \cap L^\infty) \times L^2_{\rm loc}(H^1) \times L^2_{\rm loc}(H^1)$
such that, for all fixed $T>0$, one has
\begin{equation} \label{conv:r-u-b}
\big(u_j, b_j\big)\,\wtend\, (u, b) \qquad \text{ in }\; L^2_T(H^1) \qquad\qquad \text{ and } \qquad\qquad r_j\, \wtend^*\, r \qquad \text{ in }\; L^\infty(L^2 \cap L^\infty)\,.
\end{equation}

In order to achieve convergence in the non-linear terms, we are going to prove that both $\big(\partial_t u_j\big)_{j}$ and $\big(\partial_t b_j\big)_{j}$ are uniformly bounded in
$L^2_{\rm loc}(H^{-1})$.
Notice that both $\mathbb{P}$ and $A_j$ are continuous for all the $H^s$ topologies (with $s \in \mathbb{R}$). Therefore, by using \eqref{ApprMHD}, we get
\begin{align*}
\left\| \partial_t u_j \right\|_{L^2_T(H^{-1})} + \left\| \partial_t b_j \right\|_{L^2_T(H^{-1})}\,&\leq\,\left\| u_j \right\|_{L^2_T(H^1)} + \left\| b_j(t) \right\|_{L^2_T(H^{1})} \\
&\qquad\quad+\sum_{f,g\in\{u_j,b_j\}} \Big\| \D \big( f \otimes g\big)\Big\|_{L^2_T(H^{-1})} + \| r_{j-1} u_j \|_{L^2_T(H^{-1})}\,.
\end{align*}
The last term is bounded by $\| r_{j-1} u_j \|_{L^2_T(H^{-1})} \leq C_T\| r_{j-1} \|_{L^\infty_T(L^\infty)} \| u_j \|_{L^\infty_T(L^2)}$, so we only have to worry about the quadratic terms.
If $f, g \in L^2_T(H^1)$, then, for all $0 \leq t \leq T$, using the Sobolev embedding $H^1 \subset L^4$ followed by the Cauchy-Schwarz inequality in $L^2_T$, we get
\begin{equation*}
\Big\| \D \big(f \otimes g\big) \Big\|^2_{L^2_T(H^{-1})}\,\leq\,\int_0^T \| f(t) \|^2_{L^4} \| g(t) \|^2_{L^4}\,\dt \leq \| f \|^2_{L^2_T(H^1)} \| g \|^2_{L^2_T(H^1)}\, <\, +\infty\,.
\end{equation*}
These computations show that $\big(\partial_t u_j\big)_{j}$ and $\big(\partial_t b_j\big)_{j}$ are indeed bounded in $L^2_T(H^{-1})$, from which we infer the uniform boundness of 
$\big(u_j\big)_{j}$ and $\big(b_j\big)_{j}$ in $C^{0, 1/2}_T(H^{-1})$.
On the other hand, thanks to the the compact embedding $L^2(K) \subset H^{-1}(K)$ for all compact $K \subset\Omega$, we deduce that, for almost all $0  \leq t \leq T$,
the sequences $\big(u_j(t)\big)_{j}$ and $\big(b_j(t)\big)_{j}$ are relatively compact in $H^{-1}_{\rm loc}$. An application of the Ascoli-Arzel\`a theorem gives then,
for all $T>0$, the strong convergence
\begin{equation*}
\big(u_j, b_j\big)\,\tend\,(u, b) \qquad\qquad \text{ in }\qquad L^\infty_T(H^{-1}_{\rm loc})\,.
\end{equation*}
Because $(u_j)_j$ and $(b_j)_j$ are also bounded in $L^2_T(H^1)$, interpolation between Sobolev spaces gives strong $L^2$ convergence
\begin{equation*}
\big(u_j, b_j\big)\,\tend\,(u, b) \qquad\qquad \text{ in }\qquad L^2_T(L^2_{\rm loc})\,.
\end{equation*}

Next, using the fact that the $r_j$ solve a linear transport equation and arguing exactly as in Subsection \ref{SecStrDen}, we get strong convergence
\begin{equation*}
r_j\, \tend\, r \qquad\qquad \text{ in }\qquad L^2_T(L^2_{\rm loc})\,,
\end{equation*}
which in turn gives convergence of the products $\big(r_j u_{j-1}\big)_j$ and $\big(r_ju_j\big)_j$ in the sense of distributions in $\R_+\times\Omega$.

\textbf{Step 3: weak solutions.} We aim to prove that the triplet $(r, u, b)$, identified in \eqref{conv:r-u-b}, is in fact a weak solution of \eqref{EQL}.
The only terms whose convergence is not completely obvious at this point are the quadratic terms in $u_j$ and $b_j$.
Let $\phi \in \mathcal{D}\big([0, T[\,\times \Omega;\mathbb{R}^2\big)$ be a divergence-free test function. We will only prove the convergence of
\begin{equation*}
\int_0^T \int_\Omega A_j \mathbb{P}\big(u_j \otimes b_j\big) : \nabla \phi\, \dx \dt \,=\, \int_0^T \int_\Omega \big(u_j \otimes b_j\big) : A_j \nabla \phi\, \dx \dt\,,
\end{equation*}
all other quadratic terms being similar. Taking the difference between the previous integral and the one we desire, we get
\begin{multline*}
\left| \int_0^T \int_\Omega \bigg\{ (u_j \otimes b_j): A_j \nabla \phi  -  (u \otimes b): \nabla \phi \bigg\} \dx \dt \right| \\
\leq  \left| \int_0^T \int_\Omega (u_j \otimes b_j): (A_j - I) \nabla \phi\, \dx \dt \right| + \left| \int_0^T \int_\Omega \big(u_j \otimes b_j - u \otimes b\big): \nabla \phi \,\dx \dt \right|\,.
\end{multline*}
Using the Sobolev embedding $H^1 \subset L^4$, we see that the first integral on the right-hand side is bounded by 
\begin{equation*}
\left| \int_0^T \int_\Omega (u_j \otimes b_j): (A_j - I) \nabla \phi \dx \dt \right| \leq \| u_j \|_{L^2_T (L^4)} \| b_j \|_{L^2_T (L^4)}\| (A_j - I)\nabla \phi \|_{L^\infty_T(L^2)}\,,
\end{equation*}
which converges to $0$ for $j\ra +\infty$. This comes from the uniform bound (with respet to time)
\begin{equation*}
\| (A_j - I)\nabla \phi (t) \|_{L^2}^2 \leq \frac{1}{j^2} \left\| \nabla^2 \phi (t) \right\|_{L^2}^2 \leq \frac{C}{j^2}\,.
\end{equation*}
For the other integral, we simply recall that $\big(u_j, b_j\big)\, \tend\, (u, b)$ in $L^2_T(L^2_{\rm loc})$, so that we have strong convergence of the tensor products
\begin{equation}
u_j \otimes b_j \, \tend\, u \otimes b \qquad \qquad \text{ in } \qquad L^1_T(L^1_{\rm loc}).
\end{equation}
Thus, we have proved convergence for the quadratic term:
\begin{equation*}
\int_0^T \int_\Omega A_j \mathbb{P}\big(u_j \otimes b_j\big) : \nabla \phi\, \dx \dt \,\tend_{j\rightarrow +\infty}\, \int_0^T \int_\Omega (u \otimes b) : \nabla \phi \,\dx \dt\,.
\end{equation*}

In the end, we have shown that $(r, u, b)$ is indeed a weak solution of \eqref{EQL}.
Finally, the Banach-Steinhaus theorem makes sure that the inequalities of Proposition \ref{PropO2} are carried from the approximate solutions to the limit point $(r, u, b)$.
This completes the proof of Proposition \ref{PropEx}.

\subsection{Uniqueness for the limit system}

The proof of the uniqueness for system \eqref{EQL} is based on stability estimates in the energy space. Hence, we require enough regularity to perform those energy estimates, without any
attempt of sharpness in our statement (very likely, a well-posedness result like the one in \cite{P-Z-Z} may be proved also for our system).

\begin{prop}\label{PropUni}
Fix $0 < \beta < 1$. Let $\big(r_{0, 1}, u_{0, 1}, b_{0, 1}\big) \in H^{1+\beta}\times H^1 \times H^1$ and
$\big(r_{0, 2}, u_{0, 2}, b_{0, 2}\big) \in \big(L^2 \cap L^\infty\big) \times L^2 \times L^2$ be two sets of initial data, and for $i=1,2$, consider $\big(r_i, u_i, b_i\big)$ two associated solutions
of \eqref{EQL} such that:
\begin{enumerate}[(1)]
\item $r_i\in L^\infty(L^2 \cap L^\infty)$ and $u_i, b_i \in L^\infty(L^2)$, with $\nabla u_i,\nabla b_i \in L^2(L^2)$;
\item for all $T>0$, $u_1, b_1 \in L^\infty_T(H^1) \cap L^2_T(H^2)\cap W^{1,2}_T(L^2)$ and $r_1\in L^\infty_T(H^{1+\gamma})$ for all $0 < \gamma < \beta$.
\end{enumerate}
Define $\delta r = r_2 - r_1$, $\delta u = u_2 - u_1$ and $\delta b = b_2 - b_1$, and note by $\delta r_0$, $\delta u_0$, $\delta b_0$ the same quantities computed on the initial data.

Then, we have the following stability estimate: for every $0 \leq t \leq T$,
\begin{multline}\label{EQUni}
\| \delta u(t) \|_{L^2}^2\, +\, \| \delta b(t) \|_{L^2}^2\, +\, \| \delta r(t) \|_{L^2}^2\,+\,\int_0^t \bigg\{\| \nabla \delta u \|_{L^2}^2\, +\, \| \nabla \delta b \|_{L^2}^2\bigg\}\, \text{d}s \\
\leq\, C\, \bigg(  \| \delta u_0 \|_{L^2}^2\, +\, \| \delta b_0 \|_{L^2}^2 \,+\, \| \delta r_0 \|_{L^2}^2 \bigg)\,,
\end{multline}
where the constant $C > 0$ depends on $T$, $\mu(1)$, $\nu(1)$ and $\big(\| r_{0,1} \|_{H^{1+\beta}}, \| u_{0,1} \|_{H^1},\| b_{0,1} \|_{H^1}\big)$.
\end{prop}

Our statement is pretty much a \emph{weak-strong uniqueness} result. 
In particular, we deduce uniqueness of the solutions in the energy space, given regular initial data.

\begin{cor}\label{CorWPL}
Consider $0 < \beta < 1$ and $(r_0, u_0, b_0) \in H^{1 + \beta} \times H^1 \times H^1$. For that initial datum, there is exactly one solution $(r, u, b)$ of \eqref{EQL} in the energy space,
that is such that $r \in L^\infty(L^2 \cap L^\infty)$ and  $u, b \in L^\infty(L^2)$, with $\nabla u, \nabla b\in L^2(H^1)$.
\end{cor}

Let us now prove the previous proposition.

\begin{proof}[Proof of Proposition \ref{PropUni}]
We start by remarking that the existence of the two sets of solutions $\big(r_i, u_i, b_i\big)_{i=1,2}$, with the claimed level of regularity, is a consequence of Proposition \ref{PropEx}.

In order to prove inequality \eqref{EQUni}, we take the difference between the equation solved by $\big(r_2, u_2, b_2\big)$ and the one solved by $\big(r_1, u_1, b_1\big)$: this gives
\begin{equation*}
\begin{cases}
\partial_t \delta u + (u_2 \cdot \nabla) \delta u + (\delta u \cdot \nabla) u_1 + \nabla\Pi + r_2 \delta u^\perp + \delta r u_1^\perp
= \Delta \delta u + (\delta b \cdot \nabla ) b_1 + ( b_2 \cdot \nabla) \delta b\\[1ex]
\partial_t\delta b + (\delta u \cdot \nabla) b_1 + (u_2 \cdot \nabla) \delta b = (\delta b \cdot \nabla) u_1 + (b_2 \cdot \nabla) \delta u + \Delta\delta b\\[1ex]
\partial_t\delta r + (u_2 \cdot \nabla) \delta r = - \delta u \cdot \nabla r_1\\[1ex]
\D (\delta u) = \D (\delta b) = 0\,,
\end{cases}
\end{equation*}
where we have set $\Pi\,:=\,\pi_2 - \pi_1 + \frac{1}{2}|b_2|^2 - \frac{1}{2}|b_1|^2$ in the first equation. Recall that we are taking $\nu(1)=\mu(1)=1$ for simplicity of presentation.

Omitting (for the sake of brevity) a standard regularisation process, let us perform energy estimates directly on the previous system.
So, test the first equation with $\delta u$, the second one with $\delta b$ and the third one with $\delta r$: one gets
\begin{align}
&\frac{1}{2} \frac{\text{d}}{\dt} \int_\Omega |\delta u|^2 \dx + \int_\Omega (u_2 \cdot \nabla ) \delta u \cdot \delta u \dx + \int_\Omega (\delta u \cdot \nabla) u_1 \cdot \delta u \dx + 
\int_\Omega \delta r u_1^\perp \cdot \delta u \dx \label{U1} \\
&\qquad\qquad\qquad\qquad\qquad\qquad
+ \int_\Omega |\nabla \delta u|^2 \dx = \int_\Omega (\delta b \cdot \nabla) b_1 \cdot \delta u \dx + \int_\Omega (b_2 \cdot \nabla) \delta b \cdot \delta u \dx  \nonumber \\[1ex]
&\frac{1}{2} \frac{\rm d}{\dt} \int_\Omega |\delta b|^2 \dx + \int_\Omega (\delta u \cdot \nabla) b_1 \cdot \delta b \dx + \int_\Omega (u_2 \cdot \nabla) \delta b \cdot \delta b \dx +
 \int_\Omega |\nabla \delta b|^2 \dx \label{U2}\\
&\qquad\qquad\qquad\qquad\qquad\qquad\qquad\qquad
= \int_\Omega (\delta b \cdot \nabla) u_1 \cdot \delta b \dx + \int_\Omega (b_2 \cdot \nabla) \delta u \cdot \delta b \dx \nonumber \\[1ex]
&\frac{1}{2} \frac{\rm d}{\dt} \int_\Omega |\delta r|^2 \dx + \int_\Omega (u_2 \cdot \nabla \delta r) \delta r \dx = - \int_\Omega (\delta u \cdot \nabla r_1 ) \delta r \dx\,.\label{U3}
\end{align}
An integraton by parts shows that the second integral in \eqref{U1}, the third in \eqref{U2} and the second in \eqref{U3} are all equal to zero.
Next, note that the last integrals in \eqref{U1} and \eqref{U2} are opposite. Therefore, by adding the three equations together, we gather
\begin{multline}\label{U4}
\frac{1}{2} \frac{\rm d }{\dt} \int_\Omega \bigg\{ |\delta u|^2 + |\delta b |^2 + |\delta r|^2  \bigg\} \dx + \int_\Omega \bigg\{ |\nabla \delta u|^2 +  |\nabla b|^2  \bigg\} \dx  \\
\leq \left|  \int_\Omega (\delta u \cdot \nabla) u_1 \cdot \delta u \dx \right| +      \left| \int_\Omega \delta r u_1^\perp \cdot \delta u \dx  \right| +   
\left|  \int_\Omega (\delta u \cdot \nabla r_1 ) \delta r \dx \right|\\
+      \left| \int_\Omega (\delta u \cdot \nabla) b_1 \cdot \delta b \dx  \right|  +  \left| \int_\Omega (\delta b \cdot \nabla) u_1 \cdot \delta b \dx  \right|  + 
\left| \int_\Omega (\delta b \cdot \nabla) b_1 \cdot \delta u \dx  \right|\,.
\end{multline}

The first three integrals, which do not involve the magnetic field, can be dealt with as in \cite{FG} (see Paragraph 4.4.2 therein). We briefly summarise the computations.
Firstly, using in turn the H\"older, the GN and Young's inequalities with exponents $1/4 + 3/4 = 1$, we infer
\begin{align*}
\left| \int_\Omega (\delta u \cdot \nabla) u_1 \cdot \delta u \dx    \right| & \leq \| \delta u \|_{L^4} \| u_1 \|_{L^4} \| \nabla \delta u \|_{L^2}\; \leq\;
C \| \delta u \|_{L^2}^{1/2} \| \nabla \delta u \|_{L^2}^{3/2} \| u_1 \|_{L^4}\\
& \leq \eta \| \nabla \delta u \|_{L^2}^2 + C(\eta) \| u_1 \|_{L^4}^4 \| \delta u \|_{L^2}^2\;=\;\eta \| \nabla \delta u \|_{L^2}^2 + C(\eta) M_1 \| \delta u \|_{L^2}^2\,.
\end{align*}
Notice that the GN inequality again gives that $M_1(t) = \| u_1 (t) \|_{L^4}^4 \leq C \| u_1(t) \|_{L^2}^{2} \| \nabla u_1(t) \|_{L^2}^{2}$, hence $M_1\in L^1 (\mathbb{R}_+)$
is an integrable function.

Next, making use of Proposition \ref{PropOP}, we have
\begin{equation*}
\left| \int_\Omega \delta r u_1^\perp \cdot \delta u \dx  \right| \leq \| u_1 \|_{L^\infty} \left( \| \delta u \|_{L^2}^2 + \| \delta r \|_{L^2}^2\right)\,=\,
N_4\,\left( \| \delta u \|_{L^2}^2 + \| \delta r \|_{L^2}^2\right)\,,
\end{equation*}
with $N_4(t) = \| u_1(t) \|_{L^\infty} \leq C \| u_1(t)\|_{L^2}^{1/2} \| \Delta u_1(t) \|_{L^2}^{1/2} \in L^4_T$ for any fixed $T>0$.

As for the third integral, we use the fact that $\nabla r_1 \in L^\infty_T(H^\gamma)$ for some $\gamma > 0$. By fractional Sobolev embedding
(see equation \eqref{Emb} in the appendix), we know that $\nabla r_1 \in L^\infty_T(L^p)$ for some $p > 2$. Let $q$ be an exponent such that $1/p + 1/q = \frac{1}{2}$.
Applying the GN inequality first, and then Young's inequality with exponents $1/(1 - 2/p) + 1/(2/p) = 1$ gives
\begin{align*}
\left|  \int_\Omega \delta u  \cdot \nabla r_1 \, \delta r \dx \right| & \leq \| \nabla r_1 \cdot \delta u \|_{L^2}^2 + \| \delta r \|_{L^2}^2 \leq \| \nabla r_1 \|_{L^p}^2 \| \delta u \|_{L^q}^2 +  \| \delta r \|_{L^2}^2 \\
& \leq C \big(p,\| r_{0, 1} \|_{H^{1+\beta}}, \|u_{0, 1}\|_{H^1}, \|b_{0, 1}\|_{H^1} \big) \| \delta u \|_{L^2}^{4/p} \| \nabla \delta u \|_{L^2}^{2 \left( 1 - 2/p \right)} + \| \delta r \|_{L^2}^2 \\
& \leq \eta \| \nabla \delta u \|_{L^2}^2 + C \big(\eta, \beta, \| r_{0, 1} \|_{H^{1+\beta}}, \|u_{0, 1}\|_{H^1}, \|b_{0, 1}\|_{H^1} \big) \| \delta u \|_{L^2}^2 + \| \delta r \|_{L^2}^2.
\end{align*}

We still have to handle three integrals, which involve the magnetic field. Firstly, integration by parts gives 
\begin{align*}
\left|  \int_\Omega (\delta u \cdot \nabla) b_1 \cdot \delta b \dx \right| & = \left|  \int_\Omega (\delta u \cdot \nabla)  \delta b \cdot b_1 \dx \right| \leq \| \delta u \|_{L^4} \| \nabla \delta b \|_{L^2} \| b_1 \|_{L^4}\\
& \leq \eta \| \nabla \delta b \|_{L^2}^2 + C(\eta) \| \delta u \|_{L^4}^2 \| b_1 \|_{L^4}^2\,\leq\,\eta \| \nabla \delta b \|_{L^2}^2 + \| \delta u \|_{L^2} \| \nabla \delta u \|_{L^2} \| b_1 \|_{L^4}^2 \\
&\leq \eta \big(  \| \nabla u \|_{L^2}^2 + \| \nabla b \|_{L^2}^2  \big) + M_1(t) \| \delta b \|_{L^2}^2\,,
\end{align*}
where we have also used the GN inequality in the second line, and Young's inequality in order to get the last inequality, and where we have set
$M_1(t) = C(\eta) \| b_1(t) \|_{L^4}^4 \in L^1(\R_+)$, because of the same bounds exhibited above for $\|u_1\|_{L^4}^4$.
The last integral in \eqref{U4} can be treated in the same way: after integration by parts, we get
\begin{equation*}
\left| \int_\Omega (\delta b \cdot \nabla) b_1 \cdot \delta u \dx  \right| = \left| \int_\Omega (\delta b \cdot \nabla)  \delta u \cdot b_1 \dx  \right| \leq
\| \delta u \|_{L^4} \| \nabla \delta b \|_{L^2} \| b_1 \|_{L^4} \leq M_1(t) \| \delta b \|_{L^2}^2\,.
\end{equation*}
Finally, for the remaining term we can apply one last time the GN inequality: we infer
\begin{align*}
\left| \int_\Omega (\delta b \cdot \nabla) u_1 \cdot \delta b \dx  \right| & \leq \| \nabla u_1 \|_{L^2} \| \delta b \|_{L^4}^2 \leq C \| \nabla u_1 \|_{L^2} \| \delta b \|_{L^2} \| \nabla \delta b \|_{L^2}\\
& \leq \eta \| \nabla \delta b \|_{L^2}^2 + C(\eta) \| \nabla u_1 \|_{L^2}^2 \| \delta b \|_{L^2}^2 \leq \eta \| \nabla \delta b \|_{L^2}^2 + M_1(t) \| \delta b \|_{L^2}^2,
\end{align*}
where $M_1(t) = C(\eta) \| \nabla u_1(t) \|_{L^2}^2 \in L^1(\R_+)$.

Define now
\begin{equation*}
E(t)\, :=\, \| \delta u(t) \|_{L^2}^2 + \| \delta b(t) \|_{L^2}^2 + \| \delta r(t) \|_{L^2}^2\qquad\mbox{ and }\qquad
E_0\, :=\, \| \delta u_0 \|_{L^2}^2 + \| \delta b_0 \|_{L^2}^2 + \| \delta r_0 \|_{L^2}^2\,.
\end{equation*}
Putting all the previous bounds together and choosing $\eta$ so small that the gradient terms can be absorbed in the left-hand side, from \eqref{U4} we arrive at the differential inequality
\begin{equation}\label{Gron}
\frac{1}{2} \frac{\rm d}{\dt} E(t)+ 
\frac{1}{2} \int_\Omega \bigg\{ |\nabla \delta u|^2 + |\nabla \delta b|^2 \bigg\} \dx\,\leq\,N_1(t) E(t)\,,
\end{equation}
where $N_1(t) \in L^1_T$ is the sum of all the functions $M_1\in L^1(\R_+)$, $N_4(t) \in L^4_T$ in the previous inequalities.
Therefore, Gr\"onwall's lemma implies that, for all $t\geq0$, one has
\begin{equation*}
E(t)\,\leq\,C(t)\,E_0\,.
\end{equation*}
Coming back to \eqref{Gron}, we finally get \eqref{EQUni}. The proposition is then proved.
\end{proof}

\appendix

\section{Appendix -- Fourier and harmonic analysis toolbox} \label{app:LP}

We recall here the main ideas of Littlewood-Paley theory, which we exploited in the previous analysis.
We refer e.g. to Chapter 2 of \cite{BCD} for details.
For simplicity of exposition, let us deal with the $\R^d$ case; however, the whole construction can be adapted also to the $d$-dimensional torus $\T^d$.

\medbreak
First of all, let us introduce the so called ``Littlewood-Paley decomposition'', based on a non-homogeneous dyadic partition of unity with
respect to the Fourier variable. 
We fix a smooth radial function $\chi$ supported in the ball $B(0,2)$, equal to $1$ in a neighborhood of $B(0,1)$
and such that $r\mapsto\chi(r\,e)$ is nonincreasing over $\R_+$ for all unitary vectors $e\in\R^d$. Set
$\varphi\left(\xi\right)=\chi\left(\xi\right)-\chi\left(2\xi\right)$ and
$\vphi_j(\xi):=\vphi(2^{-j}\xi)$ for all $j\geq0$.

The dyadic blocks $(\Delta_j)_{j\in\Z}$ are defined by\footnote{Throughout we agree  that  $f(D)$ stands for 
the pseudo-differential operator $u\mapsto\mc{F}^{-1}[f(\xi)\,\what u(\xi)]$.} 
$$
\Delta_j\,:=\,0\quad\mbox{ if }\; j\leq-2,\qquad\Delta_{-1}\,:=\,\chi(D)\qquad\mbox{ and }\qquad
\Delta_j\,:=\,\varphi(2^{-j}D)\quad \mbox{ if }\;  j\geq0\,.
$$
We  also introduce the following low frequency cut-off operator:
\begin{equation} \label{eq:S_j}
S_ju\,:=\,\chi(2^{-j}D)\,=\,\sum_{k\leq j-1}\Delta_{k}\qquad\mbox{ for }\qquad j\geq0\,.
\end{equation}
Note that the operator $S_j$ is a convolution operator with a function $K_j(x) = 2^{dj}K_1(2^j x) = \mathcal{F}^{-1}[\chi (2^{-j} \xi)] (x)$ of constant $L^1$ norm,
and hence defines a continuous operator for the $L^p \longrightarrow L^p$ topologies, for any $1 \leq p \leq +\infty$.

The following classical property holds true: for any $u\in\mc{S}'$, then one has the equality~$u=\sum_{j}\Delta_ju$ in the sense of $\mc{S}'$.
Let us also mention the so-called \emph{Bernstein inequalities}, which explain the way derivatives act on spectrally localized functions.
  \begin{lemma} \label{l:bern}
Let  $0<r<R$.   A constant $C$ exists so that, for any nonnegative integer $k$, any couple $(p,q)$ 
in $[1,+\infty]^2$, with  $p\leq q$,  and any function $u\in L^p$,  we  have, for all $\lambda>0$,
$$
\displaylines{
{\Supp}\, \widehat u \subset   B(0,\lambda R)\quad
\Longrightarrow\quad
\|\nabla^k u\|_{L^q}\, \leq\,
 C^{k+1}\,\lambda^{k+d\left(\frac{1}{p}-\frac{1}{q}\right)}\,\|u\|_{L^p}\;;\cr
{\Supp}\, \widehat u \subset \{\xi\in\R^d\,|\, r\lambda\leq|\xi|\leq R\lambda\}
\quad\Longrightarrow\quad C^{-k-1}\,\lambda^k\|u\|_{L^p}\,
\leq\,
\|\nabla^k u\|_{L^p}\,
\leq\,
C^{k+1} \, \lambda^k\|u\|_{L^p}\,.
}$$
\end{lemma}   

By use of Littlewood-Paley decomposition, we can define the class of Besov spaces.
\begin{defi} \label{d:B}
  Let $s\in\R$ and $1\leq p,r\leq+\infty$. The \emph{non-homogeneous Besov space}
$B^{s}_{p,r}$ is defined as the subset of tempered distributions $u$ for which
$$
\|u\|_{B^{s}_{p,r}}\,:=\,
\left\|\left(2^{js}\,\|\Delta_ju\|_{L^p}\right)_{j\geq -1}\right\|_{\ell^r}\,<\,+\infty\,.
$$
\end{defi}
Besov spaces are interpolation spaces between Sobolev spaces. In fact, for any $k\in\N$ and~$p\in[1,+\infty]$
we have the following chain of continuous embeddings: $ B^k_{p,1}\hookrightarrow W^{k,p}\hookrightarrow B^k_{p,\infty}$,
where  $W^{k,p}$ denotes the classical Sobolev space of $L^p$ functions with all the derivatives up to the order $k$ in $L^p$.
When $1<p<+\infty$, we can refine the previous result (this is the non-homogeneous version of Theorems 2.40 and 2.41 in \cite{BCD}): we have
$$
 B^k_{p, \min (p, 2)}\hookrightarrow W^{k,p}\hookrightarrow B^k_{p, \max(p, 2)}\,.
$$
In particular, for all $s\in\R$ we deduce the equivalence $B^s_{2,2}\equiv H^s$, with equivalence of norms:
\begin{equation} \label{eq:LP-Sob}
\|f\|_{H^s}\,\sim\,\left(\sum_{j\geq-1}2^{2 j s}\,\|\Delta_jf\|^2_{L^2}\right)^{1/2}\,.
\end{equation}

As an immediate consequence of the first Bernstein inequality, one gets the following embedding result.
\begin{prop}\label{p:embed}
The space $B^{s_1}_{p_1,r_1}$ is continuously embedded in the space $B^{s_2}_{p_2,r_2}$ for all indices satisfying $p_1\,\leq\,p_2$ and
$$
s_2\,<\,s_1-d\left(\frac{1}{p_1}-\frac{1}{p_2}\right)\qquad\mbox{ or }\qquad
s_2\,=\,s_1-d\left(\frac{1}{p_1}-\frac{1}{p_2}\right)\;\;\mbox{ and }\;\;r_1\,\leq\,r_2\,. 
$$
\end{prop}
In particular, in dimension $d = 2$, we get the regular Sobolev embeddings
\begin{equation}\label{Emb}
H^s \simeq B^s_{2, 2} \hookrightarrow B_{p, 2}^{s - 2\left( \frac{1}{2} - \frac{1}{p} \right)} \hookrightarrow B_{p, 2}^0 = B_{p, \min(p, 2)}^0 \hookrightarrow L^p
\end{equation}
as long as $0 \leq s < 1$ and $2 \leq p \leq 2/(1-s)$. Also note that, still for $d = 2$, one has the embedding $H^s \hookrightarrow L^\infty \cap C^0$ for all $s > 1$.


Let us now introduce the paraproduct operator (after J.-M. Bony, see \cite{Bony}). Constructing the paraproduct operator relies on the observation that, 
formally, any product  of two tempered distributions $u$ and $v,$ may be decomposed into 
\begin{equation}\label{eq:bony}
u\,v\;=\;T_u(v)\,+\,T_v(u)\,+\,R(u,v)\,,
\end{equation}
where we have defined
$$
T_u(v)\,:=\,\sum_jS_{j-1}u\Delta_j v,\qquad\qquad\mbox{ and }\qquad\qquad
R(u,v)\,:=\,\sum_j\sum_{|j'-j|\leq1}\Delta_j u\,\Delta_{j'}v\,.
$$
The above operator $T$ is called ``paraproduct'' whereas
$R$ is called ``remainder''.
The paraproduct and remainder operators have many nice continuity properties. 
The following ones have been of constant use in this paper (see the proof in e.g. Chapter 2 of \cite{BCD}).
\begin{prop}\label{p:op}
For any $(s,p,r)\in\R\times[1,\infty]^2$ and $t>0$, the paraproduct operator 
$T$ maps continuously $L^\infty\times B^s_{p,r}$ in $B^s_{p,r}$ and  $B^{-t}_{\infty,\infty}\times B^s_{p,r}$ in $B^{s-t}_{p,r}$.
Moreover, the following estimates hold:
$$
\|T_u(v)\|_{B^s_{p,r}}\,\leq\, C\,\|u\|_{L^\infty}\,\|\nabla v\|_{B^{s-1}_{p,r}}\qquad\mbox{ and }\qquad
\|T_u(v)\|_{B^{s-t}_{p,r}}\,\leq\, C\|u\|_{B^{-t}_{\infty,\infty}}\,\|\nabla v\|_{B^{s-1}_{p,r}}\,.
$$
For any $(s_1,p_1,r_1)$ and $(s_2,p_2,r_2)$ in $\R\times[1,\infty]^2$ such that 
$s_1+s_2>0$, $1/p:=1/p_1+1/p_2\leq1$ and~$1/r:=1/r_1+1/r_2\leq1$,
the remainder operator $R$ maps continuously~$B^{s_1}_{p_1,r_1}\times B^{s_2}_{p_2,r_2}$ into~$B^{s_1+s_2}_{p,r}$.
In the case $s_1+s_2=0$, provided $r=1$, operator $R$ is continuous from $B^{s_1}_{p_1,r_1}\times B^{s_2}_{p_2,r_2}$ with values
in $B^{0}_{p,\infty}$.
\end{prop}

As a corollary of the previous proposition, we deduce the following continuity properties of the product in Sobolev spaces, which have been used in the course of the analysis.
In the statement, we limit ourselves to the case of space dimension $d=2$,   the only relevant one for this study.
\begin{lemma}\label{Para}
Take the space dimension to be $d = 2$. For appropriate $f$ and $g$, one has the following properties:
\begin{enumerate}[(1)]
\item for $s \in \mathbb{R}$ and $t > 0$, $\| T_f(g) \|_{H^{s - t}} \leq C \| f \|_{H^{1 - t}} \| g \|_{H^s}$.
\item for $s \in \mathbb{R}$, $\| T_f(g) \|_{H^{s}} \leq C \| f \|_{L^\infty} \| g \|_{H^s}$;
\item for $s_1, s_2  \in \mathbb{R}$ such that $s_1 + s_2 > 0$, $\| R(f, g) \|_{H^{s_1 + s_2 - 1}} \leq C \| f \|_{H^{s_1}} \| g \|_{H^{s_2}}$.
\end{enumerate}
\end{lemma}

\begin{proof}
We start by proving the first point. From the second inequality in Proposition \ref{p:op}, we get
\begin{equation*}
\| T_f(g) \|_{H^{s - t}} = \| T_f(g) \|_{B^{s - t}_{2, 2}} \leq C \| f \|_{B^{-t}_{\infty, \infty}} \| \nabla g \|_{B^{s-1}_{2, 2}} = C \| f \|_{B^{-t}_{\infty, \infty}} \| \nabla g \|_{H^{s-1}}\,.
\end{equation*}
Because $d = 2$, Proposition \ref{p:embed} gives the embedding $H^{1 - t} = B^{1 - t}_{2, 2} \hookrightarrow B^{-t}_{\infty, \infty}$, from which we infer
the first inequality.

Next, using the first inequality in Proposition \ref{p:op}, we gather
\begin{equation*}
\| T_f(g) \|_{H^{s}}  = \| T_f(g) \|_{B^{s - t}_{2, 2}}  \leq C \| f \|_{L^\infty} \| \nabla g \|_{B^{s-1}2, 2} \leq C \| f \|_{L^\infty} \| g \|_{H^s}\,,
\end{equation*}
which proves the second point.

Finally, using Proposition \ref{p:op} to estimate the remainder term, because we have assumed that $s_1 + s_2 > 0$, we get
\begin{equation*}
\| R(f, g) \|_{B^{s_1 + s_2}_{1, 1}} \leq C \| f \|_{H^{s_1}} \| g \|_{H^{s_2}}\,.
\end{equation*}
Proposition \ref{p:embed} provides the embedding $B^{s_1 + s_2}_{1, 1} \hookrightarrow B^{s_1 + s_2 - 1}_{2, 2} = H^{s_1 + s_2 - 1}$, which gives the last inequality, thus completing the proof of the lemma.
\end{proof}

\begin{cor}\label{c:product}
As a consequence of the previous lemma, wee see that
\begin{enumerate}[(i)]
\item for any $\delta > 0$, the space $H^{1 + \delta}$ is a Banach algebra;
\item for all $s > -1$ and all $(f, g) \in H^1 \times H^s$, we have $fg \in H^{s - \delta}$ for any $\de>0$, with
\begin{equation*}
\| fg \|_{H^{s-\delta}}\, \leq\, C\, \| f \|_{H^1}\,  \| g \|_{H^s}\,.
\end{equation*}
\end{enumerate}
\end{cor}

The next two propositions are functional inequalities which we use repeatedly in this article. The first one is the classical Gagliardo-Nirenberg inequality, whose proof can be found e.g. in
Corollary 1.2 of \cite{CDGG}.

\begin{lemma}\label{GN}
Let $2 \leq p < +\infty$ such that $1/p > 1/2 - 1/d$. Then, for all $u \in H^1$, one has
\begin{equation*}
\| u \|_{L^p}\, \leq\, C(p)\, \| u \|_{L^2}^{\lambda}\, \| \nabla u \|_{L^2}^{1 - \lambda}\,, \qquad\qquad \text{ with } \qquad \lambda = \frac{d(p-2)}{2p}\,.
\end{equation*}
In particular, in dimension $d = 2$, we have $\| u \|_{L^p}\, \leq\, \| u \|_{L^2}^{2/p}\, \| \nabla u \|_{L^2}^{1 - 2/p}$ for any $p\in[2,+\infty[\,$.
\end{lemma}

The following proposition, which is in the same spirit of the Gagliardo-Nirenberg inequality, gives a bound for the endpoint case $p=+\infty$. It is proved by resorting to
Littlewood-Paley decomposition and the Bernstein inequalities.

\begin{prop}\label{PropOP}
Let $f \in H^{s}$, for some $s>d/2$. Then there exists a constant $C = C(d,s) > 0$ and an exponent $\alpha = \alpha(d,s) = d/(2s)$ such that
\begin{equation*}
\| f \|_{L^\infty}\, \leq\, C\, \| f \|_{L^2}^{1 - \alpha}\, \left\| (-\Delta)^{s/2} f \right\|_{L^2}^\alpha\,.
\end{equation*}
In particular, when $d=2$ and $s=2$, then $\alpha = 1/2$ and $\| f \|_{L^\infty}\,\leq\, C\, \| f \|_{L^2}^{1/2}\, \| \Delta f \|_{L^2}^{1/2}$.
\end{prop}

\begin{proof}
The main idea of the proof is to look separately at the high and low frequencies. Let $N \geq 1$ be an integer to be fixed later on. Thanks to the Littlewood-Paley decomposition, we can write
\begin{equation*}
\| f \|_{L^\infty} \leq \sum_{j < N} \| \Delta_j f \|_{L^\infty} + \sum_{j\geq N} \| \Delta_j f \|_{L^\infty}\,.
\end{equation*}
Using the first Bernstein inequality in the fist sum gives $\| \Delta_j f \|_{L^\infty} \leq C 2^{jd/2} \| \Delta_j f \|_{L^2}$.
On the other hand, the two Bernstein inequalities applied to the high frequency term yield $\| \Delta_jf \|_{L^2} \leq C 2^{jd/2} 2^{-sj} \| \Delta_j (-\Delta)^{s/2} f \|_{L^2}$. Therefore
\begin{align*}
\| f \|_{L^\infty}\,&\leq\, C\, \|f \|_{L^2}\,\sum_{j < N} 2^{jd/2}\,+\,C\,\left\|(-\Delta)^{s/2} f \right\|_{L^2}\sum_{j \geq N} 2^{-j(s-d/2)} \\
&\leq\,C\,\left(\|f \|_{L^2}\,2^{Nd/2}\, +\,\left\|(-\Delta)^{s/2} f \right\|_{L^2} 2^{-N(s-d/2)}\right)\,.
\end{align*}
By choosing $N$ so that $2^{Ns} \approx\left\| (-\Delta)^{s/2} f \right\|_{L^2}\,\| f \|^{-1}_{L^2}$ (say that $N$ is the largest integer such that $2^{Ns}$ is smaller than
$\left\| (-\Delta)^{s/2} f \right\|_{L^2}\,\| f \|^{-1}_{L^2}$), we deduce the desired inequality.
\end{proof}

Finally, we recall a classical commutator estimate, which we have needed in our analysis (see Lemma 2.97 of \cite{BCD} for the proof).
\begin{lemma}\label{lemma:Comm}
Let $\chi \in C^1(\mathbb{R}^d)$ be such that $H(\xi) := (1 + |\xi|) \what{\chi}(\xi) \in L^1$. There exists a constant $C > 0$ depending only on $\| H \|_{L^1}$ such that
\begin{equation*}
\forall\, f \in W^{1, \infty}\,,\; \forall\, g \in L^2\,,\; \forall\, \lambda > 0\,, \qquad \left\| \left[   \chi \left(  \frac{1}{\lambda} D \right), f   \right] g \right\|_{L^2}\, \leq \,
C\, \frac{1}{\lambda}\, \| \nabla f \|_{L^\infty} \,\| g \|_{L^2}\,.
\end{equation*}
\end{lemma}

\end{document}